\documentclass[reqno,12pt]{amsart} 
\usepackage{mathrsfs,amsmath,amsthm,amssymb,amsfonts,bbm}
\usepackage{verbatim,pifont,empheq}
\usepackage{esint}
\usepackage{ulem}
\usepackage{enumitem}

\usepackage{color}
\usepackage[hyperfootnotes=false,pagebackref]{hyperref}
\hypersetup{
  colorlinks,
  citecolor=blue,
  linkcolor=red,
  urlcolor=blue}

\usepackage{tikz}
\usetikzlibrary{positioning, arrows.meta}
\usetikzlibrary{calc, decorations.pathmorphing}

\usepackage[left=1.2in, right=1.2in, top=1.2in, bottom = 1.2in]{geometry}
\allowdisplaybreaks[3]

\newtheorem{theorem}{Theorem}[section]
\newtheorem{lemma}[theorem]{Lemma}
\newtheorem{proposition}[theorem]{Proposition}
\newtheorem{corollary}[theorem]{Corollary}
\theoremstyle{definition}

\theoremstyle{remark}
\newtheorem{remark}[theorem]{Remark}

\numberwithin{equation}{section}

\newcommand{\cL}{\mathcal{L}}
\newcommand{\M}{\mathcal{M}}

\newcommand{\Q}{\mathcal{Q}}

\newcommand{\Z}{\mathbb{Z}}

\newcommand{\R}{\mathbb{R}}
\newcommand{\T}{\mathbb{T}}
\newcommand{\N}{\mathbb{N}}
\newcommand{\e}{\varepsilon}

\DeclareMathOperator{\diam}{diam}

\DeclareMathOperator{\dist}{dist}

\DeclareMathOperator{\proj}{proj}

\begin{document}
\title[Large-scale harmonic measures and nontangential maximal functions]{Large-scale harmonic measures and nontangential maximal functions in periodic homogenization}

\author{Zhongwei Shen}
\address{Zhongwei Shen: Institute for Theoretical Sciences, Westlake University, No. 600 Dunyu Road, Xihu District, Hangzhou,
Zhejiang 310030, P.R. China.}
\email{shenzhongwei@westlake.edu.cn}

\author{Jinping Zhuge}
\address{Jinping Zhuge: Morningside Center of Mathematics, Academy of Mathematics and systems science,
Chinese Academy of Sciences, Beijing 100190, China.}
\email{jpzhuge@amss.ac.cn}



\date{}

\begin{abstract}
    In this paper, we consider the elliptic operators $\mathcal{L}_\varepsilon = -\nabla\cdot (A(X/\e) \nabla )$ with periodic coefficients in a bounded domain $\Omega$ without any local smoothness assumption on $A = A(Y)$, where $\e \ll \diam(\Omega)$ is a microscopic scale. Due to the irregularity of the coefficients at $\varepsilon$ scale, we introduce the correct forms of the large-scale nontangential maximal functions for the Dirichlet, Neumann and regularity problems that measure the behaviors of solutions at an $\varepsilon$ distance away from the boundary. The $L^p$ estimates uniform in $\e$ are established for these nontangential maximal functions for the same and optimal ranges of $p$ as the Laplace operator in the Lipschitz or $C^1$ domains. With some additional regularity assumption on the coefficients, the large-scale estimates combined with the small-scale estimates recover the classical full-scale estimates of the nontangential maximal functions. Our proofs are based on the notion of large-scale $\mathcal{L}_\varepsilon$-harmonic measures, the periodic structure of operators in the transversal direction to the boundaries, and the homogenization tools, including convergence rates and large-scale regularity.
\end{abstract}

\maketitle


\section{Introduction}

\subsection{Motivations}
Consider the elliptic operator $\cL_\e = -\nabla\cdot (A^\e\nabla)$,
which arises in homogenization theory, 
where $A^\e=A(X/\e)$ and  $\e \in (0, 1)$ is a small parameter.
Assume that the coefficient matrix $A$ is real, symmetric, and satisfies the following conditions:
\begin{itemize}
    \item Ellipticity and boundedness: there exists $\Lambda \ge 1$ such that for a.e. $\xi \in \R^d, Y\in \R^d$,
    \begin{equation}\label{ellipticity}
       \xi \cdot A(Y)\xi \ge \Lambda^{-1} |\xi|^2, \quad |A(Y) \xi| \le \Lambda |\xi|.
    \end{equation}

    \item Periodicity: for any $Z\in \Z^d$, 
    \begin{equation}\label{periodicity}
        A(\cdot + Z) = A(\cdot).
    \end{equation}
\end{itemize}
We are interested in the $L^p$ Dirichlet and Neumann problems in a bounded domain $\Omega$,

\begin{equation}\label{Dp}
    (D)_p \quad \left\{
    \begin{aligned}
        & \cL_\e(u_\e) = 0 \quad \text{in } \Omega, \\
        & u_\e = f \in L^p(\partial \Omega) \quad \text{on } \partial \Omega,
    \end{aligned}
    \right.
\end{equation}
and 
\begin{equation}\label{Np}
    (N)_p \quad \left\{
    \begin{aligned}
        & \cL_\e(u_\e) = 0 \quad \text{in } \Omega, \\
        & \frac{\partial u_\e}{\partial \nu_\e} = g \in L^p(\partial \Omega) \quad \text{on } \partial \Omega,
    \end{aligned}
    \right.
\end{equation}
where $\partial u_\e/\partial \nu_\e = n\cdot A^\e \nabla u_\e$ and $n$ represents the unit outer normal vector.
We will also consider the $L^p$ regularity problem,
\begin{equation}\label{Rp}
    (R)_p \quad \left\{
    \begin{aligned}
        & \cL_\e(u_\e) = 0 \quad \text{in } \Omega, \\
        & u_\e = f \in W^{1, p}(\partial \Omega) \quad \text{on } \partial \Omega.
    \end{aligned}
    \right.
\end{equation}
In the case $\e=1$ without the periodicity condition, these problems have been studied 
extensively since the late 1970's.
Let $N(u)$ denote the nontangential maximal function of $u$.
If $\Omega$ is a bounded Lipschitz domain,
it is well known that under certain smoothness conditions on the matrix $A$, 
the Dirichlet problem $(D)_p$ with $\e=1$ and the natural regularity condition $N(u_1)\in L^p(\partial\Omega)$ 
is solvable for $p_0< p\le \infty$, while 
$(N)_p$ and $(R)_p$ with $\e=1$ and $N(\nabla u_1)\in L^p(\partial\Omega)$ are solvable for $1< p< p_1$, where $p_0<2 $ and $p_1>2$ depend on $A$ and $\Omega$ \cite{K94}.
Under the periodicity condition \eqref{periodicity},
these results were extended to the case $\e\in (0, 1)$
with nontangential-maximal-function estimates that are uniform in $\e$
by C. Kenig and the first author.
Indeed, it was proved in \cite{KS11, KS11-2} that 
solutions of \eqref{Dp} and \eqref{Rp} satisfy the estimates,
\begin{equation}\label{uni-1}
\|N(u_\e) \|_{L^p(\partial\Omega)}
\le C \| f \|_{L^p(\partial\Omega)}
\quad \text{ and } \quad 
\|\widetilde{N}(\nabla u_\e) \|_{L^p(\partial\Omega)}
\le C \| f \|_{W^{1, p}(\partial\Omega)},
\end{equation}
respectively, and solutions of \eqref{Np} satisfy 
\begin{equation}\label{unif-2}
\| \widetilde{N} (\nabla u_\e)\|_{L^p(\partial\Omega)}
\le C \| g \|_{L^p(\partial\Omega)},
\end{equation}
where $C$ is independent of $\e$.
See \eqref{NM}-\eqref{MNM} for the definitions of 
nontangential maximal functions $N(u)$ and $\widetilde{N}(u)$.
We point out that uniform estimates \eqref{uni-1}-\eqref{unif-2}
cannot be expected to hold without some structure conditions on $A$,
as the same estimates for the case $\e=1$ fail 
without some smoothness conditions on $A$, even in a smooth domain
\cite{MM-1980, CFK81, K94}.

The primary purpose of this paper is to address the following question: 
Under the periodicity condition on $A$, 
are there large-scale nontangential-maximal-function estimates
for the operator $\cL_\e$,
without additional smoothness condition on $A$?
The question is partially motivated by the observation that if $u_\e$ is a weak solution of $\cL_\e(u_\e)=0$ in $B_1(0)=B(0, 1)$, where $A$ satisfies 
\eqref{ellipticity} and \eqref{periodicity}, then 
\begin{equation}\label{lip}
    \fint_{B_r (0)} |\nabla u_\e|^2
    \le C \fint_{B_1(0)} |\nabla u_\e|^2,
\end{equation}
for $\e\le r\le 1$, 
where $C$ depends only on $d$ and $\Lambda$.
This estimate, which follows from \cite{AL87}, is now referred to as the large-scale Lipschitz estimate.
By combining \eqref{lip} with the small-scale estimate,
$$
|\nabla u_\e (0) |^2
\le C \fint_{B_\e(0)} |\nabla u_\e|^2,
$$
which holds with some smoothness condition  on $A$ by rescaling, one obtains
the uniform  Lipschitz estimates for $\cL_\e$. The approach of deriving the large-scale estimates (due to periodic structure) and the small-scale estimate (due to the local smoothness of the coefficients) separately turns out to be natural and general in the study of uniform regularity in homogenization. Therefore, it is equally natural to pursue the large-scale nontangential-maximal-function estimates in the same spirit.

In this paper, we find the correct forms of the large-scale nontangential maximal functions and establish their uniform estimates independent of $\e$ for operators with periodic and merely bounded measurable coefficients.

\subsection{Statement of main results}
Recall the definitions of the nontangential maximal function,
\begin{equation}\label{NM}
    N(F) (Q)=\sup \big\{ |F(X)|: X \in \Gamma (Q) \big\}
\end{equation}
and its modified version,
\begin{equation}\label{MNM}
    \widetilde{N}(F)(Q) = \sup \bigg\{ \bigg(\fint_{B(X, \delta(X)/2) }|F|^2  \bigg)^{1/2}: X\in \Gamma(Q)\bigg\},
\end{equation}
where $\delta(X) = \dist(X, \partial \Omega)$, the nontangential region $\Gamma(Q)$ for $Q \in  \partial \Omega$ is given by
\begin{equation}\label{def.NTcone}
    \Gamma(Q) = \left\{X\in \Omega: |X - Q| < \beta \delta(X)\right \},
\end{equation}
and $\beta > 1$ is a large fixed constant depending on $\Omega$.
For $Q \in \partial \Omega$ and $\e\in (0, 1)$, define the large-scale nontangential maximal function by
\begin{equation}\label{def.N_e}
    N_\e(F)(Q) = \sup \big\{ |F(X)|: X\in \Gamma(Q)
    \text{ and } \delta (X)\ge 10\e\big\},
\end{equation}
and its modified version by
\begin{equation}\label{def.tN_e}
    \widetilde{N}_\e(F)(Q) = \sup \bigg\{ \bigg(\fint_{B(X, \delta(X)/2) }|F|^2  \bigg)^{1/2}: X\in \Gamma(Q) \text{ and  }
    \delta (X)\ge 10 \e\bigg\}.
\end{equation}
For $Q \in \partial \Omega$, define
\begin{equation}
    S_{\e}(f)(Q) = \sup \big\{ |f(P)|: P\in \partial\Omega \text{ and } |P-Q|< \e\big \}.
\end{equation}

The following theorems are the main results of the paper.

\begin{theorem}\label{thm.Dp}
Assume that $A$ is real, symmetric and satisfies \eqref{ellipticity} and \eqref{periodicity}.
    Let $\Omega$ be a bounded Lipschitz domain and
    $f\in H^{1/2}(\partial\Omega)\cap C(\partial\Omega)$. Let
     $u_\e\in H^1(\Omega)$  be a weak solution of
     $\cL_\e (u_\e)=0$ in $\Omega$ with the
     Dirichlet condition $u_\e =f$
     on $\partial\Omega$. Then there exists $\delta\in (0, 1)$,
     depending on $d$, $\Lambda$ and $\Omega$,
     such that for $2-\delta < p \le \infty$,
    \begin{equation}\label{est.Dp}
        \| N_\e(u_\e) \|_{L^p(\partial \Omega)} \le C \| S_{\e}(f)\|_{L^p(\partial \Omega)},
    \end{equation}
    where $C$  depends on $d$,  $p$,  $\Lambda$ and $\Omega$.
    Moreover, if $\Omega$ is a bounded $C^1$ domain,
    the estimate \eqref{est.Dp} holds for $1< p\le \infty$.
\end{theorem}



\begin{theorem}\label{thm.Rp}
Let $A$ and $\Omega$ be the same as in Theorem \ref{thm.Dp}.
     Let $u_\e\in H^1(\Omega)$ be a weak solution of $\cL_\e (u_\e)=0$ in $\Omega$
     with the Dirichlet condition $u_\e=f\in H^1(\partial\Omega)$
     on $\partial\Omega$. Then there exists $\delta>0$, depending only on $d$,  $\Lambda$ and $\Omega$, such that for $1<p<2+\delta$,
    \begin{equation}\label{est.Rp}
        \| \widetilde{N}_\e( \nabla u_\e) \|_{L^p(\partial \Omega)} 
        \le C\| f \|_{W^{1,p}(\partial \Omega)},
    \end{equation}
    where $C$ depends on $d$, $p$,  $\Lambda$ and $\Omega$.
    Moreover, if $\Omega$ is a bounded $C^1$ domain,
    the estimate \eqref{est.Rp} holds for $1< p< \infty$.
\end{theorem}

\begin{theorem}\label{thm.Np}
Let $A$ and $\Omega$ be the same as in Theorem \ref{thm.Dp}.
     Let $u_\e\in H^1(\Omega)$ be a weak solution of 
    the Neumann problem: $\cL_\e (u_\e)=0$ in $\Omega$
    and $\frac{\partial u_\e}{\partial \nu_\e} =g$ on $\partial\Omega$,
    where $g\in L^2(\partial\Omega)$ and
    $\int_{\partial\Omega} gd\sigma =0$. Then there exists $\delta>0$, depending on $d$,  $\Lambda$ and $\Omega$, such that for $1<p<2+\delta$, 
    \begin{equation}\label{est.Np}
        \| \widetilde{N}_\e( \nabla u_\e) \|_{L^p(\partial \Omega)} \le C  \| g \|_{L^p(\partial \Omega)},
    \end{equation}
    where $C$ depends on $d$, $p$, $\Lambda$, and $\Omega$.
    Moreover, if $\Omega$ is a bounded $C^1$ domain,
    the estimate \eqref{est.Np} holds for $1< p< \infty$.
\end{theorem}





As a corollary, we obtain the full-scale estimates of the nontangential maximal functions under the additional regularity assumption on the coefficients.
\begin{corollary}\label{coro.fullscale}
    If in addition, $A$ is  H\"{o}lder continuous, then the large-scale nontangential maximal functions, $N_\e(u_\e)$ and $\widetilde{N}_\e(\nabla u_\e)$ in Theorems \ref{thm.Dp}-\ref{thm.Np}, can be upgraded to the full-scale nontangential maximal functions, $N(u_\e)$ and $\widetilde{N}(\nabla u_\e)$, respectively.
\end{corollary}

Corollary \ref{coro.fullscale} recovers the results of \cite{KS11} in Lipschitz domains. The results in $C^1$ domains for the full range of $p\in (1,\infty)$ are new. Moreover, we point out that the large-scale estimates in Theorems \ref{thm.Dp}-\ref{thm.Np} hold also in convex domains; see Remark \ref{rmk.convex} for more details.




\subsection{Proof sketch}
We now describe our approaches to Theorems \ref{thm.Dp}-\ref{thm.Np}.
Let 
$$
\Omega_r = \big\{X \in \Omega:\ \  \dist(X, \partial \Omega) < r \big\}
$$
denote a boundary layer with thickness $r$.
Our starting point is the two large-scale Rellich estimates
\eqref{eq.LSRellich-D} and \eqref{eq.LSRellich-N} established in \cite{KS11, Shen17}.
Let $u_\e$ be a  solution of the Dirichlet problem:
$\mathcal{L}_\e (u_\e)=0$ in $\Omega$ and $u_\e =f \in H^1(\partial\Omega)$
on $\partial \Omega$. Then
\begin{equation}\label{eq.LSRellich-D}
    \bigg( \frac{1}{r} \int_{\Omega_r} |\nabla u_\e|^2 \bigg)^{1/2} \le C  \| f \|_{H^1(\partial \Omega)},
\end{equation}
for all $\e \le r < \diam(\Omega)$, where
 the constant $C$ depends only on $d$, $\Lambda$ and $\Omega$.  
 If $u_\e$ is a  solution of the Neumann problem:
 $\cL_\e(u_\e)=0 $ in $\Omega$ and 
 $\frac{\partial u_\e}{\partial \nu_\e} =g \in L^2(\partial\Omega)$
 on $\partial\Omega$,
 then 
\begin{equation}\label{eq.LSRellich-N}
    \bigg( \frac{1}{r} \int_{\Omega_r} |\nabla u_\e|^2 \bigg)^{1/2} \le C \| g \|_{L^2(\partial \Omega)},
\end{equation}
for all $\e \le r< \diam(\Omega)$.
For $X\in \Omega$, let $\omega_\e^X $ denote the $\cL_\e$-harmonic 
measure in $\Omega$.
We introduce the large-scale $\cL_\e$-harmonic measure $\overline{\omega}^X_\e d \sigma$ (or $\overline{\omega}^X_\e$ for simplicity), where
$$
    \overline{\omega}^X_\e(Q) := \frac{\omega_\e^X(\Delta_\e(Q))}{\sigma(\Delta_\e(Q))}
$$
and $\Delta_\e (Q)=B_\e(Q)\cap \partial \Omega$.
Using a localized version of \eqref{eq.LSRellich-D},
we  show that the large-scale $\cL_\e$-harmonic 
measure is a $B_2$ weight on $\partial\Omega$ (i.e., $\overline{\omega}^X_\e$ satisfies the $L^2$ reverse H\"{o}lder inequality on $\partial\Omega$).
This leads to the estimate \eqref{est.Dp} for the Dirichlet problem in a Lipschitz domain for
$2-\delta < p\le \infty$.
For the Neumann and regularity problems in a bounded Lipschitz domain, we apply the technique of the difference operator to utilize the 
periodicity of the coefficient matrix, as in \cite{KS11}.
Together with the estimate \eqref{est.Dp}, this allows us to bound $\widetilde{N}_\e (u_\e)$ by the boundary data as well as an integral of $|\nabla u_\e|^2$
over the boundary layer $\Omega_{10\e}$.
The desired estimates for $p=2$ then follow from the large-scale Rellich 
estimates \eqref{eq.LSRellich-D} and \eqref{eq.LSRellich-N}.
The extensions to the range $2<p< 2+\delta$
rely on  the large-scale reverse H\"older estimates, 
while the case $1< p< 2$ uses the interpolation and some well-known estimates
for the Green and Neumann functions. 

 If $\Omega$ is a bounded $C^1$ domain,
the large-scale Rellich estimates hold in the $L^p$ setting 
for  $2< p< \infty$, 
as demonstrated in \cite{Shen17} under some smoothness condition on $A$.
Without the smoothness condition, the solutions must be averaged at $\e$-scale in the estimates.
As a result,  by systematic applications of large-scale analysis (involving different types  of average operators), the approach outlined above for Lipschitz domains yields the large-scale
nontangential-maximal-function estimates \eqref{est.Dp}-\eqref{est.Np}
for any $1< p< \infty$; see Figure \ref{fig.proofflow} for the flowchart of the proofs for both the $L^2$ estimate in Lipschitz domains and the $L^p$ estimates in $C^1$ domains.

\begin{figure}[htbp]
\begin{tikzpicture}[node distance=1.5cm, every node/.style={draw, rectangle, minimum width=2cm, minimum height=1cm, align = center}, arrow/.style={->, >=Stealth, thick}]

    \node (step1) {Localized Rellich estimates \\ Propositions \ref{prop.localR2}, \ref{prop.localN2}, Theorems \ref{thm.LocalRellich.Rp}, \ref{thm.LocalRellich.Np}};
    \node[below=1cm of step1] (step2) {Reverse H\"{o}lder inequality \\ of large-scale $\cL_\e$-harmonic measure\\ Theorem \ref{thm.B2}, Theorem \ref{thm.C1reverse}};
    \node[below=1cm of step2] (step3) {$(D)_p$ estimates \\ Theorem \ref{thm.Dp}};
    \node[below=2cm of step3] (step4) {Reduce nontangential maximal functions \\ to Rellich estimates \\ Lemmas \ref{lem.NT.ptws}, \ref{lem.NeQe}, \ref{lem.C1Rp}};
    \node[below=2cm of step4] (step5) {$(R)_p$ and $(N)_p$ estimates \\ Theorems \ref{thm.Rp}, \ref{thm.Np}};

    \node[left=2cm of step3] (Diff) at ($(step3)!0.5!(step4)$) [yshift = 3.5pt] {Difference operator \\ in graph domains \eqref{def.Qe}}; 
    \node[right=2cm of step3] (GloRell) at ($(step4)!0.5!(step5)$) [yshift=-3.5pt] {Global Rellich estimates\\ \eqref{eq.LSRellich-D}, \eqref{eq.LSRellich-N}, Theorems \ref{thm.C1Rellich.Rp}, \ref{thm.C1Rellich.Np}};
    \draw[arrow] (step1) -- (step2);
    \draw[arrow] (step2) -- (step3);
    \draw[arrow] (step3) -- coordinate[midway] (midpoint1) (step4);
    \draw[arrow] (step4) -- coordinate[midway] (midpoint2) (step5);
    \draw[arrow] (Diff.east) --  (midpoint1);
    \draw[arrow] (GloRell.west) --  (midpoint2);
    \draw[<->,>=Stealth, thick] (GloRell.north) --  ++(0,0.1) |- (step1.east);
\end{tikzpicture}
    \caption{The flowchart of proofs}
    \label{fig.proofflow}
\end{figure}
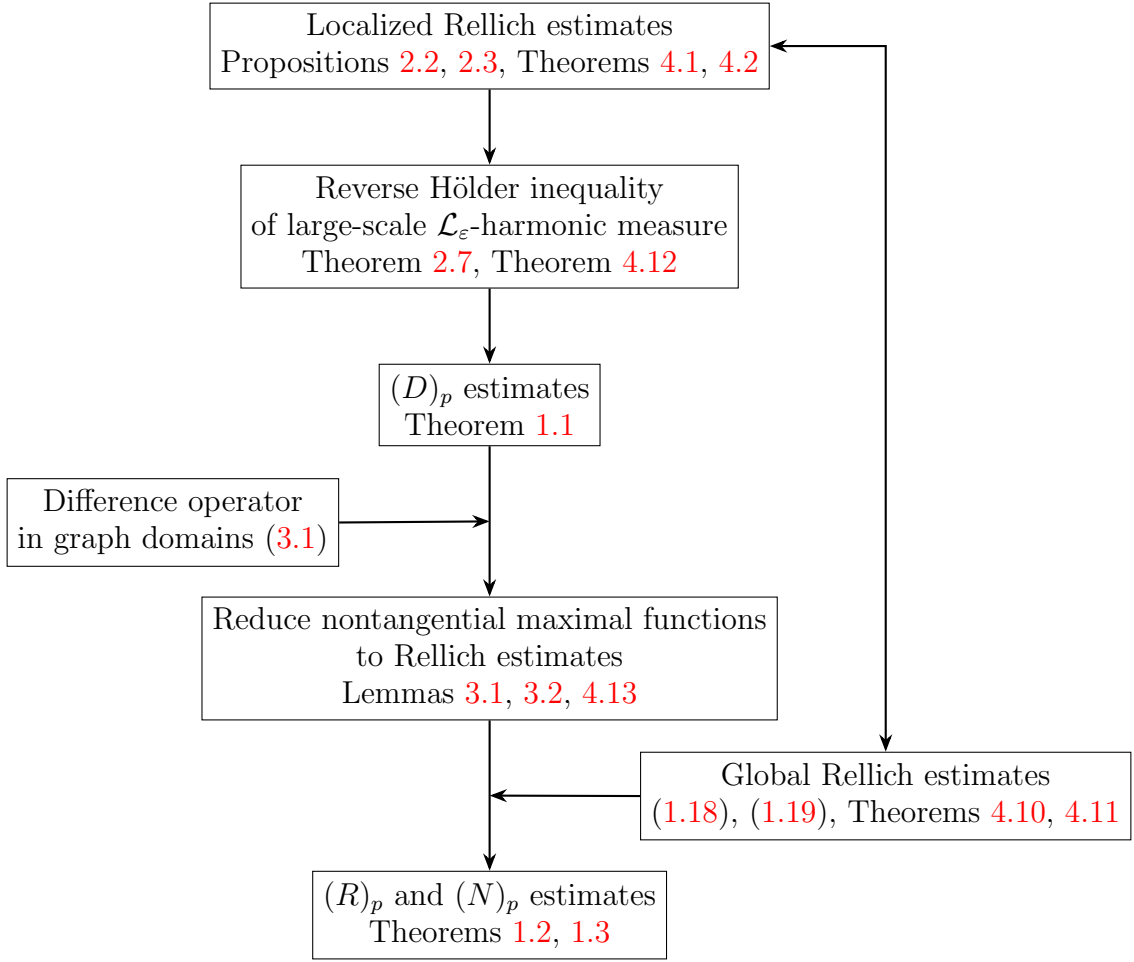



\subsection{Notations}

We list below some notations that are frequently used throughout this paper.

\begin{itemize}
    \item We use the capital letters $X, Y, Z, \cdots$ to denote the interior points contained in the domain $\Omega \subset \R^d$, and $P, Q, \cdots$ to denote the boundary points on $\partial \Omega$.

    \item $\delta(X) = \dist(X, \partial \Omega): = \inf\{ |X-Q|: Q\in \partial \Omega \}$.


    \item $\sigma$ denotes the surface measure of $\partial \Omega$. For a Borel subset $E\subset \partial \Omega$, we also write $|E| = \sigma(E).$ 

    \item We use $\fint_E = \frac{1}{|E|} \int_E$ to represent the average integral over a set $E$. Here $|E|$ denotes either the volume of $E$ if it is a subset of $\R^d$, or the surface measure of $E$ if it is a subset of $\partial \Omega$.
    
    \item We use $B_r(X)$ or $B(X,r)$ to denote the Euclidean balls in $\R^d$ centered at $X$ with radius $r$. We use $\Delta_r(P) = \partial \Omega \cap B_r(P)$ for some $P\in \partial \Omega$ to denote the surface balls on $\partial \Omega$. Define $D_r(P) = \Omega \cap B_r(P)$ for $P\in \partial \Omega$.

    \item $\Omega_r = \{X \in \Omega: \dist(X,\partial \Omega) < r \}$ denotes the boundary layer of $\Omega$ with thickness $r$. Define $\Omega^r = \Omega \setminus \overline{\Omega_r}$.

    \item Let $\Omega$ be a bounded Lipschitz domain. For $Q\in \partial \Omega$, the nontangential region at $Q$ is defined by
    \begin{equation}
        \Gamma(Q) = \{ X\in \Omega: |X-Q| < \beta \dist(X,\partial \Omega)  \},
    \end{equation}
    where $\beta>1$ is a large fixed constant depending on the Lipschitz character of $\Omega$. A different version of nontangential cone will be defined in Lipschitz graph domains; see \eqref{def.Gamma.graph}.

    \item We will use various nontangential maximal functions on $ \partial \Omega$.
    Let $F \in C^0(\Omega)$. The classical nontangential maximal function $N(F)$ is defined by \eqref{NM}.
    In this paper, we define the large-scale (or precisely $\e$-scale) nontangential maximal function $N_\e (F)$ by \eqref{def.N_e}.
    
    \item Let $F \in L^2_{\rm loc}(\Omega)$ be either a scalar or a vector-valued function. The modified nontangential maximal function $\widetilde{N}(F)$ is given by \eqref{MNM}.
    Define the large-scale modified nontangential maximal function $\widetilde{N}_\e(F)$ by \eqref{def.tN_e}. 
    Note that  $\widetilde{N}(F)(Q) = \lim_{\e \to 0} \widetilde{N}_\e(F)(Q)$.
    For $r> 10\e$, define the truncated large-scale nontangential maximal function by $\widetilde{N}_\e^r(F)(Q)$; see \eqref{def.tNe.truncated}.
    A variant of the large-scale maximal function in a graph domain can be found in \eqref{def.hatN}.

    \item The conormal derivative of $u_\e$ is denoted by $\frac{\partial u_\e}{\partial \nu_\e} = n\cdot A^\e \nabla u_\e$, where $n$ is the outer unit normal vector on $\partial \Omega$.

    \item Given $f \in C^1(\partial \Omega)$, the tangential derivative (gradient) of $f$ on $\partial \Omega$ is given by $\nabla_{\tan } f = (I - n\otimes n) \nabla f$.


    \item For $u \in L^2(\Omega)$, define the ($\e$-scale) average operator $M_\e$ in $\Omega$ by
    \begin{equation}
        M_\e(u)(X) = \bigg( \fint_{B_\e(X) \cap \Omega} |u|^2 \bigg)^{1/2}.
    \end{equation}
    Clearly, $M_\e(u)(Q)$ can also be defined for $Q \in \partial \Omega$. Similarly, we define the boundary average operator by
    \begin{equation}
        M_\e^{\partial}(f)(Q) = \bigg( \fint_{\Delta_\e(Q)} |f|^2 d\sigma \bigg)^{1/2}.
    \end{equation}
    This definition will be modified if $\Omega$ is a graph domain; see \eqref{def.bdryMt.inIr}.

    \item Let $\psi \in C_0^\infty(\R^d)$ and ${\rm supp}(\psi) \subset B_{1/2}(0)$. Assume $\psi \ge 0$ and $\int \psi = 1$. Let $\psi_\e(X) = \e^{-d} \psi(X/\e)$. Define the smoothing operator $K_\e$ by
    \begin{equation}\label{def.Ke}
        K_\e f(X) = \psi_\e * f(X).
    \end{equation}



    \item For nonnegative quantities $a$ and $b$, we write $a\lesssim b$ if there exists some implicit constant $C$ independent of $a$ and $b$ such that $a\leq C b$. We write $a \gtrsim b$ if $b \lesssim a$. We write $a \simeq b$ if $a \lesssim b \lesssim a$. All the implicit constants in this paper depend at most on $d, \Lambda$, the exponent $p$ and the geometric characters of $\Omega$. In particular, they will never depend on the parameter $\e$.
    \end{itemize}

\textbf{Organization.} In Section \ref{sec.Dp}, we introduce the large-scale $\cL_\e$-harmonic measure and prove Theorem \ref{thm.Dp} in Lipschitz domains with $2-\delta<p\le \infty$. In Section \ref{sec.RpNp}, we prove Theorem \ref{thm.Rp}-\ref{thm.Np} in the case of Lipschitz domains for $1<p<2+\delta$. In Section \ref{sec.C1}, we prove Theorem \ref{thm.Dp}-\ref{thm.Np} in the case of bounded $C^1$ domains for the full range of $p\in (1,\infty)$. In Section \ref{sec.full-scale}, we prove Corollary \ref{coro.fullscale}. Some auxiliary analysis tools are stated or proved in appendices.

\textbf{Acknowledgments.} J. Zhuge is partially supported by NNSF of China (No. 12288201, 12494541, 12471115).

\section{Large-scale $\cL_\e$-harmonic measure and Dirichlet problem} \label{sec.Dp}

\subsection{$\cL_\e$-harmonic measure}

The $\cL_\e$-harmonic measure and Green function can be defined for the elliptic operator $\cL_\e = -\nabla\cdot (A^\e\nabla)$ in a bounded domain, as long as $A^\e$ is bounded measurable and satisfies the uniform ellipticity condition. In particular, the periodic structure or regularity on $A^\e$ is not needed.
Let $\{ \omega^X_\e \}_{X\in \Omega}$ denote the family of $\cL_\e$-harmonic measures for the operator $\cL_\e$ in $\Omega$. This means that if $f\in C(\partial \Omega)$, the classical solution of \eqref{Dp} can be expressed as
\begin{equation}\label{eq.hm}
    u_\e(X) = \int_{\partial \Omega} f d \omega^X_\e.
\end{equation}
Let $G_\e(X,Y)$ be the Green function of $\cL_\e$ in $\Omega$, i.e., for each $Y\in \Omega$, $\cL_\e(G_\e(\cdot,Y)) = \delta_Y(\cdot)$ in the sense of distributions and $G_\e(\cdot, Y) = 0$ on $\partial \Omega$. Moreover, $G_\e(X,Y) = G_\e^*(Y,X)$, where $G_\e^*$ is the Green function for the adjoint operator $\cL_\e^* = -\nabla\cdot (A^*(Y/\e) \nabla)$. In other words, $\cL_\e^*(G_\e(X,\cdot)) = \delta_X(\cdot)$ for each $X\in \Omega$.

We recall some basics of the $\cL_\e$-harmonic measure and Green function, which hold for any bounded measurable coefficients; see \cite{K94} for a collection of these materials. 
Let $\Delta_r(Q) = \partial \Omega \cap B_r(Q)$. Let $A_r(Q)$ be a point in $\Gamma(Q)$ such that $|A_r(Q) - Q| \simeq r$.
\begin{proposition}\label{prop.hmBasics}
    Let $\Omega$ be a bounded Lipschitz domain and $Q\in \partial \Omega$.
    \begin{enumerate}[label=(\roman*)]
        \item \label{item.doubling}
        $\omega^X_\e$ is doubling, i.e., for $X\in \Omega \setminus B_{4r}(Q)$,
        \begin{equation}
            \omega^X_\e(\Delta_{2r}(Q)) \lesssim \omega^X_\e(\Delta_{r}(Q)).
        \end{equation}
        \item \label{item.w=1} It holds
        \begin{equation}
            \omega_\e^{A_r(Q)}(\Delta_{r}(Q)) \simeq 1.
        \end{equation}
        
        \item  \label{item.w=G}  
        For $X\in \Omega \setminus B_{2r}(Q)$, 
    \begin{equation}\label{est.w=G}
        \omega_\e^X(\Delta_r(Q)) \simeq r^{d-2} G_\e(X, A_r(Q)).
    \end{equation}
        \item \label{item.G-L2}
        For $X\in \Omega \setminus B_{2r}(Q)$,
    \begin{equation}
        G_\e(X, A_r(Q)) \simeq \bigg( \fint_{\Omega \cap B_r(Q)} |G_\e(X,Y)|^2 dY\bigg)^{1/2}.
    \end{equation}

    \item \label{item.G.Holder}
    There exists $\alpha \in (0,1)$ such that for $X,Y\in \Omega$ and $\delta(X) < \frac12 |X-Y|$, 
    \begin{equation}
        G_\e(X,Y) \lesssim \frac{\delta(X)^\alpha}{|X-Y|^{d-2+\alpha}}.
    \end{equation}
    \end{enumerate}
\end{proposition}


\subsection{Localization}
\label{subsec.localization}
Though our main theorems are proved for bounded Lipschitz domains, it is standard to apply a localization argument to reduce the local boundary estimates to Lipschitz graph domains to avoid some technical issues. For example, for a bounded Lipschitz domain $\Omega$, $D_r(P) = \Omega\cap B_r(P)$ may not be a Lipschitz domain or even be disconnected. Here, we will briefly describe the localization argument.

Let $\Omega$ be a bounded Lipschitz domain. If we want to estimate $\widetilde{N}_\e(\nabla u)$, it suffices to estimate the truncated nontangential maximal function $\widetilde{N}_\e^R(\nabla u)$. This is due to the following observation
\begin{equation}\label{est.ReduceToLayer}
    \widetilde{N}_\e(u)(Q) \le \widetilde{N}_\e^{R}(u)(Q) + C_R \| \nabla u \|_{L^2(\Omega\setminus \Omega_{R/2})}.
\end{equation}
If we choose $R \simeq r_0 \gg \e$, where $r_0$ is a constant depending only on ${\rm diam}(\Omega)$ and the Lipschitz character of $\Omega$, then the second term on the right-hand side of \eqref{est.ReduceToLayer} can be easily controlled by the energy estimate. 
Consequently, we may concentrate on the estimates on the boundary layer $\Omega_{2r_0}$. Next, we cover $\Omega_{2r_0}$ by a finite number of balls $B(P_i,10r_0)$ centered at $P_i \in \partial \Omega$ with finite overlaps. This reduces the estimates to each $ D_{10r_0}(P_i) = B(P_i, 10r_0) \cap \Omega$. By a translation and a rotation, we can assume
\begin{equation}\label{eq.graphdomain}
    B(P_i, 10r_0) \cap \Omega = B(0,10r_0) \cap \left\{ X = (x',x_d): x_d > \phi(x') \right\},
\end{equation}
where $\phi:\R^{d-1} \to \R$ is a Lipschitz function, representing the local graph of $\partial \Omega$. We will use $x',y',z' \in \R^{d-1}$ to represent the first $d-1$ components of $X,Y,Z \in \R^d$ and $x_d, y_d, z_d$  the last components.


With the above localization and reduction, we only need to consider the boundary value problems in the Lipschitz  graph domain $\Omega = \{ X = (x', x_d): x_d > \phi(x') \}$, whose boundary is given by $\partial \Omega = \{ X = (x',\phi(x')) \}$, where $\phi$ is a Lipschitz function. 
Without loss of generality, assume $O = (0',0) \in \partial \Omega$, where $0'$ denotes the origin in $\R^{d-1}$.
Given $Q = (x',\phi(x')) \in \partial \Omega$, let 
\begin{equation}\label{I-r}
I_r(Q) = \left\{ P = (y',\phi(y')): |y'-x'|<r \right\}
\end{equation}
and 
\begin{equation}\label{T-r}
T_r(Q) = \left\{ Y = (y',y_d): |y'-x'|<r, \phi(y') < y_d < \phi(x') + M r \right\},
\end{equation}
where $M=100d(1+ \| \nabla \phi \|_\infty).$
These are analogs of $\Delta_r(Q)$ and $D_r(Q)$. Note that $T_r(Q)$ is always a Lipschitz domain and $I_r(Q)$ is the bottom boundary of $T_r(Q)$. The crucial fact is that the estimates in $I_r(Q)$ or $T_r(Q)$ are interchangeable with those in $\Delta_r(Q)$ and $D_r(Q)$ due to the following simple comparisons: 
\begin{equation}\label{eq.Tr=Dr}
    I_{cr}(Q) \subset \Delta_r(Q) \subset I_r(Q) \quad \text{and} \quad T_{cr}(Q) \subset D_r(Q) \subset T_{Cr}(Q),
\end{equation}
for some constants $0<c<1<C$. Oftentimes, we will simply write $I_r$ and $T_r$ if $Q = O \in \partial \Omega$, where $O$ is the origin.

    \begin{figure}[htbp]
\begin{tikzpicture}[scale = 0.5]
    \draw (-1, 1) -- (0,0.5) -- (1,-1) -- (2,0) -- (3,2) -- (4,-1) -- (5,1.4) -- (6, 0) -- (7,0.9) -- (8,-2) -- (9,2) -- (10, 0.4) -- (11, -0.2) -- (12, 1.2) -- (13, -0.4) -- (14, 0.1) -- (15, 1.6) -- (16, -0.1) -- (17, 0.3) -- (18, -1) -- (19, -1.2);

\draw[thick] (3.5, 0.5) -- (3.5, 8.5) -- (13.5, 8.5) -- (13.5, -0.15);


    \draw[blue, thick] (3.5, 0.5) -- (4,-1) -- (5,1.4) -- (6, 0) -- (7,0.9) -- (8,-2) -- (9,2) -- (10, 0.4) -- (11, -0.2) -- (12, 1.2) -- (13, -0.4) -- (13.5, -0.15);

    \node at (18,0) {$\partial \Omega$};
    \node[blue] at (11,-0.8) {$I_r(Q)$};
    \node at (5,7) {$T_r(Q)$};

    \coordinate (Q) at (8.5,0);
    \fill (Q) circle (1pt) node[below] {$Q$};

    \draw[red, thick] (Q) -- ++(83:6) node[midway, right] {$\Gamma(Q)$};
    \draw[red, thick] (Q) -- ++(97:6);

    \draw[dashed, red] (Q) -- ++(90:6);


\end{tikzpicture}
\caption{A graph domain}
\label{fig.graphdomain}
\end{figure}
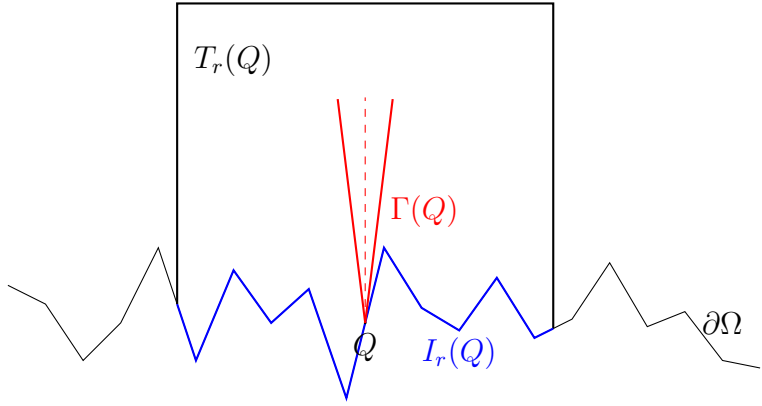

Note that although $\Omega$ is an unbounded graph domain, we can still define $\delta(X)$ as the distance from $X\in \Omega$ to the boundary $\partial \Omega$, and $\Omega_r = \{ X\in \Omega: \delta(X) < r \}$. For each $Q = (x',\phi(x')) \in \partial \Omega$, we redefine the nontangential cone on $Q$ by
\begin{equation}\label{def.Gamma.graph}
    \Gamma(Q) = \left\{Y = (y', y_d) \in \Omega: |x'-y'| < (\tan \beta) (y_d - \phi(x')) \right\}
\end{equation}
for some $\beta \in (0,\pi/2)$ (representing the half-aperture of the cone), which is a harmless modification of the original nontangential maximal function in a bounded domain. By this, all kinds of nontangential maximal functions can be defined analogously as in bounded Lipschitz domains.
For example, we redefine the truncated large-scale nontangential maximal function by
\begin{equation}\label{def.tNe.truncated}
    \widetilde{N}_\e^r(F)(Q) = \sup\bigg\{ \bigg(\fint_{B(Y,\delta(Y)/2)} |F|^2 \bigg)^{1/2}: Y\in \Gamma(Q) \cap (\Omega_{r}\setminus \Omega_{10\e}) \bigg\}.
\end{equation}

Finally, we will see how the operator $\cL_\e$ changes in the above localization. If $A^\e(X) = A(X/\e)$ is $\e$-periodic (i.e., under the assumption \eqref{periodicity}), then after a rotation it is still $\e$-periodic with respect to a rotated lattice. Here, we will introduce the idea in \cite{KS11} that can handle more general cases for which $A^\e$ only needs to be periodic in the $x_d$-direction (and has no structure assumption in the remaining $d-1$ directions). Actually, if $A$ satisfies \eqref{periodicity}, we can always choose a rational rotation in the above reduction to \eqref{eq.graphdomain} such that $A^\e(X)$ is still $\ell \e$-periodic in $x_d$-direction for some integer $\ell \ge 1$ (the size of $\ell$ depends on the Lipschitz character of $\partial \Omega$), i.e., $A^\e(x',x_d) = A^\e(x', x_d + z\ell \e)$ for all $z\in \Z$. Without loss of generality, we assume $\ell = 1$, which yields
\begin{equation}\label{eq.periodicInxd}
    A^\e(x',x_d) = A^\e(x', x_d + \e z ) \qquad \text{for any } z\in \Z.
\end{equation}
Indeed, the results in this section are available whenever $A^\e$ satisfies \eqref{eq.periodicInxd} in each localized graph domain, which is a  weaker assumption than \eqref{periodicity}.

\subsection{$L^2$ Rellich estimates}
\label{subsec.L2Rel}
There are two different approaches to achieve the large-scale $L^2$ Rellich estimates. The first one was introduced in \cite{KS11} using the difference operator under the assumption \eqref{eq.periodicInxd}, and the second was introduced in \cite{Shen17} based on the convergence rates in homogenization theory. These two approaches have their own advantages and disadvantages, though they are both applicable to periodic (in all directions) coefficients. The approach in \cite{KS11} relies on the strict periodicity in the $x_d$-direction (i.e., \eqref{eq.periodicInxd}); while the approach in \cite{Shen17} also applies to non-periodic coefficients, but needs the structure assumption on the coefficients in all directions so that the quantitative homogenization applies.
We also point out that the approach in \cite{KS11} requires $A$ to be symmetric, while the approach in \cite{Shen17} only requires symmetry for the Neumann condition. For simplicity, we will always assume that $A$ is symmetric in this paper.

In order to be consistent with the class of operators for regularity and Neumann problems, we recall the large-scale $L^2$ Rellich estimates in \cite{KS11} based on the assumption \eqref{eq.periodicInxd}.

\begin{proposition}[{\cite[Lemma 3.4]{KS11}}] \label{prop.localR2}
    Let $\e < r$ and $A^\e$ satisfy \eqref{ellipticity} and \eqref{eq.periodicInxd}. Let $u_\e$ be a solution of
        \begin{equation}\label{eq.Rp.T3r}
    \left\{
    \begin{aligned}
        & -\nabla\cdot A^\e \nabla u_\e = 0 \quad \text{in } T_{3r}, \\
        & u_\e = f  \quad \text{on } I_{3r}.
    \end{aligned}
    \right.
    \end{equation}
    Then for any $\e < t< r$,
    \begin{equation}\label{est.L2Rellich-1}
        \frac{1}{t} \int_{T_r \cap \Omega_t} |\nabla u_\e|^2 \lesssim \int_{I_{3r}} |\nabla_{\tan} f|^2  + \frac{1}{r} \int_{T_{3r}} |\nabla u_\e|^2.
    \end{equation}
\end{proposition}

\begin{proposition}[{\cite[Lemma 3.2]{KS11} }] \label{prop.localN2}
    Let $\e < r$ and $A^\e$ satisfy \eqref{ellipticity} and \eqref{eq.periodicInxd}. Let $u_\e$ be a solution of
        \begin{equation}\label{eq.Np.T3r}
    \left\{
    \begin{aligned}
        & -\nabla\cdot A^\e \nabla u_\e = 0 \quad \text{in } T_{3r}, \\
        & \frac{\partial u_\e}{\partial \nu_\e} = g \quad \text{on } I_{3r}.
    \end{aligned}
    \right.
    \end{equation}
    Then for any $\e < t< r$,
    \begin{equation}\label{est.L2Rellich-2}
        \frac{1}{t} \int_{T_r \cap \Omega_t} |\nabla u_\e|^2 \lesssim \int_{I_{3r}} |g|^2  + \frac{1}{r} \int_{T_{3r}} |\nabla u_\e|^2.
    \end{equation}
\end{proposition}

\begin{remark}
 The above propositions are not stated in the same form as \cite[Lemmas 3.2, 3.4]{KS11}. In this remark, we clarify that they are actually equivalent. First, the above propositions are stated in a scaled version with period $\e$, which definitely is equivalent to the original version in \cite[Lemmas 3.2, 3.4]{KS11}  with period 1. Second,
    \cite[Lemma 3.2, 3.4]{KS11} were stated for the flat boundary $\{ x_d = 0 \}$, while the above propositions are stated over Lipschitz boundaries $\{x_d = \phi(x') \}$. However, by a change of variables, $(x', x_d) \to (y',y_d) = (x', x_d - \phi(x'))$, we can flatten the boundary without changing the condition \eqref{eq.periodicInxd}. Third, the estimates \eqref{est.L2Rellich-1} and \eqref{est.L2Rellich-2} hold for all $t \in (\e,r)$, instead of a fixed scale $t = \e$, because if $A_\e$ is $\e$-periodic in $x_d$, then it is $m\e$-periodic for any $m\in \N$ and the desired estimates are valid for any $t \simeq [t/\e]\e$.
\end{remark}

The above propositions can be converted to $D_r$ in a bounded Lipschitz domain for the operator $\cL_\e$ using \eqref{eq.Tr=Dr} and a covering argument. Consequently, they implies the global $L^2$ Rellich estimates \eqref{eq.LSRellich-D} and \eqref{eq.LSRellich-N}. In fact, by \eqref{est.L2Rellich-1}, we have
\begin{equation}
        \frac{1}{t} \int_{D_{cr_0}(P_i) \cap \Omega_t} |\nabla u_\e|^2 \lesssim \int_{\Delta_{r_0}(P_i)} |\nabla_{\tan} f|^2  + \frac{1}{r_0} \int_{D_{r_0}(P_i)} |\nabla u_\e|^2.
    \end{equation}
Choosing $P_i$'s such that $\Omega_{cr_0}$ is covered by the union of $D_{cr_0}(P_i)$ with finite overlaps, then we sum over $i$ for the above inequality and obtain
\begin{equation}\label{est.L2RR}
    \frac{1}{t}\int_{\Omega_t} |\nabla u_\e|^2 \lesssim \int_{\partial \Omega} |\nabla_{\tan} f|^2 + \frac{1}{r_0}\int_{\Omega} |\nabla u_\e|^2.
\end{equation}
Finally, the energy estimate yields $\| \nabla u_\e \|_{L^2(\Omega)} \lesssim \| f \|_{H^1(\partial \Omega)}$, which combined with \eqref{est.L2RR} leads to \eqref{eq.LSRellich-D}. The derivation of \eqref{eq.LSRellich-N} is completely similar.

For an application to the Dirichlet problem in the next subsection, we state the following local estimate in $D_r$, which follows readily from \eqref{eq.Rp.T3r}.

\begin{proposition}\label{prop.LSRellich}
    Let $\e \le r$. Let $u_\e$ be a weak solution of $\cL_\e(u_\e) = 0$ in $D_{2r}(Q) = \Omega \cap B_{2r}(Q)$ for some $Q \in \partial \Omega$, and $u_\e = 0$ on $\Delta_{2r}(Q)$. Then for any $t\in [\e,r]$, 
    \begin{equation}
        \frac{1}{t} \int_{D_r(Q) \cap \Omega_t} |\nabla u_\e|^2 \lesssim \frac{1}{r} \int_{D_{2r}(Q)} |\nabla u_\e|^2.
    \end{equation}
\end{proposition}


\subsection{Large-scale $\cL_\e$-harmonic measure}
Define the large-scale $\cL_\e$-harmonic measure by
\begin{equation}\label{def.LSomega}
    \overline{\omega}^X_\e(Q) := \frac{\omega_\e^X(\Delta_\e(Q))}{\sigma(\Delta_\e(Q))}.
\end{equation}
For a Lipschitz domain $\Omega$, we have $\sigma(\Delta_\e(Q)) \simeq \e^{d-1}$.
By the doubling property of $\omega^X_\e$, we can show the doubling property of $\overline{\omega}^X_\e$.
\begin{lemma}\label{lem.doubling}
For any $0<r<r_0, \Delta_r = \Delta_r(P), X \in \Omega $ and $|X - P| > 5\max\{ r,\e \}$, we have
\begin{equation}\label{est.doubling}
    \int_{\Delta_{2r}} \overline{\omega}^X_\e(Q) d\sigma(Q) \lesssim \int_{\Delta_{r}} \overline{\omega}^X_\e(Q) d\sigma(Q).
\end{equation}
\end{lemma}
\begin{proof}
    First, consider $r<10\e$. In this case, by the doubling property of $\omega_\e^X$, for any $Q, Q' \in \Delta_{2\e}(P)$, we have $\omega_\e^X(\Delta_\e(Q)) \simeq  \omega_\e^X(\Delta_\e(Q'))$. 
    Here we have used the assumption $|X-P|>5\e$.
    This implies that $\overline{\omega}^X_\e(Q) \simeq \overline{\omega}^X_\e(Q') $ for any $Q, Q^\prime\in \Delta_{2r} \subset \Delta_{2\e}(P)$. Thus, \eqref{est.doubling} follows.

    Now consider the case $r\ge 10\e$. By Fubini's Theorem, we have
    \begin{equation}\label{est.doubling<}
        \begin{aligned}
            \int_{\Delta_{2r}} \overline{\omega}^X_\e(Q) d\sigma(Q) & = \int_{\Delta_{2r}} \int_{\Delta_{2r+\e}} \mathbbm{1}_{\{|Q'-Q|<\e\}} |\Delta_\e(Q')|^{-1} d\omega_\e^X(Q') d\sigma(Q) \\
            & \lesssim \int_{\Delta_{2r+\e}} \int_{\Delta_{2r}} \mathbbm{1}_{\{|Q'-Q|<\e\}} d\sigma(Q) \e^{1-d} d\omega_\e^X(Q') \\
            & \lesssim \int_{\Delta_{2r+\e}} d\omega_\e^X(Q') \\
            & \lesssim \omega_\e^X(\Delta_{2r+\e}) \lesssim \omega_\e^X(\Delta_{r-\e}),
        \end{aligned}
    \end{equation}
    where in the last line we have used the fact $2r+\e \le 3(r-\e)$ for $r\ge 10\e$ and the doubling property of $\omega_\e^X$ in Proposition \ref{prop.hmBasics} \ref{item.doubling}.

    On the other hand, using Fubini's Theorem again, we have
    \begin{equation}
    \begin{aligned}
        \int_{\Delta_{r}} \overline{\omega}^X_\e(Q) d\sigma(Q) & = \int_{\Delta_{r}} \int_{\Delta_{r+\e}} \mathbbm{1}_{\{|Q'-Q|<\e\}} |\Delta_\e(Q')|^{-1} d\omega_\e^X(Q') d\sigma(Q) \\
        & \gtrsim \int_{\Delta_{r+\e}} \int_{\Delta_{r}} \mathbbm{1}_{\{|Q'-Q|<\e\}} d\sigma(Q) \e^{1-d} d\omega_\e^X(Q').
    \end{aligned}
    \end{equation}
    Observe that for $Q'\in \Delta_{r-\e}$,
    \begin{equation}
        \int_{\Delta_{r}} \mathbbm{1}_{\{|Q'-Q|<\e\}} d\sigma(Q) = \sigma(\Delta_\e(Q')) \simeq \e^{d-1}.
    \end{equation}
    It follows that
    \begin{equation}\label{est.doubling>}
        \int_{\Delta_{r}} \overline{\omega}^X_\e(Q) d\sigma(Q) \gtrsim \int_{\Delta_{r-\e}} d\omega_\e^X(Q') = \omega_\e^X(\Delta_{r-\e}).
    \end{equation}
    Combining \eqref{est.doubling<} and \eqref{est.doubling>}, we get \eqref{est.doubling}.
\end{proof}



\begin{theorem}\label{thm.B2}
    It holds $\overline{\omega}_\e^X \in B_2(d\sigma)$ uniformly in $\e$, in the sense that, any $r > 0, P\in \partial \Omega, X \in \Omega$ and $|X-P| > 5\max \{r, \e \}$,
    \begin{equation}\label{est.B2}
        \bigg( \fint_{\Delta_r(P)} \overline{\omega}_\e^X(Q)^2 d\sigma(Q) \bigg)^{1/2} \lesssim \fint_{\Delta_{r}(P)} \overline{\omega}_\e^X(Q) d\sigma(Q).
    \end{equation}
\end{theorem}
\begin{proof}
    The case $r<10\e$ follows by the same reasoning as Lemma \ref{lem.doubling}. It suffices to assume $r>10\e$. Let $\Delta_r = \Delta_r(P)$ and $B_r = B_r(P)$.

    By Proposition \ref{prop.hmBasics} \ref{item.w=G}, \ref{item.G-L2}  and \eqref{def.LSomega}, for each $Q\in \Delta_r$,
    \begin{equation}
        \overline{\omega}_\e^X(Q) \simeq \e^{-1} G(X, A_\e(Q)) \lesssim \bigg( \fint_{B_{2\e}(Q) \cap \Omega} |\nabla_Y G_\e(X, Y)|^2 dY \bigg)^{1/2},
    \end{equation}
    where we also used the Poincar\'{e} inequality. Thus, by Fubini's Theorem, Proposition \ref{prop.LSRellich}, Caccioppoli's inequality and Proposition \ref{prop.hmBasics} \ref{item.G-L2}, \ref{item.w=G}, in order, we have
    \begin{equation}
    \begin{aligned}
        \int_{\Delta_r} \overline{\omega}_\e^X(Q)^2 d\sigma(Q) & \lesssim \e^{-d} \int_{\Delta_r} \int_{B_{2\e(Q)} \cap \Omega} |\nabla_Y G_\e(X,Y)|^2 dY d\sigma \\
        & \lesssim \e^{-d} \int_{\Omega_{2\e} \cap B_{1.5r}} |\nabla G(X,Y)|^2 \int_{\Delta_r} \mathbbm{1}_{|Y-Q| < 2\e} d\sigma dY\\
        & \lesssim \e^{-1} \int_{\Omega_{2\e} \cap B_{1.5r}} |\nabla G_\e(X,Y)|^2 dY \\
        & \lesssim r^{-1} \int_{D_{3r}} |\nabla G_\e(X,Y)|^2 dY \\
        & \lesssim r^{-3} \int_{D_{4r}} |G_\e(X, Y)|^2 dY \\
        & \lesssim r^{d-3} |G_\e(X, A_r(P))|^2 \\
        & \simeq r^{1-d} \omega_\e^X(\Delta_r)^2.
    \end{aligned}
    \end{equation}
    Dividing both sides by $r^{d-1}$ and taking square root, we obtain
    \begin{equation}
        \bigg( \fint_{\Delta_r} \overline{\omega}_\e^X(Q)^2 d\sigma(Q) \bigg)^{1/2} \lesssim \frac{\omega_\e^X(\Delta_r)}{\sigma(\Delta_r)} \lesssim  \frac{\omega_\e^X(\Delta_{r/2})}{\sigma(\Delta_r)} \lesssim \fint_{\Delta_{r}} \overline{\omega}_\e^X(Q) d\sigma(Q),
    \end{equation}
    where we also used the doubling property of $\omega^X_\e$ in Proposition \ref{prop.hmBasics} \ref{item.doubling}, and Fubini's Theorem in the last inequality.
\end{proof}

The estimate \eqref{est.B2} is a reverse H\"{o}lder inequality valid over all small scales $r>0$. The well-known  self-improving property of the reverse H\"{o}lder inequality implies the following.
\begin{corollary}\label{coro.2+delta}
    There exists $\delta>0$, depending on $d$, $\Lambda$ and $\Omega$, such that $\overline{\omega}_\e^X \in B_{2+\delta}(d\sigma)$ in the sense that for any $r>0, P\in \partial \Omega, X\in \Omega$ and $|X - P| > 5\max \{r, \e\}$,
    \begin{equation}
        \bigg( \fint_{\Delta_r(P)} \overline{\omega}_\e^X(Q)^{2+\delta} d\sigma(Q) \bigg)^{1/(2+\delta)} \lesssim \fint_{\Delta_{r}(P)} \overline{\omega}_\e^X(Q) d\sigma(Q).
    \end{equation}
\end{corollary}

\subsection{Dirichlet problem}

In this subsection, we consider the Dirichlet problem \eqref{Dp} and prove Theorem \ref{thm.Dp} for Lipschitz domains. The following is the key theorem that explains how the large-scale $\cL_\e$-harmonic measure is involved.

\begin{theorem}\label{thm.Neve<f}
    Let $f\in C(\partial \Omega)$ and
    \begin{equation}\label{def.ve}
        v_\e(X) = \int_{\partial \Omega} f(Q) \overline{\omega}^X_\e(Q) d\sigma(Q).
    \end{equation}
    Then $v_\e$ is a solution of
    \begin{equation}\label{eq.ve=barf}
    \left\{
    \begin{aligned}
        & \cL_\e(v_\e) = 0 \quad \text{in } \Omega, \\
        & v_\e(Q) = \bar{f}(Q) := \int_{\Delta_\e(Q)} |\Delta_\e|^{-1} f d\sigma \quad \text{on } \partial \Omega,
    \end{aligned}
    \right.
\end{equation}
where $|\Delta_\e| = |\Delta_\e(P)| \simeq \e^{d-1}$.
    Moreover, for any $p\in (2-\delta, \infty)$,
    \begin{equation}\label{est.Neve}
        \| N_\e(v_\e) \|_{L^p(\partial \Omega)} \lesssim \| f \|_{L^p(\partial \Omega)}.
    \end{equation}
\end{theorem}

\begin{proof}
    By Fubini's Theorem,
    \begin{equation}
    \begin{aligned}
        v_\e(X) & = \int_{\partial \Omega} \int_{\partial \Omega} |\Delta_\e(Q)|^{-1} f(Q) 1_{\Delta_\e(Q)}(P) d\omega^X_\e(P) d\sigma(Q) \\ & = \int_{\partial \Omega} \int_{\partial \Omega} |\Delta_\e(Q)|^{-1} f(Q) 1_{\Delta_\e(P)}(Q) d\sigma(Q) d\omega_\e^X(P) \\
        & = \int_{\partial \Omega} \int_{\Delta_\e(P)} |\Delta_\e(Q)|^{-1} f(Q) d\sigma(Q) d\omega_\e^X(P)\\
        & = \int_{\partial \Omega} \bar{f} d\omega_\e^X.
    \end{aligned}
    \end{equation}
    This proves \eqref{eq.ve=barf}.

    Next, we prove \eqref{est.Neve}. Let $X\in \Gamma(Q)$ and $\dist(X,\partial \Omega) > 10\e$. Let $r_0 = c|X - Q|$ and $r_j = 2^j r_0$. Let $R_0 = \Delta_{r_0}(Q)$, $R_j = \Delta_{r_j}(Q) \setminus \Delta_{r_{j-1}}(Q)$ for $j = 1,2, \cdots, L$, where $L$ is the minimal integer such that $2^L r_0 \ge c\diam(\Omega)$,
    and $R_{L+1} = \partial\Omega \setminus \Delta_{r_L}$. Then
    \begin{equation}
        \partial \Omega = \bigcup_{j = 0}^{L+1} R_j.
    \end{equation}
    Observe that for every $j = 0,1,\cdots, L+1$, we have $\dist(X, R_j) \simeq 2^j r_0$. Thus, Corollary \ref{coro.2+delta} and a covering argument yields
    \begin{equation}
        \bigg( \fint_{R_j} \overline{\omega}_\e^X(Q)^{2+\delta} d\sigma(Q) \bigg)^{1/(2+\delta)} \lesssim \fint_{R_j} \overline{\omega}_\e^X(Q) d\sigma(Q).
    \end{equation}

    Now, we bound \eqref{def.ve} by
    \begin{equation}\label{est.BdryDecomp}
    \begin{aligned}
    |v_\e(X)| & \lesssim \sum_{j = 0}^{L+1} \int_{R_j} |f(P')| \overline{\omega}_\e^X(P') d\sigma(P')\\
    & \le \sum_{j = 0}^{L+1} \bigg(\int_{R_j} |f(P')|^{p'_0} d\sigma(P') \bigg)^{1/p'_0} \bigg( \int_{R_j} \overline{\omega}_\e^X(P')^{p_0} d\sigma(P') \bigg)^{1/p_0} \\
    & \lesssim \sum_{j = 0}^{L+1} \bigg(\int_{R_j} |f(P')|^{p'_0} d\sigma(P') \bigg)^{1/p'_0} \int_{R_j} \overline{\omega}_\e^X(P') d\sigma(P') |R_j|^{-1/p^\prime_0}, 
    \end{aligned}
\end{equation}
where $p_0=2+\delta$.
By the definition of $\overline{\omega}_\e^X$ and Proposition \ref{prop.hmBasics} \ref{item.w=G} \ref{item.G.Holder},
\begin{equation}
    \int_{R_j} \overline{\omega}_\e^X(P') d\sigma(P') \lesssim \int_{R_j} d \omega_\e^X = \omega_\e^X(R_j) \simeq 2^{-j\alpha}.
\end{equation}
It follows from \eqref{est.BdryDecomp} that
\begin{equation}
\begin{aligned}
    |v_\e(X)| & \lesssim \sum_{j = 0}^{L+1} \bigg(\fint_{R_j} |f(P')|^{p'_0} d\sigma(P') \bigg)^{1/p'_0} 2^{-j \alpha} \\ 
    & \lesssim \big\{ \M_{\partial \Omega}(|f|^{p'_0})(Q) \big\}^{1/{p'_0}},
\end{aligned}
\end{equation}
where $\M_{\partial \Omega}(F)$ denotes the Hardy-Littlewood maximal function on $\partial \Omega$.

Taking supremum over all $X\in \Gamma(Q)$ with $\dist(X,\partial \Omega) > 10\e$, we arrive at
\begin{equation}
    N_\e(v_\e)(Q) \lesssim \big\{ \M_{\partial \Omega}(|f|^{p'_0})(Q)\big\}^{1/{p'_0}}.
\end{equation}
Now, for any $p> p'_0$, by the $L^q$ boundedness of Hardy-Littlewood maximal function for any $q>1$, we have
\begin{equation}
    \| N_\e(v_\e) \|_{L^p(\partial \Omega)} \lesssim \bigg( \int_{\partial \Omega} \big\{ \M_{\partial \Omega}(|f|^{p'_0})(Q)\big\}^{p/{p'_0}} d\sigma(Q) \bigg)^{1/p} \lesssim \| f \|_{L^p(\partial \Omega)}.
\end{equation}
Note that $p = \infty$ is trivial from the maximal principle.
This ends the proof.
\end{proof}

\begin{remark}
    The equations \eqref{def.ve} and \eqref{eq.ve=barf}  can be written as
\begin{equation}\label{eq.bar.dual}
    \int_{\partial \Omega} f(Q) \overline{\omega}^X_\e(Q) d\sigma(Q) = \int_{\partial \Omega} \bar{f}(Q) d\omega_\e^X(Q).
\end{equation}
\end{remark}

\begin{remark}\label{rmk.Bp-Dp'}
    The proof of Theorem \ref{thm.Neve<f} gives the following result. If $\overline{\omega}_\e^X \in B_p(d\sigma)$ for some $p>1$, namely, for $X\in \Omega\setminus B_{10r}(P)$,
    \begin{equation}
        \bigg( \fint_{\Delta_r(P)} \overline{\omega}_\e^X(Q)^p d\sigma(Q) \bigg)^{1/p} \lesssim \fint_{\Delta_r(P)} \overline{\omega}_\e^X(Q) d\sigma(Q),
    \end{equation}
    then $v_\e$ given by \eqref{def.ve} satisfies
    \begin{equation}
            \| N_\e(v_\e) \|_{L^{p'}(\partial \Omega)} \lesssim \| f \|_{L^{p'}(\partial \Omega)}.
    \end{equation}
    This large-scale estimate can be compared to the classical result for $\cL_\e$-harmonic functions; see e.g. \cite{D77,D79,FKP91}.
\end{remark}



\begin{proof}[\textbf{Proof of Theorem \ref{thm.Dp}}]
We use the representation formula by $\cL_\e$-harmonic measure,
\begin{equation}
    u_\e(X) = \int_{\partial \Omega} f(P) d\omega_\e^X(P).
\end{equation}
Observe that by the definition of $S_\e(f)$, we have $|f(P)| \le S_\e(f)(Q)$ for any $Q\in \Delta_\e(P)$. Since $|\Delta_\e(P)| \simeq |\Delta_\e(Q)|$, we have
$|f(P)| \lesssim \overline{S_\e(f)}(P)$.
Consequently, in view of \eqref{eq.bar.dual},
\begin{equation}
    |u_\e(X)| \lesssim \int_{\partial \Omega} \overline{S_\e(f)} d\omega_\e^X = \int_{\partial \Omega} S_\e(f) \overline{\omega}_\e^X d\sigma.
\end{equation}
Then \eqref{est.Dp} follows from this and \eqref{est.Neve} in Theorem \ref{thm.Neve<f}.
\end{proof}

Note that if $\Omega$ is a Lipschitz domain and $\e$ is small enough (depending on $\Omega$), then $\Omega^\e= \Omega \setminus \overline{\Omega_\e}$ is also a Lipschitz domain.
 Applying Theorem \ref{thm.Dp} to the subdomain $\Omega^\e $, we obtain
\begin{equation}\label{est.N<S}
    \| N_\e(u_\e) \|_{L^p(\partial \Omega)} \lesssim \| S_\e(u_\e) \|_{L^p(\partial \Omega^\e)}.
\end{equation}
Here, we have also used an observation that the large-scale nontangential maximal function on $\partial \Omega$ can be bounded by that on $\partial \Omega^\e$ with slight modifications (enlarge $\beta$ in \eqref{def.NTcone} and replace $10\e$ by $8\e$ in \eqref{def.N_e}) on the definition for the latter.
We find a bi-Lipschitz map $\tau: \partial \Omega \mapsto \partial \Omega^\e$ such that $Q^\e = \tau(Q)$ satisfies $B_{2\e}(Q_\e) \subset B_{4\e}(Q)$.
Hence, by the interior $L^\infty$ estimate, we have for $P \in \partial \Omega^\e$
\begin{equation}
    S_\e(u_\e)(Q^\e) = \sup_{ \substack{P\in \partial \Omega^\e \\ |P-Q^\e|<\e }} |u_\e(P)|  \lesssim \fint_{B_{2\e}(Q^\e) \cap \Omega} |u_\e| \lesssim \fint_{B_{4\e}(Q) \cap \Omega} |u_\e| = M_{4\e}(u_\e)(Q).
\end{equation}
It follows that $\| S_\e(u_\e) \|_{L^p(\partial \Omega^\e)} \lesssim \| M_{4\e}(u_\e) \|_{L^p(\partial \Omega)}$, which together with \eqref{est.N<S} leads to
\begin{equation}\label{N-M}
    \| N_\e (u_\e)\|_{L^p(\partial\Omega)}
    \lesssim \| M_{4\e} (u_\e)\|_{L^p(\partial\Omega)}.
\end{equation}
This is an alternative version of the large-scale $(D)_p$ estimates.


\section{Regularity and Neumann problems}
\label{sec.RpNp}

Apart from the $L^2$ Rellich estimate, for the regularity and Neumann problems, we also need the difference operator, which relies essentially on the periodicity of the coefficient matrix $A^\e$ in the $x_d$-direction as in \eqref{eq.periodicInxd} in graph domains (after localization). 

\subsection{Difference operator}

Let $u_\e$ be a weak solution of $-\nabla\cdot A^\e \nabla u_\e = 0$ in $T_{10r} = T_{10r}(O)$ with $r>10\e$. Assume that $A^\e$ satisfies \eqref{eq.periodicInxd}. We will estimate $\widetilde{N}_\e^r(u_\e)(Q)$ for $Q \in I_{r}$.

For $X = (x',x_d) \in \Omega$, define the difference operator by
\begin{equation}\label{def.Qe}
    \Q_\e(u)(X) = \frac{ u(X + \e e_d) - u(X)}{\e} = \frac{1}{\e} \int_{x_d}^{x_d + \e} \frac{\partial u}{\partial x_d}(x',t) dt.
\end{equation}
The key observation for the difference operator $\Q_\e$ is that if $u_\e$ is a solution of $-\nabla\cdot A^\e \nabla u_\e = 0$ with $A^\e$ satisfying \eqref{eq.periodicInxd}, then $\Q_\e(u_\e)$ is also a solution in the domain where $\Q_\e(u_\e)$ is defined. The following lemma reduces the local estimate of large-scale nontangential maximal functions to the difference estimate and a boundary layer term.

\begin{lemma}\label{lem.NT.ptws}
    Let $\e \lesssim r$ and $I_r = I_r(O)$ with $O \in \partial \Omega$. Let $u_\e$ be a solution of $-\nabla\cdot A^\e \nabla u_\e = 0$ in $T_{10r}$.
    Then for any $Q = (z',\phi(z')) \in I_r$,
    \begin{equation}\label{est.NeDue.local}
    \begin{aligned}
        \widetilde{N}_\e^r(\nabla u_\e)(Q) &\lesssim \M_{I_{3r}} (\widetilde{N}_\e^{3r}(\Q_\e(u_\e)))(Q)  + \M_{I_{3r}} (V_\e)(Q),
    \end{aligned}
    \end{equation}
    where
\begin{equation}\label{def.VeP}
    V_\e(P) = \fint_{0}^{11\e} |\nabla u_\e(y',\phi(y') + t)| dt, \ \text{with } P = (y',\phi(y')),
\end{equation}
    and
    \begin{equation}\label{def.MI3r}
        \M_{I_{3r}}f(Q) = \sup \bigg\{ \fint_{I'} |f| d\sigma: Q\in I' \subset I_{3r}\bigg\}.
    \end{equation}
\end{lemma}

\begin{proof}
Fix $Q_0 = (z',\phi(z')) \in I_r$ and $X \in \Gamma(Q_0) \cap (\Omega_{r} \setminus \Omega_{10\e})$. Let $s:= \delta(X) \le r$. 
By the interior Caccioppoli inequality,
\begin{equation}\label{est.Due2ue-E}
\begin{aligned}
    \bigg( \fint_{B(X,s/2)} |\nabla u_\e|^2 \bigg)^{1/2} & \lesssim \frac{1}{s} \bigg( \fint_{B(X,3s/5)} | u_\e - E|^2 \bigg)^{1/2} \\
    & \lesssim \frac{1}{s} \fint_{B(X,4s/5)} |u_\e - E|,
    \end{aligned}
\end{equation}
where $E$ is an arbitrary constant to be specified later.

Let $Y = (y',y_d) \in B(X, 4s/5)$.
Note that 
\begin{equation}
\begin{aligned}
    & |u_\e(y',y_d) - E| \\
    & \le |u_\e(y',y_d) - u_\e(y',y_d-\e)| +  |u_\e(y',y_d-\e) - u_\e(y',y_d-2\e)| \\
    & \qquad + \cdots + |u_\e(y',y_d-m\e) - u_\e(y',\phi(y') + h)| + |u_\e(y',\phi(y') + h) - E|,
    \end{aligned}
\end{equation}
where $m$ is an integer such that $y_d-m\e > \phi(y')+10\e \ge y_d-(m+1)\e $, and $h \in [0,\e]$ is a free parameter. On both sides taking average over $h\in [0,\e]$, we have
\begin{equation}\label{est.ue-E}
\begin{aligned}
     |u_\e(y',y_d) - E| & \le \sum_{j= 1}^{m-1} \e |\Q_\e(u_\e)(y',y_d - j \e)| \\
     & \qquad + \frac{1}{\e} \int_0^\e |u_\e(y',y_d-m\e) - u_\e(y',\phi(y') + h)|dh \\
     & \qquad + \frac{1}{\e} \int_0^\e |u_\e(y',\phi(y') + h) - E|dh \\
    & \le (m-1) \e \sup_{\phi(y')+10\e < t <\phi(y') + 3s} |\Q_\e(u_\e)(y',t)| \\
    & \qquad + \int_{\phi(y')}^{\phi(y') + 11\e} |\nabla u_\e(y',t)| dt \\
    & \qquad + \frac{1}{\e} \int_0^\e |u_\e(y',\phi(y') + h) - E| dh.
\end{aligned}
\end{equation}
Note that $m\e \simeq s$. We will write $u_\e(y',\phi(y') + h) = u_\e(Q + he_d)$ for $Q = (y', \phi(y'))$.

Integrating \eqref{est.ue-E} over $Y \in B(X,4s/5)$, we get
\begin{equation}\label{est.ue-E2}
\begin{aligned}
    & \fint_{B(X,4s/5)} |u_\e - E| \\
    & \lesssim s \fint_{I_s(x',\phi(x'))} \widetilde{N}_\e^{3r}(\Q_\e(u_\e)) d\sigma + \e \fint_{I_s(x',\phi(x'))} V_\e d\sigma \\
    & \qquad +  \frac{1}{\e} \int_0^\e \fint_{I_s(x',\phi(x'))} |u_\e(Q + he_d) - E| d\sigma(Q) dh,
    \end{aligned}
\end{equation}
where $V_\e$ is given by \eqref{def.VeP} and $(x',\phi(x'))$ is the vertical projection of $X = (x',x_d)$ on $\partial \Omega$.
Let $E = E_0 = \fint_{I_s(x',\phi(x'))} u_\e(Q) d\sigma(Q)$ and $E_h = \fint_{I_s(x',\phi(x'))} u_\e(Q + he_d) d\sigma(Q)$. Then by the Poincar\'{e} inequality,
\begin{equation}
    \fint_{I_s(x',\phi(x'))} |u_\e(Q + he_d) - E_h| \lesssim s \fint_{I_s(x',\phi(x'))} |\nabla_{\tan} u_\e(Q + he_d) | d\sigma(Q), 
\end{equation}
and
\begin{equation}
\begin{aligned}
    |E_h - E| & = \bigg| \fint_{I_s(x',\phi(x'))} ( u_\e(Q + he_d) - u_\e(Q)) d\sigma(Q)\bigg| \\
        & \le \fint_{I_s(x',\phi(x'))} \int_0^\e |\partial_{x_d} u_\e(Q + he_d)| dh d\sigma(Q).
\end{aligned}
\end{equation}
By these two inequalities, we have
\begin{equation}
\begin{aligned}
    & \frac{1}{\e} \int_0^\e \fint_{I_s(x',\phi(x'))} |u_\e(Q + he_d) - E| d\sigma(Q) dh \\
    & \le \fint_0^\e \fint_{I_s(x',\phi(x'))} |u_\e(Q + he_d) - E_h| d\sigma dh + \fint_0^\e |E_h - E| dh \\
    & \lesssim  s \fint_{I_s(x',\phi(x'))} \fint_0^\e |\nabla_{\tan} u_\e(Q + he_d) | dh d\sigma(Q) \\
    & \qquad + \e \fint_{I_s(x',\phi(x'))} \fint_0^\e |\partial_{x_d} u_\e(Q + he_d)| dh d\sigma(Q) \\
    & \lesssim   s \fint_{I_s(x',\phi(x'))} V_\e(Q) d\sigma(Q).
\end{aligned}
\end{equation}
Inserting this into \eqref{est.ue-E2}, we get
\begin{equation}
    \fint_{B(X,4s/5)} |u_\e - E| \lesssim s \fint_{I_s(x',\phi(x'))} \widetilde{N}_\e^{3r}(\Q_\e(u_\e)) d\sigma + s \fint_{I_s(x',\phi(x'))} V_\e d\sigma.
\end{equation}
Using the observation that $I_s(x',\phi(x')) \subset I_{2s}(Q_0) \subset I_{3r}(O)$ due to $X\in \Gamma(Q_0) \cap \Omega_r$, we have
\begin{equation}
    \frac{1}{s}\fint_{B(X,4s/5)} |u_\e - E| \lesssim \M_{I_{3r}} (\widetilde{N}_\e^{3r}(\Q_\e(u_\e)))(Q_0) + \M_{I_{3r}} (V_\e)(Q_0),
\end{equation}
which combined with \eqref{est.Due2ue-E} gives the desired estimate.
\end{proof}

\subsection{$L^2$ estimates}\label{subsec.L2}


The pointwise estimate \eqref{est.NeDue.local} and the $L^2$ boundedness of the Hardy-Littlewood maximal function imply
\begin{equation}\label{est.NeDue.RN}
\begin{aligned}
    \int_{I_{r}} \widetilde{N}_\e^{r}(\nabla u_\e)^2 d\sigma & \lesssim \int_{I_{3r}} \widetilde{N}_\e^{3r} (\Q_\e(u_\e))^2 d\sigma + \int_{I_{3r}} V_\e^2 d\sigma. 
    \end{aligned}
\end{equation}
Now it suffices to estimate the difference term with $\Q_\e(u_\e)$, which is also a solution.


Recall that $T_t = T_t(O) = \{ (x',x_d): |x'|<t, \phi(x') < x_d < \phi(0) + Mt \}$ is a Lipschitz domain.
Define $T_{t,\e} = (T_t)_\e = \{ X \in T_t: \dist(X, \partial T_t) < \e \}$.

\begin{lemma}\label{lem.NeQe}
    Let $u_\e$ be a solution of $-\nabla\cdot A^\e \nabla u_\e = 0$ in $T_{5r}$ with $r \ge 10 \e$. Then
    \begin{equation}
    \int_{I_{r}} \widetilde{N}_\e^{r} (\Q_\e(u_\e))^2 d\sigma \lesssim \frac{1}{\e} \int_{T_{5r} \cap \Omega_{5\e}} |\nabla u_\e|^2 + \int_{T_{5r}} |\nabla u_\e|^2.
\end{equation}
\end{lemma}
\begin{proof}
By the estimate \eqref{N-M} applied to $\Q_\e(u_\e)$ in $T_{t}$ for every $t \in [4r, 5r]$, we have
\begin{equation}
    \int_{I_{r}} \widetilde{N}_\e^{r} (\Q_\e(u_\e))^2 d\sigma \lesssim \int_{\partial T_{t}} |N_\e(\Q_\e(u_\e))|^2 \lesssim \frac{1}{\e} \int_{T_{t,4\e}} |\Q_\e(u_\e)|^2.
\end{equation}
Using \eqref{def.Qe}, we see that 
\begin{equation}
    \frac{1}{\e} \int_{T_{t,4\e}} |\Q_\e(u_\e)|^2 \lesssim \frac{1}{\e} \int_{T_{t+\e, 5\e}} |\nabla u_\e|^2. 
\end{equation}
Combining the last two inequalities, we obtain
\begin{equation}
    \int_{I_{r}} \widetilde{N}_\e^{r} (\Q_\e(u_\e))^2 d\sigma \lesssim \frac{1}{\e} \int_{T_{t+\e, 5\e}} |\nabla u_\e|^2.
\end{equation}
Note that the left-hand side is independent of $t$ as $t$ varies in $[4r,5r]$. Thus, we integrate the above inequality in $t$ over $[4r,5r - \e]$ and use Fubini's Theorem to obtain
\begin{equation}
    \int_{I_{r}} \widetilde{N}_\e^{r} (\Q_\e(u_\e))^2 d\sigma \lesssim \frac{1}{\e} \int_{T_{5r} \cap \Omega_{5\e}} |\nabla u_\e|^2 + \frac{1}{r} \int_{T_{5r}} |\nabla u_\e|^2.
\end{equation}
The proof is complete.
\end{proof}

\begin{corollary}\label{coro.localR2}
    Let $u_\e$ be a solution of $-\nabla\cdot A^\e \nabla u_\e =0 $ in $T_{15r}$ with $r \ge 10\e$. Then
    \begin{equation}\label{est.NeDue.localL2}
    \int_{I_{r}} \widetilde{N}_\e^{r}(\nabla u_\e)^2 d\sigma  \lesssim \frac{1}{\e} \int_{T_{15r} \cap \Omega_{20\e}} |\nabla u_\e|^2 + \frac{1}{r}\int_{T_{15r}} |\nabla u_\e|^2.
    \end{equation}
\end{corollary}
\begin{proof}
    This follows from \eqref{est.NeDue.RN} and Lemma \ref{lem.NeQe}.
\end{proof}



\begin{proof}[\textbf{Proof of Theorems \ref{thm.Rp} and \ref{thm.Np} for $p = 2$}]
    First of all, by rotation and a localization argument as in Subsection \ref{subsec.localization}, the local estimate of large-scale nontangential maximal function \eqref{est.NeDue.localL2}, established in a graph domain, can be transferred into a form in a bounded Lipschitz domain, i.e.,
    \begin{equation}
    \int_{\Delta_{r}} \widetilde{N}_\e^{r}(\nabla u_\e)^2 d\sigma \lesssim \frac{1}{\e} \int_{D_{15r} \cap \Omega_{20\e}} |\nabla u_\e|^2 + \frac{1}{r}\int_{D_{15r}} |\nabla u_\e|^2. 
    \end{equation}
    By taking $r \simeq r_0$ as in Subsection \ref{subsec.localization} and a covering argument, we have
    \begin{equation}\label{est.RN.L2}
        \| \widetilde{N}_\e(\nabla u_\e) \|_{L^2(\partial \Omega)} \lesssim \bigg( \frac{1}{\e} \int_{\Omega_{20\e} } |\nabla u_\e|^2 \bigg)^{1/2} + \| \nabla u_\e \|_{L^2(\Omega)}.
    \end{equation}
    Now,  if $u_\e$ satisfies \eqref{Rp}, then the energy estimate implies
    \begin{equation}
        \| \nabla u_\e \|_{L^2(\Omega)} \lesssim  \|  f \|_{H^1(\partial \Omega)}.
    \end{equation}
    The large-scale Rellich estimate \eqref{eq.LSRellich-D} implies
    \begin{equation}
        \bigg( \frac{1}{\e} \int_{\Omega_{20\e} } |\nabla u_\e|^2 \bigg)^{1/2} \lesssim  \|  f \|_{H^1(\partial \Omega)}.
    \end{equation}
    Combining these estimates together, we obtain \eqref{est.Rp} for $p = 2$. Similarly, for the Neumann problem we apply the energy estimate and the large-scale Rellich estimate \eqref{eq.LSRellich-N} to the right-hand side of \eqref{est.RN.L2} to obtain \eqref{est.Np} for $p = 2$.
\end{proof}


\begin{corollary}[Localized $L^2$ estimates]
    Let $u_\e$ be a solution of $-\nabla \cdot A^\e \nabla u_\e = 0$ in $T_{2r}$. Then for $r> 10\e$,
    \begin{equation}\label{est.localR2}
        \int_{I_r} \widetilde{N}^r_\e(\nabla u_\e)^2 d\sigma \lesssim \int_{I_{2r}}  |\nabla_{\tan} u_\e|^2 d\sigma + \frac{1}{r}\int_{T_{2r}} |\nabla u_\e|^2,
    \end{equation}
    and
    \begin{equation}\label{est.localN2}
        \int_{I_r} \widetilde{N}^r_\e(\nabla u_\e)^2 d\sigma \lesssim \int_{I_{2r}}  \Big|\frac{\partial u_\e}{\partial \nu_\e}\Big|^2 d\sigma + \frac{1}{r}\int_{T_{2r}} |\nabla u_\e|^2.
    \end{equation}
\end{corollary}
\begin{proof}
    This follows from Corollary \ref{coro.localR2} and Propositions \ref{prop.localR2} and \ref{prop.localN2}.
\end{proof}

\subsection{$L^p$ estimates for $1<p<2$}
It is classical that the $(R)_2$ (resp. $(N)_2$) estimate will directly imply the $(R)_p$ (resp. $(N)_p$) estimates for any $1<p<2$, by an interpolation between $L^2$ estimate and $L^1$ estimate (with data in atom Hardy space $H^1_{\rm at}(\partial \Omega)$); see \cite[Theorem 5.2, 6.2]{KP93} (also see \cite{DK87,KS11}).
In this subsection, we will adjust the classical argument to the large-scale $L^p$ estimate with $1<p<2$, only using the large-scale $(R)_2$ or $(N)_2$ estimates.


We first consider the regularity problem.

\begin{proposition}\label{prop.R1}
    If $f\in C^1(\partial \Omega)$ and $u_\e$ is a solution of $\cL_\e(u_\e)=0$ in $\Omega$ with
    $u_\e=f$ on $\partial\Omega$, then
\begin{equation}\label{est.R1}
    \| \widetilde{N}_\e(\nabla u_\e) \|_{L^1(\partial \Omega)} \lesssim \| \nabla_{\tan} f \|_{H^1_{\rm at}(\partial \Omega)}.
\end{equation}
\end{proposition}

It is sufficient to prove \eqref{est.R1} for an atom function $f$, whose support is contained in $\Delta_r(P)$ for some $P\in \partial \Omega$ with $r>0$ and $|\nabla_{\tan} f| \le Cr^{1-d}$. Let $u_\e$ be the solution of $\cL_\e(u_\e)=0 $ in $\Omega$ with Dirichlet data $f$. Then it suffices to show
\begin{equation}\label{est.DeuL1<1}
    \| \widetilde{N}_\e(\nabla u_\e) \|_{L^1(\partial \Omega)} \lesssim 1.
\end{equation}

First of all, by H\"{o}lder's inequality and the $(R)_2$ estimate, we have
\begin{equation}\label{est.Delta10r}
\begin{aligned}
    \int_{\Delta_{10r}(P)} |\widetilde{N}_\e(\nabla u_\e)| & \le |\Delta_{10r}(P)|^{1/2}\bigg( \int_{\Delta_{10r}(P)} |\widetilde{N}_\e(\nabla u_\e)|^2 \bigg)^{1/2}  \\
    & \lesssim r^{(d-1)/2} \bigg( \int_{\Delta_r(P)} |\nabla_{\tan} f|^2 \bigg)^{1/2} \lesssim 1.
\end{aligned}
\end{equation}

Next, we decompose the rest of the boundary as
\begin{equation}\label{eq.Omega.decomp}
    \partial \Omega \setminus \Delta_{10r}(P) = \bigcup_{j=1}^{L+1} R_j,
\end{equation}
where $R_j = \Delta_{r_j}(P) \setminus \Delta_{r_{j-1}}(P)$, $r_j = 2^j 10 r$ for $1\le j \le L$, $R_{L+1} =\partial\Omega \setminus \Delta_{r_L}$,  and $L$ is an integer such that $r_L \approx \diam(\Omega)$. We will estimate $\widetilde{N}_\e(\nabla u_\e)(Q)$ for $Q \in R_j$ for each $j \ge 1$.

\begin{lemma}\label{lem.QinRj}
    For $Q\in R_j$, we have
    \begin{equation}\label{est.L1inRj}
        \widetilde{N}_\e(\nabla u_\e)(Q) \lesssim \widetilde{N}_\e^{cr_j}(\nabla u_\e)(Q) + \frac{r^\alpha}{r_j^{d-1+\alpha}},
    \end{equation}
    where $c\in (0,1)$ is a constant independent of $r_j$ and $\e$.
    Moreover, if $cr_j \le 10\e$, then the first term on the right-hand side of \eqref{est.L1inRj} does not appear.
\end{lemma}
\begin{proof}
    Let $X \in \Gamma(Q)$ with $\delta(X) \ge \max \{ 10\e, cr_j \}$. By the Caccioppoli inequality,
    \begin{equation}\label{est.R1.NeDue2ue}
        \bigg( \fint_{B(X,\delta(X)/2)} |\nabla u_\e|^2 \bigg)^{1/2} \lesssim \delta(X)^{-1} \bigg( \fint_{B(X,\frac35\delta(X))} | u_\e|^2 \bigg)^{1/2}.
    \end{equation}
    Note that for each $Y\in B(X,\frac35 \delta(X))$, we still have $\delta(Y) \ge \frac25 \max\{ 10\e, cr_j \}$. By using the representation by $\cL_\e$-harmonic measure \eqref{eq.hm}, we have
    \begin{equation}\label{est.ptwise.uY}
    \begin{aligned}
        |u_\e(Y)| & \le \| f \|_{L^\infty(\Delta_r(P))} \omega^Y_\e(\Delta_r(P)) \lesssim r^{d-2} \omega^Y_\e(\Delta_r(P)) \\
        & \simeq G_\e(Y, A_r(P)) \lesssim   \frac{r^\alpha }{|Y-P|^{d-2+\alpha}},
    \end{aligned}
    \end{equation}
    where we have used Proposition \ref{prop.hmBasics} \ref{item.w=G} and \ref{item.G.Holder} in the last inequality.
    Obviously, $|Y-P| \ge \delta(Y) \ge \frac25 cr_j$. Hence, for every $Y \in B(X, \frac35 \delta(X))$,
    \begin{equation}
        |u(Y)| \lesssim \frac{ r^\alpha}{ r_j^{d-2+\alpha}}.
    \end{equation}
    Combining this with \eqref{est.R1.NeDue2ue}, we obtain that for any $X\in \Gamma(Q)$ and $\delta(X) \ge \max\{ 10\e, r_j\}$,
    \begin{equation}
        \bigg( \fint_{B(X,\delta(X)/2)} |\nabla u_\e|^2 \bigg)^{1/2} \lesssim \frac{ r^\alpha}{ r_j^{d-1+\alpha}}.
    \end{equation}
    Thus, if $r_j \ge 10\e$, then we have \eqref{est.L1inRj}. If $r_j < 10\e$, we simply do not have the term $N_\e^{r_j}(\nabla u_\e)(Q)$ on the right-hand side.
\end{proof}


\begin{proof}[\textbf{Proof of Proposition \ref{prop.R1}}]
    It suffices to prove \eqref{est.DeuL1<1}.
    We show on each $R_j$ with $1\le j\le L+1$,
    \begin{equation}\label{est.rjPj}
        \int_{R_j} |\widetilde{N}_\e(\nabla u_\e)| d\sigma \lesssim  \Big( \frac{r}{r_j} \Big)^\alpha.
    \end{equation}    
    
    If $c r_j \le 10 \e$, then by Lemma \ref{lem.QinRj} and an integration over $R_j$, we directly get
    \begin{equation}
        \int_{R_j} |\widetilde{N}_\e(\nabla u_\e)| d\sigma \lesssim \int_{R_j } \frac{ r^\alpha}{ r_j^{d-1+\alpha}} d\sigma \lesssim \Big( \frac{r}{r_j} \Big)^\alpha.
    \end{equation}

    If $c r_j > 10\e$, by Lemma \ref{lem.QinRj},
    \begin{equation}\label{est.intof.rj>e}
        \int_{R_j} \widetilde{N}_\e(\nabla u_\e) d\sigma \lesssim \int_{R_j} \widetilde{N}_\e^{c r_j}(\nabla u_\e) d\sigma + \Big( \frac{r}{r_j} \Big)^\alpha.
    \end{equation}
    Note that $u_\e = 0$ on $\partial\Omega \setminus \Delta (P, r)$. We decompose $R_j$ into a union of $\Delta_{cr_j}(Q_k)$ with finite overlaps. As $\Omega$ is Lipschitz, the number of such $\Delta_{cr_j}(Q_k)$ depends only on $\Omega$. Then by \eqref{est.localR2} and the Caccioppoli inequality,
    \begin{equation}\label{est.localNrj}
        \int_{\Delta_{cr_j}(Q_k)} \widetilde{N}_\e^{cr_j}(\nabla u_\e)^2 d\sigma \lesssim \frac{1}{r_j} \int_{D_{2c r_j}(Q_k)} |\nabla u_\e|^2 \lesssim \frac{1}{r_j^3} \int_{D_{3c r_j}(Q_k)} |u_\e|^2.
    \end{equation}
    Using the same argument as Lemma \ref{lem.QinRj}, for $X\in D_{3cr_j}(Q_k)$,
    \begin{equation}\label{est.ptwise.uY1}
        |u_\e(X)| \lesssim \frac{r^\alpha}{r_j^{d-2+\alpha}}.
    \end{equation}
    Inserting this into \eqref{est.localNrj}, we get
    \begin{equation}
        \int_{\Delta_{cr_j}(Q_k)} \widetilde{N}_\e^{cr_j}(\nabla u_\e)^2 d\sigma \lesssim \frac{r^{2\alpha}}{r_j^{d-1+2\alpha}}.
    \end{equation}
    Summing over $k$ and combining this with \eqref{est.intof.rj>e} and the H\"{o}lder inequality, we arrive at \eqref{est.rjPj}.

    Consequently, by \eqref{est.Delta10r}, \eqref{eq.Omega.decomp} and \eqref{est.rjPj}, we have
    \begin{equation}
        \int_{\partial \Omega} |\widetilde{N}_\e(\nabla u_\e)| \le \int_{\Delta_{10r}(P)}|\widetilde{N}_\e(\nabla u_\e)| + \sum_{j}  \int_{R_j} |\widetilde{N}_\e(\nabla u_\e)| d\sigma \lesssim 1.
    \end{equation}
    The proof is complete.
\end{proof}

\begin{proof}[\textbf{Proof of Theorem \ref{thm.Rp} for $1<p<2$}]
    This follows by  interpolating  between the $(R)_2$ estimate (already proved in Subsection \ref{subsec.L2}) and $(R)_1$ estimate  in Proposition \ref{prop.R1} for the sublinear operator $f \mapsto N_\e(\nabla u_\e)$.
\end{proof}

Next, we prove the $(N)_p$ estimate for $1<p<2$ in Theorem \ref{thm.Np}.
We first recall the Neumann function $\mathcal{N}_\e(X,Y)$; see \cite{KP93} for the existence and basic estimates that will be used below. For any $g\in L_0^2(\partial \Omega)$, the weak solution of \eqref{Np} satisfying $\int_{\partial \Omega} u_\e d\sigma = 0$ is given by
\begin{equation}\label{eq.NP}
    u_\e(X) = \int_{\partial \Omega} \mathcal{N}_\e(X,Q) g(Q) d\sigma(Q).
\end{equation}
Moreover, the Neumann function satisfies the following estimates independent of the regularity or structure of $A$ (see \cite[Definition 2.5 and Lemma 2.10]{KP93}),
\begin{equation}
    |\mathcal{N}_\e(X,Y)| \lesssim |X-Y|^{2-d},
\end{equation}
and there exists $\alpha \in (0,1)$ such that if $|Z-X| \le \frac12 |X-Y|$ (see \cite[Corollary 2.14]{KP93}),
\begin{equation}\label{est.NF.Holder}
    |\mathcal{N}_\e(X,Y) - \mathcal{N}_\e(Z,Y)| \lesssim \frac{|X-Z|^\alpha}{|X-Y|^{d-2+\alpha}}.
\end{equation}
Since $A$ is symmetric, we have $\mathcal{N}_\e(X,Y) = \mathcal{N}_\e(Y,X)$.
Moreover, the compatibility condition holds
\begin{equation}
    \int_{\partial \Omega} \mathcal{N}_\e(P,Y) d\sigma(P) = \int_{\partial \Omega} \mathcal{N}_\e(X,Q) d\sigma(Q) = 0.
\end{equation}

\begin{proposition}\label{prop.N1}
    If $g \in H_{\rm at}^1(\partial \Omega)$ with $\int_{\partial \Omega} g d\sigma = 0$, then the solution given by \eqref{eq.NP} satisfies
\begin{equation}\label{est.NP.L1}
    \| \widetilde{N}_\e(\nabla u_\e) \|_{L^1(\partial \Omega)} \lesssim \| g \|_{H^1_{\rm at}(\partial \Omega)}.
\end{equation}
\end{proposition}

\begin{proof}
It suffices to consider an atom $g$ supported in $\Delta_r(P)$ for some $P\in \partial \Omega$ and $|g| \le Cr^{1-d}$, and prove
\begin{equation}\label{est.N1}
    \| \widetilde{N}_\e(\nabla u_\e) \|_{L^1(\partial \Omega)} \lesssim 1.
\end{equation}

Similar to the regularity problem, in $\Delta_{10r}(P)$, we apply the $L^2$ estimate for the Neumann problem to obtain
\begin{equation}\label{est.N1.P0}
    \int_{\Delta_{10r}(P)} |N_\e(\nabla u_\e)| \lesssim 1.
\end{equation}
We decompose $\partial \Omega \setminus \Delta_{10r}(P)$ as in \eqref{eq.Omega.decomp}. On each $R_j$, we can show
\begin{equation}\label{est.N1.Pj}
    \int_{R_j} |N_\e(\nabla u_\e)| \lesssim \Big( \frac{r}{r_j} \Big)^\alpha.
\end{equation}
Comparing the regularity problem, the only difference is the pointwise estimate of $u_\e(Y)$ in \eqref{est.ptwise.uY} for $Y \in \Gamma(Q')$ with $Q' \in R_j$. Instead of using the $\cL_\e$-harmonic measure, we now use the representation of Neumann function. Due to the fact $\int_{\partial \Omega} g d\sigma = 0$,
\begin{equation}
    u_\e(Y) = \int_{\Delta_r(P)} (\mathcal{N}_\e(Y,Q) - \mathcal{N}_\e(Y,P)) g(Q) d\sigma(Q).
\end{equation}
By \eqref{est.NF.Holder}, for $Y\in \Gamma(Q'), Q'\in R_j$,
\begin{equation}
    |u_\e(Y)| \lesssim \int_{\Delta_r(P)} \frac{r^\alpha}{r_j^{d-2+\alpha}} r^{1-d} d\sigma \lesssim \frac{r^\alpha}{r_j^{d-2+\alpha}},
\end{equation}
which is the same as \eqref{est.ptwise.uY1}. Thus, by the interior Caccioppoli inequality 
\begin{equation}
    \widetilde{N}_\e(\nabla u_\e)(Q') \lesssim \widetilde{N}_\e^{cr_j}(\nabla u_\e)(Q') + \frac{r^\alpha}{r_j^{d-1+\alpha}},
\end{equation}
for any $Q'\in \Delta_{r_j}(P_j)$, which together with \eqref{est.localN2} yields \eqref{est.N1.Pj}. Summing \eqref{est.N1.P0} and \eqref{est.N1.Pj} over $j$, we obtain \eqref{est.N1}. 
\end{proof}

\begin{proof}[\textbf{Proof of Theorem \ref{thm.Np} for $1<p<2$} ]
The $(N)_p$ estimate for $1<p<2$ follows by interpolation for the sublinear operator $g \mapsto \widetilde{N}_\e(\nabla u_\e)$ between the $(N)_1$ estimate in Proposition \eqref{prop.N1} and $(N)_2$ estimate proved previously.
\end{proof}



\subsection{$L^p$ estimates for $2<p<2+\delta$}
It is classical that the $L^p$ regularity and Neumann estimates for $2<p<2+\delta$ follow from the $L^2$ estimate and a real-variable argument; see \cite[Theorem 5.3, 6.3]{KP93}.
In this subsection, we prove the large-scale $L^p$ regularity/Neumann estimates for $2<p<2+\delta$ by modifying the real-variable argument in \cite{KS11}. To this end, we first consider a local problem in the graph domain $T_r$. We introduce 
a modified large-scale nontangential maximal function: for $Q \in \partial \Omega$,
\begin{equation}\label{def.hatN}
    \widehat{N}_\e(F)(Q) := \widetilde{N}_\e( F)(Q) + M_{10\e}^\star(F)(Q),
\end{equation}
where 
\begin{equation}\label{def.Mstar}
    M_{10\e}^\star(F)(Q) = \bigg( \fint_{ T_{10\e}^\star(Q)} |F|^2 \bigg)^{1/2}
\end{equation}
and
\begin{equation}
    T_{r}^\star(Q) = \{ X\in \Omega: \text{proj}(X) \in I_r(Q), \delta(X) < r \}.
\end{equation}
Here $\proj(X)$ denotes the vertical projection of $X = (x',x_d)$ on $\partial \Omega$, i.e., $\proj(X) = (x', \phi(x')) \in \partial \Omega$.
This modified large-scale nontangential maximal function includes the information of $F$ on the boundary layer of thickness $10\e$ and is useful in proving the reverse H\"{o}lder inequality. 


\begin{lemma}\label{lem.RH}
Let $u_\e$ be a solution of $\cL_\e(u_\e) = 0$ in $T_{4R}$ satisfying either $u_\e = 0$ or $\frac{\partial u_\e}{\partial \nu_\e} = 0$ on $I_{4R}$ with $R \ge  100\e$. Then there exists $p_0 > 2$ such that
\begin{equation}\label{est.ReverseHolder.Nhat}
    \bigg( \fint_{I_R} | \widehat{N}_\e (\nabla u_\e)|^{p_0} d\sigma \bigg)^{1/{p_0}} \lesssim \bigg( \fint_{I_{2R}} |\widehat{N}_\e(\nabla u_\e)|^{2} d\sigma \bigg)^{1/{2}}.
\end{equation}
\end{lemma}

\begin{proof}
This is proved by the self-improving property of the reverse H\"{o}lder inequality.
Let $P\in I_R$ and $I_{4r}(P) \subset I_{2R}, r>0$. We only need to prove
\begin{equation}\label{est.ReverseHolder.L2L1}
    \bigg( \fint_{I_r(P)} |\widehat{N}_\e(\nabla u_\e)|^2 d\sigma  \bigg)^{1/2} \lesssim \bigg( \fint_{I_{3r}(P)} |\widehat{N}_\e(\nabla u_\e)|^{p_*} d\sigma \bigg)^{p_*},
\end{equation}
where $p_* = \frac{2d}{d+2} < 2$. Then \eqref{est.ReverseHolder.Nhat} is a straightforward corollary of the self-improving property of the reverse H\"{o}lder inequality.

\textbf{Case 1: $r<50\e$.} We will show that for this case, a reverse H\"{o}lder inequality holds for $\widehat{N}_\e(\nabla u_\e)$ by its definition, irrelevant to the equation. To see this,
we first claim that if $r<50\e$, for each $Q \in I_r(P)$,
\begin{equation}\label{est.claim-1}
    \widetilde{N}_\e(\nabla u_\e)(Q) \lesssim \fint_{I_{r}(Q)} \widetilde{N}_\e(\nabla u_\e) d\sigma.
\end{equation}
To show the claim, we fix $Q$ and let $X\in \Gamma(Q)$ be such that $\delta(X) > 10\e$ and $\widetilde{N}_\e(\nabla u_\e)(Q)$ is attained at $X$, namely,
\begin{equation}
    \widetilde{N}_\e(\nabla u_\e)(Q) \simeq \bigg( \fint_{B(X,\delta(X)/2)} |\nabla u_\e|^2 \bigg)^{1/2}.
\end{equation}
Now, in view of the definition of the cone $\Gamma(Q)$ in graph domain \eqref{def.Gamma.graph}, we see that there exists a subset $E \subset I_r(Q)$ with $|E| \simeq |I_r(Q)|$ such that for each $Q' \in E$, $X \in \Gamma(Q')$. This implies $\widetilde{N}_\e(\nabla u_\e)(Q) \lesssim \widetilde{N}_\e(\nabla u_\e)(Q')$ for each $Q'\in E$. Hence,
\begin{equation}
    \widetilde{N}_\e(\nabla u_\e)(Q) \lesssim \frac{1}{|E|} \int_{E} \widetilde{N}_\e(\nabla u_\e) d\sigma \lesssim \fint_{I_r(Q)} \widetilde{N}_\e(\nabla u_\e) d\sigma,
\end{equation}
which proves the claim \eqref{est.claim-1}. 

The claim \eqref{est.claim-1} then easily yields
\begin{equation}\label{est.L2reverse-1}
    \bigg( \fint_{I_r(P)} |\widetilde{N}_\e(\nabla u_\e)|^2 d\sigma  \bigg)^{1/2} \lesssim \fint_{I_{2r}(P)} \widetilde{N}_\e(\nabla u_\e) d\sigma.
\end{equation}

Next, by a property of the average operator \eqref{est.Mef-Bs-1} in Proposition \ref{prop.Mef-Bs} (the property is proved for $M_\e$, but it holds for $M_{10\e}^\star$ as well), for any $s \lesssim 10 \e$, we have
\begin{equation}
    M_{10\e}^\star(\nabla u_\e)(Q) \lesssim \fint_{I_s(Q)} M_{10\e}^\star(\nabla u_\e) d\sigma.
\end{equation}
Taking $s = r < 50\e$ and any $Q\in I_r(P)$, we have
\begin{equation}\label{est.L2reverse-2}
    \bigg( \fint_{I_r(P)} |M_{10\e}^\star(\nabla u_\e)|^2 d\sigma  \bigg)^{1/2} \lesssim \fint_{I_{2r}(P)} M_{10\e}^\star(\nabla u_\e) d\sigma.
\end{equation}
Combining \eqref{est.L2reverse-1} and \eqref{est.L2reverse-2}, we obtain, for $r<50\e$,
\begin{equation}
    \bigg( \fint_{I_r(P)} |\widehat{N}_\e(\nabla u_\e)|^2 d\sigma  \bigg)^{1/2} \lesssim \fint_{I_{2r}(P)} \widehat{N}_\e(\nabla u_\e) d\sigma.
\end{equation}

\textbf{Case 2: $r>50\e$.} In this large-scale case, we first note that for each $Q\in \Delta_r$,
\begin{equation}\label{est.case2}
    \widetilde{N}_\e(\nabla u_\e)(Q) \lesssim \fint_{I_{2r}(P)} \widetilde{N}_\e(\nabla u_\e) d\sigma + \widetilde{N}_\e^r(\nabla u_\e)(Q).
\end{equation}
Indeed, if  $\widetilde{N}_\e(\nabla u_\e)(Q)$ is reached at a point $X$ with $\delta(X) > r$, then, for the same reason as in Case 1, $\widetilde{N}_\e(\nabla u_\e)(Q)$ is bounded by the first term of \eqref{est.case2}. If $\widetilde{N}_\e(\nabla u_\e)(Q)$ is attained at a point $X$ with $\delta(X)<r$, then it is clearly bounded by the second term of \eqref{est.case2}.

We estimate the second term of \eqref{est.case2}.
Using either the local $(R)_2$ estimate \eqref{est.localR2} if $u_\e = 0$ on $I_{4R}$ or the local $(N)_2$ estimate \eqref{est.localN2} if $\frac{\partial u_\e}{\partial \nu_\e} = 0$ on $I_{4R}$,
\begin{equation}\label{est.RH-1}
\begin{aligned}
    \bigg( \fint_{I_r(P)} \widetilde{N}_\e^r(\nabla u_\e)^2 d\sigma \bigg)^{1/2} 
    &\lesssim  \bigg( \fint_{T_{2r}(P)} |\nabla u_\e|^2 \bigg)^{1/2} \\
    & \lesssim  \frac{1}{r} \bigg( \fint_{T_{3r}(P)} |u_\e|^2 \bigg)^{1/2} \\
    & \lesssim  \bigg( \fint_{T_{3r}(P)} |\nabla u_\e|^{p_*} \bigg)^{1/p_*},
\end{aligned}
\end{equation}
where we have used the Caccioppoli inequality and Sobolev-Poincar\'{e} inequality and $1/p_* - 1/d = 1/2$. Note that in the case $\frac{\partial u_\e}{\partial \nu_\e}=0$, we need to replace $u_\e$ by $u_\e-\fint_{T_{3r}(P)} u_\e$ in \eqref{est.RH-1}.

Now we decompose the last integral into two parts
\begin{equation}
\begin{aligned}\label{est.T3r.p*}
    \bigg( \fint_{T_{3r}(P)} |\nabla u_\e|^{p_*} \bigg)^{1/p_*} & \lesssim \bigg( \frac{1}{r^d} \int_{T_{3r}(P) \cap \{ \delta(X) > 10\e \}} |\nabla u_\e|^{p_*} \bigg)^{1/p_*} \\
    & \qquad + \bigg( \frac{1}{r^d} \int_{T_{3r}(P) \cap \{ \delta(X) < 10\e \}} |\nabla u_\e|^{p_*} \bigg)^{1/p_*}.
    \end{aligned}
\end{equation}
Note that the projection $T_{3r}$ on $\partial \Omega$ is $I_{3r}$. We point out that the average operator $M_{10\e}^\star$ is introduced in the modified nontangential maximal function in order to handle the boundary layer term in \eqref{est.T3r.p*} that cannot be avoided in our large-scale analysis. In view of the definition of  $\widetilde{N}_\e(\nabla u_\e)$, if $X \in T_{3r} \cap \{ \delta(X) > 10\e \}$ has the projection $Q \in I_{3r}$, then by the large-scale Lipschitz estimate 
\begin{equation}
    M_\e(\nabla u_\e)(X) \lesssim \bigg( \fint_{B(X, \delta(X)/2)} |\nabla u_\e|^2 \bigg)^{1/2} \lesssim  \widetilde{N}_\e(\nabla u_\e)(Q).
\end{equation}
Thus, by a property of $M_\e$ in \eqref{est.f<Mef} of Proposition \ref{prop.A1} and Fubini's Theorem, we have
\begin{equation}\label{est.T3r.int}
\begin{aligned}
    \bigg( \frac{1}{r^d} \int_{T_{3r}(P) \cap \{ \delta(X) > 10\e \}} |\nabla u_\e|^{p_*} \bigg)^{1/p_*} & \lesssim \bigg( \frac{1}{r^d} \int_{T_{3r}(P) \cap \{ \delta(X) > 10\e \}} |M_\e(\nabla u_\e)|^{p_*} \bigg)^{1/p_*} \\
    & \lesssim \bigg( \fint_{I_{3r}(P)} |\widetilde{N}_\e(\nabla u_\e)|^{p_*} d\sigma \bigg)^{1/p_*}.
\end{aligned}
\end{equation}
Moreover, by Fubini's Theorem and H\"{o}lder inequality,
\begin{equation}\label{est.T3r.layer}
\begin{aligned}
    & \bigg( \frac{1}{r^d} \int_{T_{3r}(P) \cap \{ \delta(X) < 10\e \}} |\nabla u_\e|^{p_*} \bigg)^{1/p_*} \\
    & \lesssim \bigg( \frac{\e}{r^d} \int_{I_{3r}(P)} \fint_{T_{10\e}^\star(Q)} |\nabla u_\e|^{p_*} dX d\sigma(Q) \bigg)^{1/p_*} \\
    & \lesssim \bigg( \frac{\e}{r^d} \int_{I_{3r}(P)} \bigg( \fint_{T_{10\e}^\star(Q)} |\nabla u_\e|^{2} dX \bigg)^{p_*/2} d\sigma(Q) \bigg)^{1/p_*} \\
    & \lesssim \bigg( \frac{\e}{r} \fint_{I_{3r}(P)} |M_{10\e}^\star(\nabla u_\e)|^{p_*} d\sigma \bigg)^{1/p_*}.
    \end{aligned}
\end{equation}
Taking both \eqref{est.T3r.int} and \eqref{est.T3r.layer} into \eqref{est.T3r.p*}, we have
\begin{equation}\label{est.T3r-I3r}
    \bigg( \fint_{T_{3r}(P)} |\nabla u_\e|^{p_*} \bigg)^{1/p_*} \lesssim \bigg( \fint_{I_{3r}(P)} \widehat{N}_\e(\nabla u_\e)^{p_*} d\sigma \bigg)^{1/p_*},
\end{equation}
which together with \eqref{est.RH-1} yields
\begin{equation}\label{est.tNe-hNe}
    \bigg( \fint_{I_r(P)} \widetilde{N}_\e(\nabla u_\e)^2 d\sigma \bigg)^{1/2} \lesssim \bigg( \fint_{I_{3r}(P)} \widehat{N}_\e(\nabla u_\e)^{p_*} d\sigma \bigg)^{1/p_*}.
\end{equation}

To handle $M_{10\e}^\star(\nabla u_\e)$, by the $L^2$ Rellich estimate in $T_{4r}(P)$ for either the regularity problem or Neumann problem, we obtain
\begin{equation}
\begin{aligned}
    \bigg( \fint_{I_r(P)} | M_{10\e}^\star(\nabla u_\e)|^2 d\sigma \bigg)^{1/2} & \lesssim 
    \bigg( \frac{1}{\e r^{d-1}} \int_{T_{3r/2} \cap \Omega_{10\e} } |\nabla u_\e|^2 \bigg)^{1/2} \\
    & \lesssim \bigg( \fint_{T_{2r}} |\nabla u_\e|^2 \bigg)^{1/2} \\
    & \lesssim \bigg( \fint_{I_{3r}} \widehat{N}_\e(\nabla u_\e)^{p_*} d\sigma \bigg)^{1/p_*},
    \end{aligned}
\end{equation}
where we have used the combined estimate of \eqref{est.RH-1} and \eqref{est.T3r-I3r} in the last inequality.
This and \eqref{est.tNe-hNe} together give \eqref{est.ReverseHolder.L2L1} for the case $r>50\e$.
Consequently, we have proved the reverse H\"{o}lder inequality \eqref{est.ReverseHolder.L2L1} for all $I_{4r}(P) \subset I_{2R}$ and $r>0$, as desired.
\end{proof}

\begin{proof}[\textbf{Proof of Theorem \ref{thm.Rp} for $2<p<2+\delta$}]
We first prove a localized result in a graph domain. Let $u_\e$ be a solution of $\cL_\e(u_\e) = 0$ in $T_{4R}$ satisfying $u_\e = f$ in $I_{4R}$ with $R \ge 1000\e$. Then we show that there exists $\delta > 0$ such that for any $p\in (2, 2+\delta)$,
\begin{equation}\label{est.RpInTR}
\begin{aligned}
     & \bigg( \fint_{I_R} | \widetilde{N}_\e (\nabla u_\e)|^{p} d\sigma \bigg)^{1/{p}} \\
     & \lesssim \bigg( \fint_{I_{3R}} |\widetilde{N}_\e(\nabla u_\e)|^{2} d\sigma \bigg)^{1/{2}} + \bigg( \fint_{I_{4R}\cap \Omega_{10\e}} |\nabla u_\e|^{2} d\sigma \bigg)^{1/{2}} + \bigg( \fint_{I_{4R}} |\nabla_{\tan} f|^{p} d\sigma \bigg)^{1/p}. 
\end{aligned}
\end{equation}

We apply the real-variable argument. Let $ I_r(P) \subset I_R$ with $4r < R$ and $r>100\e$. Let $\phi \in C_0^1(B_{5r}(P))$ be the cut-off function such that $\phi = 1$ in $B_{4r}(P)$ and $|\nabla \phi| \lesssim r^{-1}$. Let $\lambda = \fint_{I_{4r}} f d\sigma$. Then we define $f_1 = (f-\lambda)\phi$ and let $u_\e^1$ be the solution of $\cL_\e(u_\e^1) = 0$ and $u_\e^1 = f_1 \mathbbm{1}_{I_{5r}(P)}$ on $\partial T_{4R}$. Decompose $u_\e - \lambda = u_\e^1 + u_\e^2$ and note that $\nabla u_\e = \nabla u_\e^1 + \nabla u_\e^2$.

Let $F = \widehat{N}_\e(\nabla u_\e), F_{I_r} = \widehat{N}_\e(\nabla u^1_\e)$ and $R_{I_r} =\widehat{N}_\e(\nabla u_\e^2)$. By applying the $(R)_2$ in the Lipschitz domain $T_{4R}$ and the $L^2$ Rellich estimate \eqref{eq.LSRellich-D}, we have
\begin{equation}
\begin{aligned}
    \fint_{I_{2r}(P)} |F_{I_{r}}|^2 & \lesssim \frac{1}{|I_{2r}(P)|} \int_{I_{4r}} |\widehat{N}_\e(\nabla u^1_\e)|^2 \\
    & \lesssim \frac{1}{|I_{2r}(P)|} \int_{ I_{4r}} |\widetilde{N}_\e(\nabla u^1_\e)|^2 + \frac{1}{|I_{2r}(P)|\e} \int_{T_{5r}(P) \cap \Omega_{10\e}} |\nabla u^1|^2 \\
    &\lesssim \fint_{I_{5r}(P)} |\nabla_{\tan}((f-\lambda)\phi)|^2 \\
    & \lesssim \fint_{I_{5r}(P)} |\nabla_{\tan} f|^2,
\end{aligned}
\end{equation}
where we have used 
\begin{equation}\label{est.Mstar-Layer}
    \int_{I_{4r}} M_{10\e}^\star(F)^2 \lesssim \frac{1}{\e} \int_{T_{5r}(P) \cap \Omega_{10\e}} |F|^2
\end{equation}
in the second inequality and
the Poincar\'{e} inequality in the last inequality.

For $R_{I_r}$, note that $u_\e^2 = 0$ on $I_{4r}(P)$. Then, Lemma \ref{lem.RH} yields, for some $p_0 = 2+\delta$,
\begin{equation}
\begin{aligned}
    \bigg(\fint_{I_r(P)} |R_{I_r}|^{p_0} \bigg)^{1/p_0} & \lesssim \bigg(\fint_{I_{2r}(P)} |R_{I_{r}}|^{2} \bigg)^{1/2} \\
    &\lesssim \bigg(\fint_{I_{2r}(P)} |F|^{2} \bigg)^{1/2} + \bigg(\fint_{I_{2r}(P)} |F_{I_{r}}|^{2} \bigg)^{1/2} \\
    & \lesssim \bigg(\fint_{I_{2r}(P)} |F|^{2} \bigg)^{1/2} + \bigg(\fint_{I_{5r}(P)} |\nabla_{\tan} f|^{2} \bigg)^{1/2}.
    \end{aligned}
\end{equation}
As a consequence of Theorem \ref{thm.real-variable>t}, we have
\begin{equation}
\begin{aligned}
    & \bigg( \fint_{I_R} | M_{100\e}^{\partial} (\widehat{N}_\e (\nabla u_\e))|^{p} d\sigma \bigg)^{1/{p}}  \\
    & \lesssim \bigg( \fint_{I_{2R}} |M_{100\e}^{\partial}(\widehat{N}_\e(\nabla u_\e))|^{2} d\sigma \bigg)^{1/{2}} + \bigg( \fint_{I_{3R}} |M_{100\e}^{\partial}(\nabla_{\tan} f)|^{p} d\sigma \bigg)^{1/p} \\
     & \lesssim \bigg( \fint_{I_{3R}} |\widehat{N}_\e(\nabla u_\e)|^{2} d\sigma \bigg)^{1/{2}} + \bigg( \fint_{I_{4R}} |\nabla_{\tan} f|^{p} d\sigma \bigg)^{1/p}\\
     & \lesssim \bigg( \fint_{I_{3R}} |\widetilde{N}_\e(\nabla u_\e)|^{2} d\sigma \bigg)^{1/{2}} + \bigg( \fint_{I_{4R}\cap \Omega_{10\e}} |\nabla u_\e|^{2} d\sigma \bigg)^{1/{2}} + \bigg( \fint_{I_{4R}} |\nabla_{\tan} f|^{p} d\sigma \bigg)^{1/p},
    \end{aligned}
\end{equation}
for any $p\in (2,p_0)$, where we have used  Proposition \ref{prop.A1} to remove the boundary average operators $M_{100\e}^\partial$ in the second inequality and \eqref{est.Mstar-Layer} in the last inequality. Now, to remove $M_{100\e}^\partial$ on the left-hand side, we use the fact
$0\le \widetilde{N}(\nabla u_\e) \le \widehat{N}_\e(\nabla u_\e)$ and \eqref{est.claim-1}, due to the definition of the large-scale nontangential maximal function. This yields \eqref{est.RpInTR} for $2<p<2+\delta$.

Finally, taking $R = r_0$ and applying a localization argument as in Subsection \ref{subsec.localization}, we obtain
 \begin{equation}
 \begin{aligned}
    &\| \widetilde{N}_\e(\nabla u_\e) \|_{L^p(\partial \Omega)} \\
    & \lesssim \| \widetilde{N}_\e^R(\nabla u_\e) \|_{L^p(\partial \Omega)} + \| \nabla u \|_{L^2(\Omega\setminus \Omega_{R/2})} \\
    & \lesssim \| \widetilde{N}_\e(\nabla u_\e) \|_{L^2(\partial \Omega)} + \bigg( \frac{1}{\e} \int_{\Omega_{10\e}} |\nabla u_\e|^2 \bigg)^{1/2} + \| \nabla_{\tan} f \|_{L^p(\partial \Omega)} + \| \nabla u \|_{L^2(\Omega\setminus \Omega_{R/2})} \\
    & \lesssim \| \nabla_{\tan} f \|_{L^2(\partial \Omega)} + \| \nabla_{\tan} f \|_{L^p(\partial \Omega)} \\
    & \lesssim \| \nabla_{\tan} f \|_{L^p(\partial \Omega)},
 \end{aligned}
 \end{equation}
 where we have used the $(R)_2$ estimate (proved earlier),  the global $L^2$ Rellich estimate \eqref{eq.LSRellich-D} and energy estimate in the third inequality, and the H\"{o}lder inequality in the last inequality. This completes the proof of \eqref{est.Rp} for $2<p<2+\delta$.
\end{proof}

\begin{proof}[\textbf{Proof of Theorem \ref{thm.Np} for $2<p<2+\delta$}] 
The proof is similar to the regularity problem. It suffices to prove a local estimate in a graph domain $T_{4R}$ for some $R\ge 1000\e$. Let $u_\e$ be a solution in $T_{4R}$ with $\frac{\partial u_\e}{\partial \nu_\e} = g$ on $I_{4R}$. We would like to show that there exists $\delta>0$ such that for any $p\in (2,2+\delta)$,
\begin{equation}
\begin{aligned}
     & \bigg( \fint_{I_R} | \widetilde{N}_\e (\nabla u_\e)|^{p} d\sigma \bigg)^{1/{p}} \\
     & \lesssim \bigg( \fint_{I_{3R}} |\widetilde{N}_\e(\nabla u_\e)|^{2} d\sigma \bigg)^{1/{2}} + \bigg( \fint_{I_{4R}\cap \Omega_{10\e}} |\nabla u_\e|^{2} d\sigma \bigg)^{1/{2}} + \bigg( \fint_{I_{4R}} |g|^{p} d\sigma \bigg)^{1/p}. 
\end{aligned}
\end{equation}

Again, this is proved by the real-variable argument.
Let $ I_r(P) \subset I_R$ with $4r < R$ and $r>100\e$. Let 
$$\lambda = |\partial T_{4R} \setminus I_{4r}(P)|^{-1} \int_{I_{4r}(P)} g d\sigma. $$ 
Then we define $g_1 = g \mathbbm{1}_{I_{4r}(P)} + \lambda \mathbbm{1}_{\partial T_{4R} \setminus I_{4r}(P)}$. Then we have $\int_{\partial T_{4R}} g_1 = 0$. Let $u_\e^1$ be the solution of $\cL_\e(u_\e^1) = 0$ in $T_{4R}$ and $ \frac{\partial u_\e^1}{\partial \nu_\e} = g_1$ on $\partial T_{4R}$. Decompose $u_\e = u_\e^1 + u_\e^2$ and note that $\frac{\partial u_\e^2}{\partial \nu_\e} = 0$ on $I_{4r}(P)$.

Let $F = \widehat{N}_\e(\nabla u_\e), F_{I_r} = \widehat{N}_\e(\nabla u^1_\e)$ and $R_{I_r} =\widehat{N}_\e(\nabla u_\e^2)$. By the $(N)_2$ estimate in the Lipschitz domain $T_{4R}$ and the large-scale $L^2$ Rellich estimate \eqref{eq.LSRellich-N}, we have
\begin{equation}
\begin{aligned}
    \fint_{I_{2r}(P)} |F_{I_{r}}|^2 & \lesssim \frac{1}{|I_{2r}(P)|} \int_{I_{2r}(P)} |\widehat{N}_\e(\nabla u^1_\e)|^2 \\
    & \lesssim \frac{1}{|I_{2r}(P)|}\int_{I_{4R}} |g_1|^2 
    \lesssim \fint_{I_{4r}(P)} |g|^2.
\end{aligned}
\end{equation}

For $R_{I_r}$, Lemma \ref{lem.RH} yields, for some $p_0 = 2+\delta$,
\begin{equation}
\begin{aligned}
    \bigg(\fint_{I_r(P)} |R_{I_r}|^{p_0} \bigg)^{1/p_0} & \lesssim \bigg(\fint_{I_{2r}(P)} |R_{I_{r}}|^{2} \bigg)^{1/2} \\
    &\lesssim \bigg(\fint_{I_{2r}(P)} |F|^{2} \bigg)^{1/2} + \bigg(\fint_{I_{2r}(P)} |F_{I_{r}}|^{2} \bigg)^{1/2} \\
    & \lesssim \bigg(\fint_{I_{2r}(P)} |F|^{2} \bigg)^{1/2} + \bigg(\fint_{I_{4r}(P)} |g|^{2} \bigg)^{1/2}.
    \end{aligned}
\end{equation}
Now, we are in a situation similar to the regularity problem. By Theorem \ref{thm.real-variable>t} and a localization argument as before, we can prove \eqref{est.Np}. The details are omitted.
\end{proof}

\section{$C^1$ domains}\label{sec.C1}

The estimates for the large-scale nontangential maximal functions will be proved for the full range of $p\in (1,\infty)$ for $(D)_p, (R)_p$ and $(N)_p$ problems in $C^1$ domains. The new ingredients we need from the $C^1$ domains are the $L^p$ estimates of the classical nontangential maximal functions for the homogenized operator $\cL_0$ (see \cite{Fabes1978} and  Appendix \ref{A-D}), and the large-scale boundary $W^{1,p}$ estimate for $\cL_\e$ (see Theorem \ref{prop.W1p.local} below). Throughout this section, we assume that $A$ satisfies \eqref{ellipticity} and \eqref{periodicity}.

\subsection{Localized $L^p$ Rellich estimates}
In this subsection, we will establish the large-scale $L^p$ Rellich estimates for any $p\in (2,\infty)$ in $C^1$ domains. Unlike the large-scale $L^2$ Rellich estimate, which can be derived in graph domains under the assumption \eqref{eq.periodicInxd}, the $L^p$ Rellich estimates rely on the quantitative convergence rates in $C^1$ domains. In the case of VMO coefficients, the large-scale $L^p$ Rellich estimates have been established in \cite{Shen17}. Since in this paper we do not assume any regularity on the coefficients, we will work with the averaged gradient $M_\e(\nabla u_\e)$ instead of $\nabla u_\e$ to avoid the local irregularity caused by the rough coefficients. This crucial modification leads to a sequence of large-scale estimates and some of them are even new for harmonic functions.


For the localized Rellich estimates, it is more convenient to work in a $C^1$ graph domain $\Omega$. For $Q \in I_r(Q)$, we redefine the boundary average operator as
\begin{equation}\label{def.bdryMt.inIr}
    M_t^{\partial}(f)(Q) = \bigg( \fint_{I_t(Q)} |f|^2 d\sigma \bigg)^{1/2}.
\end{equation}
The following are the main theorems of this subsection.
\begin{theorem}\label{thm.LocalRellich.Rp}
    Assume that $\Omega$ is a $C^1$ graph domain.
    Let $1\gtrsim r \gtrsim \e$. Let $u_\e$ be a weak solution of
    \begin{equation}\label{eq.Rp.T2r}
    \left\{
    \begin{aligned}
        & \cL_\e(u_\e) = 0 \quad \text{in } T_{2r}, \\
        & u_\e = f \in W^{1,2}(I_{2r}) \quad \text{on } I_{2r}.
    \end{aligned}
    \right.
    \end{equation}
    Then for any $p\in [2,\infty)$ and $\e \le t < r$,
    \begin{equation}\label{est.localRellich.Rp}
        \bigg( \frac{1}{t} \int_{\Omega_t \cap T_{r/2}} |M_\e(\nabla u_\e)|^p \bigg)^{1/p} \lesssim \bigg( \int_{I_{2r}} |M_\e^\partial (\nabla_{\tan} f)|^p \bigg)^{1/p} + r^{\frac{d-1}{p} - \frac{d}{2}} \bigg( \int_{T_{2r}} |\nabla u_\e|^2 \bigg)^{1/2}.
    \end{equation}
\end{theorem}

\begin{theorem}\label{thm.LocalRellich.Np}
    Assume that $\Omega$ is a $C^1$ graph domain.
    Let $1\gtrsim r \gtrsim \e$. Let $u_\e$ be a weak solution of
    \begin{equation}\label{eq.Np.T2r}
    \left\{
    \begin{aligned}
        & \cL_\e(u_\e) = 0 \quad \text{in } T_{2r}, \\
        & \frac{\partial u_\e}{\partial \nu_\e} = g \in L^2(I_{2r}) \quad \text{on } I_{2r}.
    \end{aligned}
    \right.
    \end{equation}
    Then for any $p\in [2,\infty)$ and $\e \le t < r$,
    \begin{equation}\label{est.localRellich.Np}
        \bigg( \frac{1}{t} \int_{\Omega_t \cap T_{r/2}} |M_\e(\nabla u_\e)|^p \bigg)^{1/p} \lesssim \bigg( \int_{I_{2r}} |M_\e^\partial (g)|^p \bigg)^{1/p} + r^{\frac{d-1}{p} - \frac{d}{2}} \bigg( \int_{T_{2r}} |\nabla u_\e|^2 \bigg)^{1/2}.
    \end{equation}
\end{theorem}




We will concentrate on the proof of Theorem \ref{thm.LocalRellich.Rp} for the regularity problem. The proof of Theorem \ref{thm.LocalRellich.Np} is similar.
The proof relies on the nontangential maximal function of the homogenized solution and the optimal convergence rates in $L^p$ spaces for the gradient $\nabla u_\e$.

We begin with a general lemma that relates the large-scale nontangential maximal function $\widehat{N}_\e$ defined by \eqref{def.hatN} to the classical nontangential maximal function $\widetilde{N}$.

\begin{lemma}\label{lem.hatN<M10eNDu}
    Let $ r > 100\e$. Let $u \in H^1(T_{2r})$. Then for any $Q \in I_r$, we have
    \begin{equation}
        \widehat{N}_\e(\nabla u)(Q) \lesssim M_{10\e}^\partial (\widetilde{N}(\nabla u))(Q).
    \end{equation}
\end{lemma}
\begin{proof}
    First, by a reasoning similar to \eqref{est.claim-1},
    \begin{equation}
        \widetilde{N}_\e(\nabla u)(Q) \lesssim \fint_{I_{10\e}(Q)} \widetilde{N}(\nabla u) d\sigma \lesssim M_{10\e}^\partial (\widetilde{N}(\nabla u))(Q).
    \end{equation}
    Next, by Fubini's Theorem,
    \begin{equation}
    \begin{aligned}
        M_{10\e}^\star(\nabla u)(Q) & = \bigg(\fint_{T_{10\e}^\star(Q)} |\nabla u|^2 \bigg)^{1/2} \\
        & \lesssim \bigg( \frac{10\e}{|T_{10\e}^\star(Q)|} \int_{I_{10\e}(Q)}  \widetilde{N}(\nabla u)^2 d\sigma  \bigg)^{1/2} \\
        & \lesssim M_{10\e}^\partial (\widetilde{N}(\nabla u))(Q).
    \end{aligned}
    \end{equation}
    These estimates together prove the lemma.
\end{proof}

The next proposition, which is our first ingredient from the $C^1$ domains, is a large-scale version for the local $L^p$ estimate of nontangential maximal function for the homogenized operator $\cL_0 = -\nabla\cdot (\widehat{A}\nabla)$; see Appendix \ref{Appendix-A} for a brief introduction of the homogenized operator.

\begin{proposition}\label{prop.L0.averageNT}
    Assume that $\Omega$ is a $C^1$ graph domain. Let $1\ge r > 100\e$.
    Let $u$ be a solution to
    \begin{equation}\label{eq.L0.T2r}
    \left\{
    \begin{aligned}
        & \cL_0(u) = 0 \quad \text{in } T_{2r}, \\
        & u = f \in W^{1,2}(\partial T_{2r}) \quad \text{on } \partial T_{2r}.
    \end{aligned}
    \right.
    \end{equation}
    Then for any $p\in [2,\infty)$,
    \begin{equation}\label{est.L0.nt}
     \begin{aligned}
        & \bigg( \fint_{I_r} |M_{\e}^\partial (\widetilde{N}(\nabla u))|^p d\sigma \bigg)^{1/p} \\
        & \lesssim \bigg( \fint_{I_{2r}} |M_\e^\partial(\nabla_{\tan} f)|^p d\sigma \bigg)^{1/p} + \bigg( \fint_{\partial T_{2r} \setminus I_{2r} } |\nabla_{\tan} f|^2 d\sigma \bigg)^{1/2}.
    \end{aligned}   
    \end{equation}
\end{proposition}

By a localization argument, the above proposition implies the following statement. If $\Omega$ is a bounded $C^1$ domain and $u$ is a solution to
    \begin{equation}\label{eq.L0.Omega}
    \left\{
    \begin{aligned}
        & \cL_0(u) = 0 \quad \text{in } \Omega, \\
        & u = f \quad \text{on } \partial \Omega,
    \end{aligned}
    \right.
    \end{equation}
then for any $p\in [2,\infty)$,
\begin{equation}\label{est.L0.ntinOmega}
    \| M_\e^\partial (\widetilde{N}(\nabla u)) \|_{L^p(\partial \Omega)} \lesssim \| M_\e^\partial (\nabla_{\tan} f) \|_{L^p(\partial \Omega)}.
\end{equation}
Note that the homogenized problem \eqref{eq.L0.Omega} has nothing to do with the parameter $\e$, while the average in \eqref{est.L0.ntinOmega} can be taken at any $\e$-scale. This property is essentially due to the multiscale nature of the real-variable argument. 

Combining Proposition \ref{prop.L0.averageNT} and Lemma \ref{lem.hatN<M10eNDu}, we obtain the following.

\begin{corollary}\label{coro.L0.hatN}
    Under the assumptions of Proposition \ref{prop.L0.averageNT}, we have
    \begin{equation}\label{est.L0.hatNnt}
    \begin{aligned}
        & \bigg( \fint_{I_{r}} | \widehat{N}_\e(\nabla u)|^p d\sigma \bigg)^{1/p} \\
        & \lesssim \bigg( \fint_{I_{2r}} |M_\e^\partial(\nabla_{\tan} f)|^p d\sigma \bigg)^{1/p} + \bigg( \fint_{\partial T_{2r} \setminus I_{2r} } |\nabla_{\tan} f|^2 d\sigma \bigg)^{1/2}.
    \end{aligned}   
    \end{equation}
\end{corollary}

\begin{proof}
    By Lemma \ref{lem.hatN<M10eNDu} and Proposition \ref{prop.Me=MKe}, we have
    \begin{equation}
    \begin{aligned}
        \bigg( \fint_{I_{r}} | \widehat{N}_\e(\nabla u)|^p d\sigma \bigg)^{1/p} & \lesssim \bigg( \fint_{I_{r}} | M_{10\e}^\partial (\nabla u)|^p d\sigma \bigg)^{1/p} \\
        & \lesssim \bigg( \fint_{I_{r+10\e }} | M_{\e}^\partial (\nabla u)|^p d\sigma \bigg)^{1/p}.
    \end{aligned}
    \end{equation}
    Since $r > 100\e$, we conclude \eqref{est.L0.hatNnt} from Proposition \ref{prop.L0.averageNT} (by slightly adjusting the size of $I_{2r}$).
\end{proof}


\begin{proof}[\textbf{Proof of Proposition \ref{prop.L0.averageNT}}]
    This is proved by the real-variable argument. By rescaling, it suffices to prove the case $r = 1$. Consider any $I_s(P) \subset I_1$ with $0 < s \le 1/10$. Let $\lambda = \fint_{I_{4s}(P)} f d\sigma$. Let $\phi$ be a cutoff function such that $\phi = 1$ in $B_{2s}(P)$, $\phi = 0$ in $\R^d\setminus B_{4s}(P)$ and $|\nabla \phi| \lesssim s^{-1}$. We decompose $u = v + w+\lambda$, where $v$ is the solution to
    \begin{equation}
    \left\{
    \begin{aligned}
        & \cL_0(v) = 0 \quad \text{in } T_{2r}, \\
        & v = (f - \lambda)\phi \in W^{1,2}(\partial T_{2r}) \quad \text{on } \partial T_{2r}.
    \end{aligned}
    \right.
    \end{equation}
    Since $T_2$ is a Lipschitz domain, we apply the $(R)_2$ estimate for the operator $\cL_0$ to get
    \begin{equation}
        \| \widetilde{N}(\nabla v) \|_{L^2(\partial T_2)} \lesssim \| \nabla_{\tan} ((f-\lambda) \phi ) \|_{L^2(I_{2})} \lesssim \| \nabla_{\tan} f \|_{L^2(I_{4s}(P))}.
    \end{equation}
    This implies
    \begin{equation}\label{est.localu0.v}
        \bigg( \fint_{I_{4s}(P)} |\widetilde{N}(\nabla v)|^2 d\sigma \bigg)^{1/2} \lesssim \bigg( \fint_{I_{4s}(P)} |\nabla_{\tan} f|^2 d\sigma \bigg)^{1/2}.
    \end{equation}

    Next, we consider $w$, which satisfies $\cL_0(w) = 0$ in $T_2$ and $w = 0$ in $I_{2s}(P)$. Then for any $Q \in I_s(P)$, we have
    \begin{equation}\label{est.ptws.NDw}
        \widetilde{N}(\nabla w)(Q) \lesssim \fint_{I_{2s}(P)} \widetilde{N}(\nabla w) d\sigma + \widetilde{N}^s(\nabla w)(Q).
    \end{equation}
    Since $I_{2s}(P)$ is a part of the $C^1$ boundary, then for any $q\in [2,\infty)$, by $(R)_q$ estimate and the local $W^{1,q}$ estimate \cite{Fabes1978, DPP07}, we have
    \begin{equation}
    \begin{aligned}
        \bigg( \fint_{I_s(P)} |\widetilde{N}^s(\nabla w)|^q d\sigma \bigg)^{1/q} & \lesssim \bigg( \fint_{T_{2s}(P)} |\nabla w|^2 \bigg)^{1/2} \\
        & \lesssim \bigg( \fint_{T_{2s}(P)} |\nabla v|^2 \bigg)^{1/2} + \bigg( \fint_{T_{2s}(P)} |\nabla u|^2 \bigg)^{1/2} \\
        & \lesssim \bigg( \fint_{I_{2s}(P)} |\widetilde{N} (\nabla v)|^2 \bigg)^{1/2} + \bigg( \fint_{I_{2s}(P)} |\widetilde{N}(\nabla u)|^2 \bigg)^{1/2} \\
        & \lesssim \bigg( \fint_{I_{4s}(P)} |\nabla_{\tan} f|^2 d\sigma \bigg)^{1/2} + \bigg( \fint_{I_{2s}(P)} |\widetilde{N}(\nabla u)|^2 \bigg)^{1/2},
    \end{aligned}
    \end{equation}
    where we have used \eqref{est.localu0.v} in the last inequality.
    Also, by the triangle inequality and \eqref{est.localu0.v}, we have
    \begin{equation}
        \fint_{I_{2s}(P)} \widetilde{N}(\nabla w) d\sigma \lesssim \bigg( \fint_{I_{4s}(P)} |\nabla_{\tan} f|^2 d\sigma \bigg)^{1/2} + \bigg( \fint_{I_{2s}(P)} |\widetilde{N}(\nabla u)|^2 \bigg)^{1/2}.
    \end{equation}
    Consequently, taking the $L^q$ average of \eqref{est.ptws.NDw} over $I_s(Q)$ and combining the last two inequalities, we have
    \begin{equation}
        \bigg( \fint_{I_{s}(P)} |\widetilde{N}(\nabla w)|^q d\sigma \bigg)^{1/q} \lesssim \bigg( \fint_{I_{4s}(P)} |\nabla_{\tan} f|^2 d\sigma \bigg)^{1/2} + \bigg( \fint_{I_{2s}(P)} |\widetilde{N}(\nabla u)|^2 \bigg)^{1/2}.
    \end{equation}

    Now, put $F = |\widetilde{N}(\nabla u)|, F_{I_{s}(P)} = |\widetilde{N}(\nabla v)|$ and $R_{I_s(P)} = |\widetilde{N}(\nabla w)|$. Then we have for any $0 < s < 1/10$, $F \le F_{I_{s}(P)} + R_{I_s(P)}$ and
    \begin{equation}
        \left\{ 
        \begin{aligned}
            & \bigg( \fint_{I_{4s}(P)} |F_{I_{s}(P)}|^2 d\sigma \bigg)^{1/2} \lesssim \bigg( \fint_{I_{4s}(P)} |\nabla_{\tan} f|^2 d\sigma \bigg)^{1/2},\\
            & \bigg( \fint_{I_{s}(P)} |R_{I_s(P)}|^q d\sigma \bigg)^{1/q} \lesssim \bigg( \fint_{I_{4s}(P)} |\nabla_{\tan} f|^2 d\sigma \bigg)^{1/2} + \bigg( \fint_{I_{2s}(P)} |F|^2 \bigg)^{1/2}.
        \end{aligned}
        \right.
    \end{equation}
    By Theorem \ref{thm.real-variable>t}, we derive
    \begin{equation}\label{est.MeF2Mef}
        \bigg( \fint_{I_1} |M_\e^\partial (F)|^p \bigg)^{1/p} \lesssim \bigg( \fint_{I_{3/2}} |M_\e^\partial (\nabla_{\tan} f)|^p d\sigma \bigg)^{1/p} + \bigg( \fint_{I_{3/2}} |M_\e^\partial (F)|^2 \bigg)^{1/2}.
    \end{equation}
    for $p\in [2,q)$ and any $0<\e \le 1$. But since $q$ can be taken arbitrarily large, $p$ can also be arbitrarily large. 

    Finally, by \eqref{est.Mef<f} (with $M_\e$ replaced by $M_\e^\partial$) and the $(R)_2$ estimate in the Lipschitz domain $T_2$, we have
    \begin{equation}
    \begin{aligned}
        \bigg( \fint_{I_{3/2}} |M_\e^\partial (F)|^2 \bigg)^{1/2} & \lesssim \bigg( \fint_{\partial T_2} |\widetilde{N}(\nabla u)|^2 \bigg)^{1/2} \\
        & \lesssim \bigg( \fint_{\partial T_2} |\nabla_{\tan} f|^2 \bigg)^{1/2} \\
        & \lesssim \bigg( \fint_{I_{2}} |M_\e^\partial(\nabla_{\tan} f)|^2 d\sigma \bigg)^{1/2} + \bigg( \fint_{\partial T_{2} \setminus I_{2} } |\nabla_{\tan} f|^2 d\sigma \bigg)^{1/2}.
    \end{aligned}  
    \end{equation}
    This together with \eqref{est.MeF2Mef} gives \eqref{est.L0.nt} for the case $r = 1$.
\end{proof}

Our second ingredient coming from the $C^1$ domains is the local large-scale $W^{1,p}$ estimate for the operator $\cL_\e$. This is a folklore result whose proof can be extracted from \cite[Chapter 4.3]{ShenBook}, \cite[Chapter 7]{AKM19}, etc. 

\begin{proposition}\label{prop.W1p.local}
    Let $\Omega$ be a $C^1$ graph domain.
    Let $u_\e$ be a weak solution of $\cL_\e(u_\e) = \nabla\cdot F$ in $T_{2r}$ and either $u_\e = 0$ or $\frac{\partial u_\e}{\partial \nu_\e} = 0$ on $I_{2r}$, where $0< r< 1$. Then for $2\le p < \infty$,
    \begin{equation}
        \bigg( \fint_{T_r} |M_\e(\nabla u_\e)|^p \bigg)^{1/p} \lesssim \bigg( \fint_{T_{2r}} |M_\e(F)|^p \bigg)^{1/p} + \bigg( \fint_{T_{2r}} |\nabla u_\e|^2 \bigg)^{1/2}.
    \end{equation}   
\end{proposition}


Now we are ready to prove Theorem \ref{thm.LocalRellich.Rp}.
By rescaling, without loss of generality, assume $r = 1 \ge 100 \e$. Let $u_\e$ be a solution of \eqref{eq.Rp.T2r}. Let $u_0$ be the corresponding homogenized solution,  i.e.,
    \begin{equation}\label{eq.L0Rp.T2r}
    \left\{
    \begin{aligned}
        & \cL_0(u_0) = 0 \quad \text{in } T_{2}, \\
        & u_0 = f \in W^{1,2}(\partial T_{2}) \quad \text{on } \partial T_{2},
    \end{aligned}
    \right.
    \end{equation}
    where $f=u_\e|_{\partial T_{2}}$.
Let $\eta_\e \in C_0^\infty(T_{2})$ be a cutoff function such that $\eta_\e = 1$ in $T_{2}^{6\e} = T_2 \setminus T_{2,6\e}$, $\eta_\e = 0$ in $T_{2,5\e}$ and $|\nabla \eta_\e| \lesssim \e^{-1}$.
Consider
\begin{equation}\label{def.we}
    w_\e = u_\e - u_0 - \e \chi_j(X/\e) K_\e (\partial_j u_0) \eta_\e.
\end{equation}
Then by a standard calculation (see \cite[Lemma 7.3]{Shen17}), we get
\begin{equation}\label{eq.we.inT2}
\left\{
\begin{aligned}
    & \cL_\e(w_\e) = \nabla\cdot F_\e \quad \text{in } T_2,\\
    & w_\e = 0 \quad \text{ on } \partial T_2,
\end{aligned}
    \right.
\end{equation} 
where $F_\e = (F_{\e,i})$ is given by
\begin{equation}\label{def.Fe}
\begin{aligned}
    F_{\e,i} & = (a_{ij}(X/\e) - \widehat{a}_{ij})(\partial_j u_0 - K_\e(\partial_j u_0) \eta_\e ) - \e \phi_{kij}(X/\e) \partial_k (K_\e(\partial_j u_0) \eta_\e )  \\
    & \quad  + \e a_{ij}(X/\e) \chi_k(X/\e) \partial_j ( K_\e (\partial_k u_0 ) \eta_\e).
\end{aligned}
\end{equation}

We state a lemma for the $L^2$ estimate of $F_\e$, which is essentially contained in the proof of \cite[Theorem 2.6]{Shen17} (also see \cite[Theorem 3.2.3]{ShenBook}).

\begin{lemma}\label{lem.weL2}
Under the above assumptions and for $F_\e$ given by \eqref{def.Fe}, we have
\begin{equation}\label{est.MeFe1}
    \| F_\e \|_{L^2(T_2)} \lesssim \e^{1/2} \| \nabla_{\tan} f \|_{L^2(\partial T_2)}.
\end{equation}
\end{lemma}

The next lemma provides the large-scale $L^p$ estimate of $F_\e$ for $p>2$.

\begin{lemma}\label{lem.weLp}
Under the above assumptions and for $F_\e$ given by \eqref{def.Fe}, we have, for any $2<p<\infty$,
\begin{equation}\label{est.MeFe}
    \| M_\e(F_\e) \|_{L^p(T_{3/4})} \lesssim \e^{1/p} \big( \| M_\e(\nabla_{\tan} f) \|_{L^p(I_2)} + \| \nabla_{\tan} f \|_{L^2(\partial T_2 \setminus I_2)} \big).
\end{equation}
\end{lemma}

\begin{proof}
According to the three terms on the right-hand side of \eqref{def.Fe}, we write $F_\e = F_\e^{1} + F_\e^2 + F_\e^3$.
We begin with the estimate of $F_\e^1$.
By the triangle inequality, we have
\begin{equation}\label{def.Fe1}
    |F_\e^1| \lesssim |(\nabla u_0 - K_\e(\nabla u_0)) \eta_\e| + |\nabla u_0 (1-\eta_\e)|.
\end{equation}
By a property of $K_\e$ in Proposition \ref{prop.f-Kef} and using the support of $\eta_\e$, we have
\begin{equation}\label{est.Du0-KeDu0}
    \| (\nabla u_0 - K_\e(\nabla u_0)) \eta_\e \|_{L^p(T_{4/5})} \lesssim \e \| \nabla^2 u_0 \|_{L^p(T_{1}^{4\e})}.
\end{equation}
Now, we estimate $\nabla^2 u_0(X)$ for $X\in T_{1}^{4\e}$. 
In fact, if $X \in  T_{1}^{4\e}$ for some $\proj(X) = Q\in I_1$, we have, by the interior estimate for the $\widehat{A}$-harmonic function $u_0$,
\begin{equation}
    |\nabla^2 u_0(X)| \lesssim \delta(X)^{-1} \bigg( \fint_{B(X,\delta(X)/2)} |\nabla u_0|^2 \bigg)^{1/2} \lesssim \delta(X)^{-1} \widehat{N}_\e(\nabla u_0)(Q).
\end{equation}
This implies that for any $4\e < t < 1$
\begin{equation}
    \int_{ T_1 \cap \{\delta(X) = t\} } |\nabla^2 u_0(X)|^p d\sigma(X) \lesssim t^{-p} \int_{I_1} |\widehat{N}_\e(\nabla u_0)(Q)|^p d\sigma(Q).
\end{equation}
By the co-area formula, we have
\begin{equation}\label{est.D2u0.coarea}
\begin{aligned}
    \int_{T_1^{4\e}} |\nabla^2 u_0(X)|^p dX 
    & \lesssim \int_{4\e}^{1} dt \int_{T_1 \cap \{ \delta(X) = t\} } |\nabla^2 u_0(X)|^p d\sigma(X) \\
    & \lesssim \int_{4\e}^{1} t^{-p} dt \int_{I_1 } |\widehat{N}_\e(\nabla u_0)(Q)|^p d\sigma(Q) \\
    & \lesssim \e^{1-p} \bigg\{ \int_{I_2} |M_\e^\partial(\nabla_{\tan} f)|^p d\sigma + \bigg( \int_{\partial T_{2} \setminus I_{2} } |\nabla_{\tan} f|^2 d\sigma \bigg)^{p/2} \bigg\},
\end{aligned}
\end{equation}
where we have used Corollary \ref{coro.L0.hatN} in the last inequality.
Combining this with \eqref{est.Du0-KeDu0}, we obtain
\begin{equation}\label{est.Du0-KDu0.final}
    \| (\nabla u_0 - K_\e(\nabla u_0)) \eta_\e \|_{L^p(T_{4/5})}  \lesssim \e^{1/p} \big( \| M_\e^\partial(\nabla_{\tan} f) \|_{L^p(I_2)} + \| \nabla_{\tan} f \|_{L^2(\partial T_2 \setminus I_2)} \big).
\end{equation}

Next, we estimate $M_\e(\nabla u_0(1-\eta_\e))$. Note that $1-\eta_\e$ is supported in 
$T_{2,4\e}$. Thus,
\begin{equation}\label{est.Du0.elayer}
\| M_\e(\nabla u_0(1-\eta_\e)) \|_{L^p(T_{3/4})} \lesssim \| M_\e(\nabla u_0) \|_{L^p(T_{1} \cap \Omega_{7\e})}.
\end{equation}
By virtue of the second term in \eqref{def.hatN} for the definition of $\widehat{N}_\e(\nabla u_0)$, we see
\begin{equation}\label{est.MeDu0Lp}
\begin{aligned}
    \| M_\e(\nabla u_0) \|_{L^p(T_{1} \cap \Omega_{7\e})} & \lesssim \| M_{10\e}^\star(\nabla u_0) \|_{L^p(I_1)} \lesssim \e^{1/p} \| \widehat{N}_\e(\nabla u_0) \|_{L^p(I_1)} \\
    & \lesssim \e^{1/p} \big( \| M_\e^\partial(\nabla_{\tan} f) \|_{L^p(I_2)} + \| \nabla_{\tan} f \|_{L^2(\partial T_2 \setminus I_2)} \big).
\end{aligned}
\end{equation}
Taking this into \eqref{est.Du0.elayer}, and combining it with \eqref{def.Fe1} and \eqref{est.Du0-KDu0.final}, we arrive at
\begin{equation}\label{est.Fe1}
\begin{aligned}
    \| M_\e(F_\e^1) \|_{L^p(T_{3/4})} & \lesssim \| (\nabla u_0 - K_\e(\nabla u_0)) \eta_\e \|_{L^p(T_{4/5})} \| + \| M_\e(\nabla u_0(1-\eta_\e)) \|_{L^p(T_{3/4})} \\
    & \lesssim \e^{1/p} \big( \| M_\e^\partial(\nabla_{\tan} f) \|_{L^p(I_2)} + \| \nabla_{\tan} f \|_{L^2(\partial T_2 \setminus I_2)} \big).
\end{aligned}
\end{equation}

Next, we estimate
\begin{equation}\label{def.Fe2}
    |M_\e(F_\e^2)| \lesssim \e |M_\e(\phi(X/\e) K_\e(\nabla^2 u_0) \eta_\e)| + \e |M_\e(\nabla \eta_\e \phi(X/\e) K_\e(\nabla u_0))|.
\end{equation}
The estimates of these two terms are similar to $F_\e^1$. For the first term, note that $\eta_\e$ is supported in $T_2^{5\e}$ and $\phi(Y)$ is a local $L^2$ periodic function. Thus, by Proposition \ref{prop.Me(gKef)},
we have
\begin{equation}\label{est.phi.KeDDu0}
\begin{aligned}
    \e \| M_\e(\phi(X/\e) K_\e(\nabla^2 u_0) \eta_\e) \|_{L^p(T_{3/4})} & \lesssim \e \| M_\e(\phi(X/\e) K_\e(\nabla^2 u_0)) \|_{L^p(T_{4/5}^{4\e})} \\
    & \lesssim \e \| \nabla^2 u_0 \|_{L^p(T_1^{2\e})} \\
    & \lesssim \e^{1/p} \big( \| M_\e^\partial(\nabla_{\tan} f) \|_{L^p(I_2)} + \| \nabla_{\tan} f \|_{L^2(\partial T_2 \setminus I_2)} \big),
\end{aligned}
\end{equation}
where the last inequality follows from  \eqref{est.D2u0.coarea} as before. For the second term in  \eqref{def.Fe2}, using the fact that $\nabla \eta_\e$ is supported in $T_{2,6\e} \setminus T_{2, 5\e}$ and the interior estimate $|\nabla u_0(X)| \lesssim M_\e(\nabla u_0)(X)$ for $X\in T_2^\e$, as well as Proposition \ref{prop.Me(gKef)}, we have
\begin{equation}\label{est.phi.KeDu0}
\begin{aligned}
    \e \| M_\e( \nabla \eta_\e \phi(X/\e) K_\e(\nabla u_0) )\|_{L^p(T_{3/4})} & \lesssim \| M_\e( \phi(X/\e) K_\e(\nabla u_0)) \|_{L^p(T_{4/5} \cap (T_{2,7\e} \setminus T_{2,4\e}))} \\
    & \lesssim \| \nabla u_0 \|_{L^p(T_{1} \cap (T_{2,9\e} \setminus T_{2,2\e}) )} \\
    & \lesssim \| M_\e(\nabla u_0) \|_{L^p(T_1 \cap T_{2,9\e})}  \\
    & \lesssim \e^{1/p} \big( \| M_\e^\partial(\nabla_{\tan} f) \|_{L^p(I_2)} + \| \nabla_{\tan} f \|_{L^2(\partial T_2 \setminus I_2)} \big),
\end{aligned}
\end{equation}
where we have used an estimate similar to \eqref{est.MeDu0Lp} in the last inequality. Consequently, we have proved
\begin{equation}\label{est.Fe2}
    \| M_\e( F_\e^2) \|_{L^p(T_{3/4})} \lesssim \e^{1/p} \big( \| M_\e^\partial(\nabla_{\tan} f) \|_{L^p(I_2)} + \| \nabla_{\tan} f \|_{L^2(\partial T_2 \setminus I_2)} \big).
\end{equation}

Finally, the estimate of $M_\e(F_\e^3)$ is the same as in $M_\e(F_\e^2)$ and we have
\begin{equation}\label{est.Fe3}
    \| M_\e(F_\e^3) \|_{L^p(T_{3/4})} \lesssim \e \big( \| M_\e^\partial(\nabla_{\tan} f) \|_{L^p(I_2)} + \| \nabla_{\tan} f \|_{L^2(\partial T_2 \setminus I_2)} \big).
\end{equation}
Thus, the desired estimate \eqref{est.MeFe} follows from \eqref{est.Fe1}, \eqref{est.Fe2} and \eqref{est.Fe3}.
\end{proof}

\begin{proposition}\label{prop.Dwe.Lp}
    Let $\Omega$ be a $C^1$ graph domain and $w_\e$ be given by \eqref{def.we}. Then for any $p\in [2,\infty)$,
    \begin{equation}
        \| M_\e(\nabla w_\e) \|_{L^p(T_{1/2})} \lesssim \e^{1/p} \big( \| M_\e^\partial(\nabla_{\tan} f) \|_{L^p(I_2)} + \| \nabla_{\tan} f \|_{L^2(\partial T_2 \setminus I_2)} \big).
    \end{equation}
\end{proposition}

\begin{proof}
    Applying Proposition \ref{prop.W1p.local} to the equation \eqref{eq.we.inT2} for $w_\e$, we have
    \begin{equation}\label{est.MeWe.T1/2}
        \| M_\e(\nabla w_\e) \|_{L^p(T_{1/2})} \lesssim \| M_\e(F_\e) \|_{L^p(T_{3/4})} + \| \nabla w_\e \|_{L^2(T_1)}.
    \end{equation}
    By the energy estimate for \eqref{eq.we.inT2} and Lemma \ref{lem.weL2},
    \begin{equation}
        \| \nabla w_\e \|_{L^2(T_2)} \lesssim  \| F_\e \|_{L^2(T_2)} \lesssim \e^{1/2} \| \nabla_{\tan} f \|_{L^2(\partial T_2)}.
    \end{equation}
 Inserting this into \eqref{est.MeWe.T1/2} and using Lemma \ref{lem.weLp}, we obtain the desired estimate.
\end{proof}

\begin{proof}[\textbf{Proof of Theorem \ref{thm.LocalRellich.Rp}}]
    It suffices to consider $\e < t < r = 1$. First of all, by Proposition \ref{prop.Dwe.Lp} and \eqref{def.we}, we have
    \begin{equation}
    \begin{aligned}
        &\| M_\e(\nabla u_\e) \|_{L^p(T_{1/2} \cap \Omega_t)} \\
        & \lesssim \e^{1/p} \big( \| M_\e^\partial(\nabla_{\tan} f) \|_{L^p(I_2)} + \| \nabla_{\tan} f \|_{L^2(\partial T_2 \setminus I_2)} \big) \\
        & \qquad + \| M_\e(\nabla u_0) \|_{L^p(T_{1/2} \cap \Omega_t)} +  \| M_\e(\nabla ( \e \chi(X/\e) K_\e(\nabla u_0) \eta_\e ) )\|_{L^p(T_{1/2}\cap \Omega_t)}.
        \end{aligned}
    \end{equation}
    By mimicking the estimate \eqref{est.MeDu0Lp}, we have
    \begin{equation}
        \| M_\e(\nabla u_0) \|_{L^p(T_{1/2} \cap \Omega_t)} \lesssim t^{1/p} \big( \| M_\e^\partial(\nabla_{\tan} f) \|_{L^p(I_2)} + \| \nabla_{\tan} f \|_{L^2(\partial T_2 \setminus I_2)} \big).
    \end{equation}
    By mimicking the estimate for $F_\e^2$ in the proof of Lemma \ref{lem.weLp} (using Proposition \ref{prop.Me(gKef)}), we can show
    \begin{equation}
    \begin{aligned}
        & \| M_\e(\nabla ( \e \chi(X/\e) K_\e(\nabla u_0) \eta_\e ) )\|_{L^p(T_{1/2}\cap \Omega_t)} \\
        & \qquad \lesssim t^{1/p} \big( \| M_\e^\partial(\nabla_{\tan} f) \|_{L^p(I_2)} + \| \nabla_{\tan} f \|_{L^2(\partial T_2 \setminus I_2)} \big).
    \end{aligned}
    \end{equation}
    Consequently, we arrive at
    \begin{equation}\label{est.MeDue.Ot}
        \| M_\e(\nabla u_\e) \|_{L^p(T_{1/2} \cap \Omega_t)} \lesssim t^{1/p} \big( \| M_\e^\partial(\nabla_{\tan} f) \|_{L^p(I_2)} + \| \nabla_{\tan} f \|_{L^2(\partial T_2 \setminus I_2)} \big).
    \end{equation}

    Now, we apply \eqref{est.MeDue.Ot} in domains $T_{2s}$ as $s$ varies in $[1/2,1]$. It follows that
    \begin{equation}\label{est.MeDue.Ts}
        \| M_\e(\nabla u_\e) \|_{L^p(T_{s/2} \cap \Omega_t)} \lesssim t^{1/p} \big( \| M_\e^\partial(\nabla_{\tan} f) \|_{L^p(I_{2s})} + \| \nabla_{\tan} u_\e \|_{L^2(\partial T_{2s} \setminus I_{2s})} \big).
    \end{equation}
    By the co-area formula, we have
    \begin{equation}
        \int_{1/2}^1 \| \nabla u_\e \|_{L^2(\partial T_s \setminus I_{2s})}^2 ds \simeq \int_{T_2\setminus T_1} |\nabla u_\e|^2.
    \end{equation}
    Consequently, integrating \eqref{est.MeDue.Ts} over $s \in [1/2,1]$ yields
    \begin{equation}
        \| M_\e(\nabla u_\e) \|_{L^p(T_{1/4} \cap \Omega_t)} \lesssim t^{1/p} \big( \| M_\e^\partial(\nabla_{\tan} f) \|_{L^p(I_{2})} + \| \nabla u_\e \|_{L^2(T_{2} )} \big).
    \end{equation}
    This proves \eqref{est.localRellich.Rp} with $r = 1$. The general cases follow by rescaling.
    \end{proof}

\subsection{Global $L^p$ Rellich estimates}
In this subsection, we deduce the global large-scale $L^p$ Rellich estimates from the local estimates established in the previous subsection.
\begin{theorem}\label{thm.C1Rellich.Rp}
Let $\Omega$ be a bounded $C^1$ domain.
Let $u_\e$ be a solution of \eqref{Rp} with $f\in W^{1,2}(\partial \Omega)$. Then for any $2<p<\infty$ and
 any $\e \le t < {\rm diam}(\Omega)$,
    \begin{equation}\label{est.GlobalLpRellich.C1}
        \bigg( \frac{1}{t} \int_{\Omega_t} |M_\e(\nabla u_\e)|^p \bigg)^{1/p} \lesssim \| M_\e^\partial(\nabla_{\tan} f) \|_{L^p(\partial \Omega)}.
    \end{equation}
\end{theorem}

\begin{proof}
    The local $L^p$ Rellich estimates in Theorem \ref{thm.LocalRellich.Rp} over a $C^1$ graph domain can be turned into a version for a bounded  $C^1$ domain $\Omega$, namely, for any $P\in \partial \Omega$ and $\e \le r \le r_0$ (where $r_0$ is a constant depending only on $\Omega$) such that for any $\e \le t \le r$,
    \begin{equation}\label{est.patch-global}
    \begin{aligned}
        & \bigg( \frac{1}{t} \int_{\Omega_t \cap D_{r/2}(P)} |M_\e(\nabla u_\e)|^p \bigg)^{1/p} \\
        & \qquad \lesssim \bigg( \int_{\Delta_{2r}(P)} |M_\e^\partial (\nabla_{\tan} f)|^p \bigg)^{1/p} + r^{\frac{d-1}{p} - \frac{d}{2}} \bigg( \int_{D_{2r}(P)} |\nabla u_\e|^2 \bigg)^{1/2}.
    \end{aligned}
    \end{equation}
    Fix $r = r_0$ and let $\{ P_i: i = 1,2,\cdots, N = N(\Omega) \}$ be a sequence of points on $\partial \Omega$ such that
    \begin{equation}
        \Omega_t = \bigcup_{i=1}^{N} ( \Omega_t \cap D_{r_0/2}(P_i) ),
    \end{equation}
    and $\{ D_{2r_0}(P_i) \}$ have finite overlaps. Hence, applying \eqref{est.patch-global} to each $D_{2r_0}(P_i)$ and summing over $i$, we obtain
    \begin{equation}
        \bigg( \frac{1}{t} \int_{\Omega_t} |M_\e(\nabla u_\e)|^p \bigg)^{1/p} \lesssim \| M_\e^\partial (\nabla_{\tan} f) \|_{L^p(\partial \Omega)} + \| \nabla u_\e \|_{L^2(\Omega)}.
    \end{equation}
    Finally, the energy estimate implies $\| \nabla u_\e \|_{L^2(\Omega)} \lesssim \| M_\e^\partial (\nabla_{\tan} f) \|_{L^p(\partial \Omega)}$. Hence, the last displayed estimate gives \eqref{est.GlobalLpRellich.C1} for $\e \le t\le r_0$, while the remaining case for $r_0 < t< \diam(\Omega)$ follows from the case $t = r_0$ and the interior large-scale $W^{1,p}$ estimate.
\end{proof}

Similarly, we also have the global large-scale $L^p$ Rellich estimate for the Neumann problem.
\begin{theorem}\label{thm.C1Rellich.Np}
Let $\Omega$ be a bounded $C^1$ domain.
Let $u_\e$ be a solution of \eqref{Np} with $g\in L^2(\partial \Omega)$. Then for any $2<p<\infty$ and
 any $\e \le t < {\rm diam}(\Omega)$,
    \begin{equation}
        \bigg( \frac{1}{t} \int_{\Omega_t} |M_\e(\nabla u_\e)|^p \bigg)^{1/p} \lesssim \| M_\e^\partial(g) \|_{L^p(\partial \Omega)}.
    \end{equation}
\end{theorem}

\subsection{Dirichlet problem}
As in the case of Lipschitz domains, the estimate \eqref{est.Dp} in Theorem \ref{thm.Dp} for $C^1$ domains follows from the reverse H\"{o}lder inequality of the large-scale $\cL_\e$-harmonic measure $\overline{\omega}_\e^X$.
\begin{theorem}\label{thm.C1reverse}
    Let $\Omega$ be a bounded $C^1$ domain. Then for any $p\in (2,\infty)$, we have $\overline{\omega}_\e^X \in B_p(d\sigma)$ uniformly in $\e$;  i.e., for any $r>0, P\in \partial \Omega, \Delta_r = \Delta_r(P)$ and $X\in \Omega\setminus B_{10r}(P)$
    \begin{equation}
        \bigg( \fint_{\Delta_r} \overline{\omega}_\e^X(Q)^p d\sigma(Q) \bigg)^{1/p} \lesssim \fint_{\Delta_{r}} \overline{\omega}_\e^X(Q) d\sigma(Q).
    \end{equation}
\end{theorem}

\begin{proof}
    Similar to the proof of Theorem \ref{thm.B2}, we begin with
    \begin{equation}
        \overline{\omega}^X_\e(Q) \lesssim M_\e(\nabla G_\e(X,\cdot))(Q).
    \end{equation}
    Taking $p$th power of this inequality and integrating $Q$ over $\Delta_r(P)$ for $r\gtrsim \e$, we get
    \begin{equation}\label{est.HM.C1Lp}
    \begin{aligned}
        \int_{\Delta_r} \overline{\omega}_\e^X(Q)^p d\sigma(Q) & \lesssim \e^{-1} \int_{\Omega_{2\e} \cap B_{1.5r}} |M_\e(\nabla G_\e(X,\cdot))(Y)|^p dY \\
        & \lesssim r^{-1} \int_{D_{6r}} |M_\e(\nabla G_\e(X,\cdot))(Y)|^p dY \\
        & \lesssim r^{-1}|D_{7r}| \bigg( \fint_{D_{7r}} |M_\e(\nabla G_\e(X,\cdot))(Y)|^2 dY \bigg)^{p/2} \\
        & \lesssim |\Delta_r| \bigg( \fint_{D_{8r}} |\nabla G_\e(X,Y)|^2 dY \bigg)^{p/2} \\
        & \lesssim |\Delta_r| r^{-p} \bigg( \fint_{D_{9r}} |G_\e(X,Y)|^2 dY \bigg)^{p/2} \\
        & \lesssim |\Delta_r| r^{-p} |G_\e(X, A_{9r}(P))|^p,
    \end{aligned}
    \end{equation}
    where we have used Theorem \ref{thm.LocalRellich.Rp} (with $f = 0$) in the second inequality, Proposition \ref{prop.W1p.local} (with $F = 0$) in the third, Proposition \ref{prop.A1} in the fourth, Caccioppoli inequality in the fifth, and Proposition \ref{prop.hmBasics} \ref{item.G-L2} in the last inequality.
    
     Finally, using Proposition \ref{prop.hmBasics} \ref{item.w=G} and \ref{item.doubling},
\begin{equation}
    \frac{G_\e(X,A_{9r}(P))}{r} \simeq \frac{\omega_\e^X(\Delta_{9r})}{|\Delta_{9r}|} \simeq \frac{\omega_\e^X(\Delta_{r})}{|\Delta_{r}|} \simeq \fint_{\Delta_r} \overline{\omega}_\e^X d\sigma,
\end{equation}
    we obtain the desired estimate from \eqref{est.HM.C1Lp}.
\end{proof}

\begin{proof}[\textbf{Proof of Theorem \ref{thm.Dp} for $C^1$ domains}]
   With the above reverse H\"{o}lder inequality at our disposal, the estimate for any $1<p<2$ in \eqref{est.Dp} follows identically from the proof of Theorem \ref{thm.Dp} for Lipschitz domains; also see Remark \ref{rmk.Bp-Dp'}. 
\end{proof}

\subsection{Regularity and Neumann problems}

We first derive a local estimate of $\widetilde{N}_\e^r(\nabla u_\e)$  for the regularity and Neumann problems in a $C^1$ graph domain for any $p\in (2,\infty)$. 

\begin{lemma}\label{lem.C1Rp}
    Let $\Omega$ be a $C^1$ graph domain. Let $\cL_\e(u_\e) = 0$ in $T_{10r}$, where $\e< r\le 1$. Then
    \begin{equation}
    \begin{aligned}
    \int_{I_r} |\widetilde{N}_\e^r(\nabla u_\e)|^p d\sigma \lesssim \frac{1}{\e} \int_{T_{10r} \cap \Omega_{20\e}} |M_\e(\nabla u_\e)|^p + \frac{1}{r} \int_{T_{10r}} |M_\e(\nabla u_\e)|^p.
    \end{aligned}
\end{equation}
\end{lemma}
\begin{proof}
Recalling Lemma \ref{lem.NT.ptws}, we have
\begin{equation}
    \begin{aligned}
        \widetilde{N}_\e^r(\nabla u_\e)(Q) &\lesssim \M_{I_{3r}} (\widetilde{N}_\e^{3r}(\Q_\e(u_\e)))(Q)  + \M_{I_{3r}} (V_\e)(Q),
    \end{aligned}    
\end{equation}
where $Q = (y',\phi(y')) \in I_{3r}$, $V_\e$ and $\M_{I_{3r}}$ are defined in \eqref{def.VeP} and \eqref{def.MI3r}, respectively.

For our application here, we need to slightly modify the estimate, by examining the proof of Lemma \ref{lem.NT.ptws}, as
\begin{equation}\label{est.NeDue.loc-e}
    \widetilde{N}_\e^r(\nabla u_\e)(Q) \lesssim \M_{I_{3r}} (\widetilde{N}_\e^{3r}(\Q_\e(u_\e)))(Q)  + \M_{I_{3r}}^\e (V_\e)(Q),
\end{equation}
where
\begin{equation}
    \M_{I_{3r}}^\e f(Q) = \sup \bigg\{ \fint_{I_s(P)} |f| d\sigma: Q\in I_s(P) \subset I_{3r} \text{ and } s \ge \e \bigg\}.
\end{equation}
For this large-scale Hardy-Littlewood maximal function, we have
\begin{equation}
    \int_{I_r} |\M_{I_{3r}}^\e f|^p d\sigma \lesssim \int_{I_{3r}} |M_\e^\partial (f)|^p d\sigma,
\end{equation}
for any $p\in (1,\infty)$. This can be proved in a similar manner as Proposition \ref{prop.LS.HLM}, replacing the volume average $M_\e$ in $\R^d$ by the surface average $M_\e^\partial$ on $I_{3r}$.
Thus, from \eqref{est.NeDue.loc-e},
\begin{equation}\label{est.local-Rp-Ir}
\begin{aligned}
    \int_{I_r} |\widetilde{N}_\e^r(\nabla u_\e)|^p(Q) d\sigma(Q) \lesssim \int_{I_{3r}} |\widetilde{N}_\e^{3r} (\Q_\e(u_\e)) |^p d\sigma + \int_{I_{3r}} |M_\e^\partial (V_\e)|^p d\sigma .
\end{aligned}
\end{equation}
We need to estimate the two terms on the right-hand side.

By the $(D)_p$ estimate in Theorem \ref{thm.Dp} with $p>2$ for $\Q_\e(u_\e)$ in $T_{t}$ for $t\in (9r,10r-5\e)$ and $r \ge 10\e$, we have
\begin{equation}
    \int_{I_{3r}} |\widetilde{N}_\e^{3r} (\Q_\e(u_\e)) |^p d\sigma \lesssim \int_{\partial T_t} |N_\e(\Q_\e(u_\e))|^p d\sigma \lesssim \int_{\partial T_{t}} |M_{4\e}(\Q_\e(u_\e))|^p d\sigma.
\end{equation}
Now, integrating over $t \in (9r,10r-5\e)$, we have
\begin{equation}
\begin{aligned}
    & \int_{I_{3r}} |\widetilde{N}_\e^{3r} (\Q_\e(u_\e)) |^p d\sigma \\
    & \qquad \lesssim \int_{I_{10r-5\e}} |M_{4\e}(\Q_\e(u_\e))|^p d\sigma + \frac{1}{r} \int_{T_{10r-5\e}} |M_{4\e}(\Q_\e(u_\e))|^p dX.
    \end{aligned}
\end{equation}
Recall that
\begin{equation}
    \Q_\e(u_\e)(X) = \frac{1}{\e} \int_0^\e \frac{\partial u_\e}{\partial x_d}(x', x_d + t) dt.
\end{equation}
Thus, by Fubini's Theorem and Proposition \ref{prop.Me=MKe}, we have
\begin{equation}
    \int_{I_{10r-5\e}} |M_{4\e}(\Q_\e(u_\e))|^p d\sigma \lesssim \frac{1}{\e} \int_{T_{10r} \cap \Omega_{10\e}} |M_\e(\nabla u_\e)|^p
\end{equation}
and
\begin{equation}
    \frac{1}{r} \int_{T_{10r-5\e}} |M_{4\e}(\Q_\e(u_\e))|^p dX \lesssim \frac{1}{r} \int_{T_{10r}} |M_\e(\nabla u_\e)|^p.
\end{equation}
Hence, we obtain
\begin{equation}\label{est.NQe-1}
    \int_{I_{3r}} |\widetilde{N}_\e^{3r} (\Q_\e(u_\e)) |^p d\sigma \lesssim \frac{1}{\e} \int_{T_{10r} \cap \Omega_{10\e}} |M_\e(\nabla u_\e)|^p + \frac{1}{r} \int_{T_{10r}} |M_\e(\nabla u_\e)|^p.
\end{equation}

On the other hand, by Proposition \ref{prop.A1}, 
\begin{equation}
    M_\e^\partial (V_\e)(Q) \lesssim \fint_{I_\e(Q) \times (0,11\e)} M_\e(\nabla u_\e)(X) dX
\end{equation}
By Fubini's Theorem,
\begin{equation}\label{est.MeVe-1}
\begin{aligned}
    \int_{I_{3r}} |M_\e^\partial (V_\e)|^p d\sigma & \lesssim \int_{I_{3r}} \fint_{I_{\e}(Q) \times (0,11\e)} |M_\e(\nabla u_\e)(X)|^p dX d\sigma \\
    & \lesssim \frac{1}{\e} \int_{I_{3r+\e} \times (0,11\e)}  |M_\e(\nabla u_\e)(X)|^p dX \\
    & \lesssim \frac{1}{\e} \int_{T_{4r} \cap \Omega_{20\e}} |M_\e(\nabla u_\e)|^p.
\end{aligned}    
\end{equation}
Hence, combining \eqref{est.NQe-1}, \eqref{est.MeVe-1} with \eqref{est.local-Rp-Ir}, we obtain the desired estimate.
\end{proof}


\begin{proof}[\textbf{Proof of Theorems \ref{thm.Rp} and \ref{thm.Np} for $C^1$ domains}]
    Since the case $p\in (1,2+\delta)$ has been proved for general Lipschitz domains, it suffices to assume $p\in (2,\infty)$. First, Lemma \ref{lem.C1Rp} can be transferred into a version in a bounded $C^1$ domain, i.e.,
    \begin{equation}
    \int_{D_r} |\widetilde{N}_\e^r(\nabla u_\e)|^p d\sigma \lesssim \frac{1}{\e} \int_{D_{10r} \cap \Omega_{20\e}} |M_\e(\nabla u_\e)|^p + \frac{1}{r} \int_{D_{10r}} |M_\e(\nabla u_\e)|^p.
\end{equation}
    Using this, by \eqref{est.ReduceToLayer} and a covering argument as in the proof of Theorem \ref{thm.C1Rellich.Rp}, we have
    \begin{equation}\label{est.MainInProof}
        \begin{aligned}
            \| \widetilde{N}_\e(\nabla u_\e) \|_{L^p(\partial \Omega)} 
           & \lesssim \| \widetilde{N}_\e^{r_0}(\nabla u_\e) \|_{L^p(\partial \Omega)}  + \| \nabla u_\e \|_{L^2(\Omega)} \\
           & \lesssim \bigg( \frac{1}{\e} \int_{\Omega_{20\e}} |M_\e(\nabla u_\e)|^p d\sigma \bigg)^{1/p} + \| M_\e(\nabla u_\e) \|_{L^p(\Omega)}.
        \end{aligned}
    \end{equation}
    Note that this estimate holds for both regularity problem \eqref{Rp} and  Neumann problem \eqref{Np}. Hence, if $u_\e$ is the weak solution of \eqref{Rp}, by the large-scale $L^p$ Rellich estimate in Theorem \ref{thm.C1Rellich.Rp} in $C^1$ domains (with $t = 20\e$), as well as the energy estimate, we conclude \eqref{est.Rp} for $p>2$. Similarly, the estimate \eqref{est.Np} for $p>2$ follows from \eqref{est.MainInProof}, Theorem \ref{thm.C1Rellich.Np} and the energy estimate.
\end{proof}

\begin{remark}\label{rmk.convex}
    All the proofs in this section remain valid if the domain is convex (or even strongly quasiconvex in the sense of \cite[Definition 1.8]{FLZ26}) instead of $C^1$. This is because, as we mentioned at the beginning of this section, the only ingredients we need from the geometry of domains are the $L^p$ estimates of the classical nontangential maximal functions for $\cL_0$ and the large-scale boundary $W^{1,p}$ estimate for $\cL_\e$ for any $p\in (1,\infty)$. These are both valid in convex domains; see \cite{Mazya2009,Geng2010} for the gradient estimate of Neumann problem, while the gradient estimate of Dirichlet problem is classical). Thus Theorems \ref{thm.Dp}-\ref{thm.Np} hold for convex domains for any $p\in (1,\infty)$. Furthermore, together with the small-scale estimate of the nontangential maximal functions (known in convex domains at least for regularity and hence Dirichlet problems), we can even obtain the full-scale estimates in Corollary \ref{coro.fullscale} in convex domains for the full range of $p\in (1,\infty)$.
\end{remark}



\section{Full-scale estimates}\label{sec.full-scale}

In the previous sections, we dealt with bounded measurable coefficients with periodic structure and established various large-scale estimates. A natural question is that if we assume some regularity on the coefficients at small scale (i.e., in each periodic cell), can we combine the large-scale estimate established previously with the small-scale estimates to obtain the classical full-scale estimates? In this section, we will answer this question positively by proving the estimates of the full-scale nontangential maximal functions in Corollary \ref{coro.fullscale}, which also yield the classical Rellich estimates and the reverse H\"{o}lder inequalities of $\cL_\e$-harmonic measures.

\begin{proof}[\textbf{Proof of Corollary \ref{coro.fullscale}}]
    We only consider the regularity problems in Lipschitz or $C^1$ domains. The estimates for Neumann problems are similar. The Dirichlet problem follows by a classical duality that (full-scale) $(R)_p$ directly implies (full-scale) $(D)_{p'}$; see \cite[Theorem 5.4]{KP93}.

    We first recall the local (small-scale) regularity estimate in Lipschitz domains. Let $A(X)$ be H\"{o}lder continuous. Let $u$ be a weak solution of 
    \begin{equation}
    \left\{
    \begin{aligned}
        & -\nabla\cdot (A \nabla u) = 0 \quad \text{in } T_{1}, \\
        & u = f \in C^1(I_1) \quad \text{on }  I_1.
    \end{aligned}
    \right.
    \end{equation}
    Then by $(R)_p$ estimates for elliptic operators with $C^\alpha$ coefficients and a localization technique (\cite[Theorem 8.1]{KS11-2} for the case $p = 2$), we have, for any $1<p<2+\delta$,
    \begin{equation}\label{est.Ca.Rp}
        \int_{I_{1/2}} |\widetilde{N}(\nabla u)|^p \lesssim \int_{I_1} |\nabla_{\tan} f|^p d\sigma + \int_{T_1} |\nabla u|^p.
    \end{equation}

    Now, let $u_\e$ be a weak solution of \eqref{Rp} in a bounded Lipschitz domain $\Omega$. We will first consider $(R)_2$ estimate. Let $P\in \partial \Omega$. By a blow-up argument and \eqref{est.Ca.Rp}, we have
    \begin{equation}\label{est.SmallScale.local}
        \int_{\Delta_\e(P)} |\widetilde{N}^\e(\nabla u_\e)|^2 d\sigma \lesssim \int_{\Delta_{100\e}(P)} |\nabla_{\tan} f|^2 d\sigma + \frac{1}{\e}\int_{D_{100\e}(P)} |\nabla u_\e|^2 d\sigma,
    \end{equation}
    where $\widetilde{N}^\e(F)(Q)$ is the truncated nontangential maximal function defined on the boundary layer:
    \begin{equation}
    \widetilde{N}^\e(F)(Q) = \sup \bigg\{ \bigg(\fint_{B(X, \delta(X)/2) }|F(Y)|^2 dY \bigg)^{1/2}: X\in \Gamma(Q) \cap \Omega_{10\e} \bigg\}.
    \end{equation}
    Integrating \eqref{est.SmallScale.local} over $P\in \partial \Omega$, we have
    \begin{equation}
        \| \widetilde{N}^\e(\nabla u_\e) \|_{L^2(\partial \Omega)} \lesssim \| \nabla_{\tan} f \|_{L^2(\partial \Omega)} + \bigg( \frac{1}{\e} \int_{\Omega_{100\e}} |\nabla u_\e|^2 \bigg)^{1/2}.
    \end{equation}
    Note that $\widetilde{N}(\nabla u_\e) \le \widetilde{N}_\e(\nabla u_\e) + \widetilde{N}^\e(\nabla u_\e)$ and the large-scale estimate for $\widetilde{N}_\e(\nabla u_\e)$ has been given in Theorem \ref{thm.Rp}. It follows that
    \begin{equation}\label{est.S+L}
    \begin{aligned}
        \| \widetilde{N}(\nabla u_\e) \|_{L^2(\partial \Omega)} & \le \| \widetilde{N}_\e(\nabla u_\e) \|_{L^2(\partial \Omega)} + \| \widetilde{N}^\e(\nabla u_\e) \|_{L^2(\partial \Omega)} \\
        & \lesssim \| \nabla_{\tan} f \|_{L^2(\partial \Omega)} + \bigg( \frac{1}{\e} \int_{\Omega_{100\e}} |\nabla u_\e|^2 \bigg)^{1/2} \\
        & \lesssim \| \nabla_{\tan} f \|_{L^2(\partial \Omega)},
        \end{aligned}
    \end{equation}
    where the last inequality follows from the large-scale Rellich estimate \eqref{eq.LSRellich-D}. Hence, we have proved the full-scale $L^2$ regularity estimate.

    Next, as before, the full-scale $L^p$ regularity estimate for $1<p<2$ follows from an interpolation between $L^2$ estimate and $L^1$ estimate with boundary data in Hardy space. While the estimate for $2<p<2+\delta$ follows from the self-improving property of the reverse H\"{o}lder inequality. These are generally the consequences of the full-scale $(R)_2$ estimates; see \cite[Theorem 5.2, 5.3]{KP93} (also see \cite[Theorem 6.2, 6.3]{KP93} for Neumann problem).

    Finally, in $C^1$ domains, \eqref{est.Ca.Rp} holds for all $2<p<\infty$; see Appendix \ref{A-D}. By blow-up and covering arguments (as well as a small-scale $W^{1,p}$ estimate), we have
    \begin{equation}
    \begin{aligned}
        \| \widetilde{N}^\e(\nabla u_\e) \|_{L^p(\partial \Omega)} & \lesssim \| \nabla_{\tan} f \|_{L^p(\partial \Omega)} + \bigg( \frac{1}{\e} \int_{\Omega_{100\e}} |\nabla u_\e|^p \bigg)^{1/p} \\
        & \lesssim \| \nabla_{\tan} f \|_{L^p(\partial \Omega)} + \bigg( \frac{1}{\e} \int_{\Omega_{100\e}} |M_\e(\nabla u_\e)|^p \bigg)^{1/p},
        \end{aligned}
    \end{equation}
    where the last term can be handled by Theorem \ref{thm.C1Rellich.Rp}. As a result, the desired $L^p$ estimate for $2<p<\infty$ follows similarly as \eqref{est.S+L}.
\end{proof}

\begin{remark}
    It is a well-known fact, see e.g. \cite[Theorem 1.5]{FKP91}, that the full-scale $(D)_p$ estimate is equivalent to the reverse H\"{o}lder inequality of $ k_\e = d\omega^{X_0}_\e/d\sigma $, i.e., for any $\Delta \subset \partial \Omega$,
    \begin{equation}
        \bigg( \fint_{\Delta} k_\e^{p'} d\sigma \bigg)^{1/p'} \lesssim \fint_{\Delta} k_\e d\sigma.
    \end{equation}
    Also, the full-scale $(R)_p$ and $(N)_p$ estimates yield the classical Rellich estimates
    \begin{equation}
        \| \nabla_{\tan} u_\e \|_{L^p(\partial \Omega)} \simeq \Big\| \frac{\partial u_\e}{\partial \nu_\e} \Big\|_{L^p(\partial \Omega)}.
    \end{equation}
    Therefore, these estimates, for the corresponding range of $p$, hold under the assumptions of Corollary \ref{coro.fullscale}.
\end{remark}

\appendix

\section{Basics in periodic homogenization}
\label{Appendix-A}
In this appendix we provide the basics of periodic homogenization for the operator $\cL_\e$ for the reader's convenience, while all of them can be found in the monograph \cite{ShenBook}. 

Let $A = A(Y)$ be a bounded measurable matrix defined on $\R^d$ satisfying \eqref{ellipticity} and \eqref{periodicity}. Due to the peridoicity, we can view $A$ as a function defined on the flat torus (a periodic cell) $\T^d = [0,1)^d$. Let $\cL_1 = -\nabla\cdot (A \nabla)$ be an elliptic operator defined in $\R^d$. The correctors $\{ \chi_j: j=1,2,\cdots, d\}$ are 1-periodic functions satisfying
\begin{equation}
    \left\{
    \begin{aligned}
        & \cL_1 (\chi_j) = -\cL_1(e_j\cdot Y), \quad \text{for } Y \in \R^d, \\
        & \int_{\T^d} \chi_j = 0.
    \end{aligned}
    \right.
\end{equation}
The homogenized (constant) coefficients $\widehat{A} = (\widehat{a}_{ij})$ are defined by
\begin{equation}\label{eq.hatA}
    \widehat{a}_{ij} = \fint_{\T^d} \Big(a_{ij}(Y) + a_{ik}(Y) \frac{\partial}{\partial y_k} \chi_j(Y) \Big) dY,
\end{equation}
where the repeated index $k$ is summed from $1$ to $d$. The homogenized matrix $\widehat{A}$ also satisfies the ellitpicity condition with the same constant. The homogenized operator is given by $\cL_0 = - \nabla\cdot (\widehat{A} \nabla)$. If $A$ is symmetric, then $\widehat{A}$ is also symmetric.

The flux correctors are defined as follows. Let $f = (f_{ij})$ be the matrix-valued 1-periodic solution of
\begin{equation}\label{def.f}
    \Delta f(Y) = A(Y) + A(Y) \nabla \chi(Y) - \widehat{A}.
\end{equation}
The $H^2$ regularity for the Laplace operator implies that $f \in H^2(\T^d)$. The flux correctors are defined by $\phi = \nabla \times f$, or in the component form
\begin{equation}
    \phi_{kij} = \frac{\partial}{\partial y_k} f_{ij} - \frac{\partial}{\partial y_i} f_{kj}.
\end{equation}
The key properties of $\phi$ are
\begin{equation}\label{eq.fluxProperty}
    \phi_{kij} = -\phi_{ikj} \qquad \text{and} \qquad \frac{\partial }{\partial y_k} \phi_{kij} = 0.
\end{equation}
The latter is due to \eqref{def.f} and the observation that $f$ is divergence free, i.e., $\partial f_{kj}/\partial y_k = 0$. The properties in \eqref{eq.fluxProperty} are useful for the calculation of \eqref{def.Fe}. In this paper, we only need to use the fact that $\chi \in H^1(\T^d)$ and $\phi \in L^2(\T^d)$ in the proof of Lemma \ref{lem.weLp}.

\section{Basic properties of average operator \\and smoothing operator}

We introduce a convenient notation for the $\e$-neighborhood of an open set $\mathcal{O} \subset \R^d$, $B_\e(\mathcal{O}): = \{ X\in \R^d: \dist(X,\mathcal{O}) < \e \}$.
In this paper, we use different versions of average operators, including $M_\e, M_\e^{\star}$ and $M_\e^\partial$. We will only prove several properties of $M_\e$, while similar properties also hold for other average operators.

\begin{proposition}\label{prop.A1}
    Suppose that $\mathcal{O}$ is an open subset of $\R^d$ such that for any $X \in \mathcal{O}, |B_\e(X) \cap \mathcal{O}| \gtrsim |B_\e(X)|$. Then for any locally $L^2$ function $f$, and any $1\le p\le 2$
    \begin{equation}\label{est.f<Mef}
        \| f \|_{L^p(\mathcal{O})} \lesssim \| M_\e( f) \|_{L^p(\mathcal{O})}
    \end{equation}
    Moreover, for any $q \ge 2$,
    \begin{equation}\label{est.Mef<f}
        \| M_\e( f) \|_{L^q(\mathcal{O})} \lesssim \| f \|_{L^q(B_\e(\mathcal{O}))}.
    \end{equation}
\end{proposition}
\begin{proof}
    We first prove \eqref{est.f<Mef}. By the H\"{o}lder inequality and Fubini's Theorem, for $1\le p\le 2$,
    \begin{equation}
    \begin{aligned}
        \int_{\mathcal{O}} M_\e(f)^p dX & \ge \int_{\mathcal{O}} \fint_{B_\e(X)} |f(Y)|^p dY dX \\
        & \ge |B_\e|^{-1} \int_{B_\e(\mathcal{O})} \int_{\mathcal{O}} \mathbbm{1}_{\{|X-Y|<\e\} } |f(Y)|^p dX dY \\
        & \ge |B_\e|^{-1} \int_{\mathcal{O}} |\mathcal{O} \cap B_\e(Y)| |f(Y)|^p dY \\
        & \gtrsim \int_{\mathcal{O}} |f(Y)|^p dY,
    \end{aligned}
    \end{equation}
    where we have used the assumption $|B_\e(Y) \cap \mathcal{O}| \gtrsim |B_\e(Y)|$ for $Y\in \mathcal{O}$.

    To prove \eqref{est.Mef<f}, we use the H\"{o}lder inequality and Fubini's Theorem again,
    \begin{equation}
    \begin{aligned}
        \int_{\mathcal{O}} |M_\e(f)|^q dX & \le \int_{\mathcal{O}} \fint_{B_\e(X)} |f(Y)|^q dY dX \\
        & \le |B_\e|^{-1} \int_{B_\e(\mathcal{O})} \int_{\mathcal{O}} \mathbbm{1}_{\{|X-Y| < \e \} } |f(Y)|^q dX dY  \\
        & \le \int_{B_\e(\mathcal{O})} |f(Y)|^q dX dY.
        \end{aligned}
    \end{equation}
    This ends the proof.
\end{proof}

\begin{proposition}\label{prop.Mef-Bs}
    Let $\e > 0$. Then for any $s \lesssim \e$,
    \begin{equation}\label{est.Mef-Bs-1}
        M_\e(f)(X) \lesssim \fint_{B_s(X)} M_\e(f),
    \end{equation}
    and
    \begin{equation}\label{est.Mef-Bs-2}
        M_s(M_\e(f))(X) \lesssim M_{\e+s}(f)(X) \lesssim M_{2s}(M_\e(f))(X).
    \end{equation}
\end{proposition}
\begin{proof}
    Let $X = 0$. Without loss of generality, assume $s < \e/10$. For any $\theta \in (0,1/2)$, there exists a cone $\mathcal{C}_\theta = \{ X \in \R^d: 1-\theta < \frac{X}{|X|}\cdot v \le 1 \} $ for some $v \in \mathbb{S}^{d-1}$ such that
    \begin{equation}\label{est.Mef0}
        M_\e(f)(0) = \bigg( \fint_{B_\e(0)} |f(Y)|^2 dY \bigg)^{1/2} \simeq \bigg( |B_\e|^{-1}\int_{B_\e(0) \cap \mathcal{C}_\theta} |f(Y)|^2 dY \bigg)^{1/2}.
    \end{equation}
    Now, observe that if $\theta$ is small enough and $s<\e/10$, then for any $X \in B_s(0) \cap \mathcal{C}_\theta$, $B_\e(0) \cap \mathcal{C}_\theta \subset B_\e(X)$. This and \eqref{est.Mef0} imply that for each $X \in B_\e(0) \cap \mathcal{C}_\theta$,
    \begin{equation}
    \begin{aligned}
        M_\e(f)(X) & = \bigg( \fint_{B_\e(X)} |f(Y)|^2 dY \bigg)^{1/2} \\
        & \ge \bigg( |B_\e|^{-1} \int_{B_\e(0) \cap \mathcal{C}_\theta} |f(Y)|^2 dY \bigg)^{1/2} \simeq M_\e(f)(0).
    \end{aligned}
    \end{equation}
    It follows that
    \begin{equation}
        M^\e(f)(0) \lesssim \fint_{B_\e(0) \cap \mathcal{C}_\theta} M_\e(f)(X) dX \lesssim \fint_{B_\e(0) } M_\e(f)(X) dX.
    \end{equation}
    This proves \eqref{est.Mef-Bs-1}.

    The first inequality in \eqref{est.Mef-Bs-2} follows from the simple fact that for each $Y \in B_s(X)$, $M_\e(f)(Y) \lesssim M_{\e+s}(f)(X)$. Taking $L^2$ average of $M_\e(f)(Y)$ over $B_s(X)$, we get the desired estimate. For the second inequality in \eqref{est.Mef-Bs-2}, we apply Fubini's Theorem as follows (assume $X = 0$ for simplicity),
    \begin{equation}
        \begin{aligned}
            M_{2s}(M_\e(f))(0) & = \frac{1}{|B_{2s}| |B_\e|} \int_{B_{2s}(0)} \int_{B_\e(Y)} |f(Z)|^2 dZ dY \\
            & = \frac{1}{|B_{2s}| |B_\e|} \int_{B_{\e+2s}(0)} |f(Y)|^2 \int_{B_{2s}(0)} \mathbbm{1}_{\{ |Y-Z|<\e \} } dZ dY.
        \end{aligned}
    \end{equation}
    Now, the key observation is that for $Y \in B_{\e+s}(0)$, we have
    \begin{equation}
        \int_{B_{2s}(0)} \mathbbm{1}_{\{ |Y-Z|<\e \} } dZ= | B_{2s}(0) \cap B_\e(Y) | \gtrsim |B_s(0)|.
    \end{equation}
    It follows that
    \begin{equation}
        M_{2s}(M_\e(f))(0) \gtrsim \frac{1}{|B_\e|} \int_{B_{\e+s}(0)} |f(Y)|^2 dY \simeq M_{\e+s}(f)(0),
    \end{equation}
    as desired.
\end{proof}

\begin{proposition}\label{prop.Me=MKe}
    Let $\mathcal{O}$ be an open subset of $\R^d$. Then given $K \ge 1$, 
    \begin{equation}\label{est.Me=MKe}
        \| M_\e(f) \|_{L^p(\mathcal{O})} \lesssim \| M_{K\e}(f) \|_{L^p(\mathcal{O})} \lesssim \| M_{\e}(f) \|_{L^p(B_{K\e}(\mathcal{O}))},
    \end{equation}
    where the implicit constant depends on $d$ and $K$.
\end{proposition}
\begin{proof}
    The first inequality is due to the observation
    \begin{equation}
        M_\e(f)(X) \lesssim M_{K\e}(f)(X) \quad \text{for any } X\in \mathcal{O}.
    \end{equation}
    To see the second, we use the following covering property: there exists $Z_1,Z_2,\cdots, Z_m \in B_{K\e}(0)$ with $m$ depending only on $d$ and $K$ such that
    \begin{equation}
        B_{K\e}(0) = \bigcup_{i=1}^m B_\e(Z_i).
    \end{equation}
    This implies that
    \begin{equation}
        M_{K\e}(f)(X) \lesssim \sum_{i=1}^m M_\e(f)(X+Z_i).
    \end{equation}
    Taking $L^p$ norm on both sides in $X$ over $\mathcal{O}$, we obtain the second inequality in \eqref{est.Me=MKe}.
\end{proof}

Define the large-scale Hardy-Littlewood maximal function by
\begin{equation}\label{def.HardyLittlewood.Me}
    \M^\e(f)(X) = \sup\bigg\{ \fint_{B_r} |f|: X\in B_r, r> \e \bigg\}.
\end{equation}
\begin{proposition}\label{prop.LS.HLM}
    Let $f \in L^p(\R^d)$ for some $1<p\le \infty$. Then
    \begin{equation}\label{est.Me.Lpbound}
        \| \M^\e(f) \|_{L^p(\R^d)} \lesssim \| M_\e(f) \|_{L^p(\R^d)}.
    \end{equation}
\end{proposition}
\begin{proof}
    Suppose $r>\e$. Then for any $X\in B_r$, then $|B_r \cap B_\e(X)| \simeq |B_\e(X)|$. By Proposition \ref{prop.A1} with $p = 1$,
    \begin{equation}
        \fint_{B_r} |f| \lesssim \fint_{B_r} M_\e(f).
    \end{equation}
    Thus, by the definition \eqref{def.HardyLittlewood.Me},  we have $\M^\e(f)(X) \lesssim \M (M_\e f)(X)$, where $\M$ is the classical Hardy-Littlewood maximal function. Thus, it follows from the $L^p$ boundedness of $\M$ that \eqref{est.Me.Lpbound} holds for any $1<p\le \infty$.
\end{proof}

Recall that the smoothing operator is defined in \eqref{def.Ke}. We have the following properties.

\begin{proposition}\label{prop.f-Kef}
    Let $\mathcal{O}$ be an open set and  $f$ be  a $W^{1,p}$ function defined in $B_{\e}(\mathcal{O})$. Then for $1\le p<\infty$,
    \begin{equation}
        \| K_\e(f) - f \|_{L^p(\mathcal{O})} \le \e \| \nabla f \|_{L^p(B_\e(\mathcal{O}))}.
    \end{equation}
\end{proposition}
\begin{proof}
    Write, for $X\in \mathcal{O}$,
    \begin{equation}
        K_\e f(X)  - f(X) = \int_{B_\e(0)} \psi_\e(Y) \int_0^1 \nabla f(X - sY) \cdot Y ds dY.
    \end{equation}
    By the Minkowski inequality,
    \begin{equation}
    \begin{aligned}
        \| K_\e f(\cdot)  - f(\cdot) \|_{L^p(\mathcal{O})} & \le \e \int_{B_\e(0)} \psi_\e(Y) \int_0^1 \| \nabla f(\cdot - sY) \|_{L^p(\mathcal{O})}  ds dY \\
        & \le \e \| \nabla f \|_{L^p(B_\e(\mathcal{O}))}.
    \end{aligned}
    \end{equation}
    The proof is complete.
\end{proof}

\begin{proposition}[{\cite[Proposition 3.1.5]{ShenBook}}]
    Let $\mathcal{O}$ be an open set and  $f$ be a function defined in $B_{\e}(\mathcal{O})$. Let $g \in L^p(\T^d)$ be a 1-periodic function. Then for any $1 \le p < \infty$,
    \begin{equation}
        \| g(X/\e) K_\e (f)  \|_{L^p(\mathcal{O})} \lesssim \| g \|_{L^p(\T^d)} \| f \|_{L^p(B_{\e}(\mathcal{O}))}.
    \end{equation}
\end{proposition}

\begin{proposition}\label{prop.Me(gKef)}
    Let $\mathcal{O}$ be an open set and  $f$ be a function defined in $B_{2\e}(\mathcal{O}): = \{ X\in \R^d: \dist(X,\mathcal{O}) < \e \}$. Let $g \in L^2(\T^d)$ be a 1-periodic function. Then for any $2\le p < \infty$,
    \begin{equation}
        \| M_\e(g(X/\e) K_\e (f) ) \|_{L^p(\mathcal{O})} \lesssim \| g \|_{L^2(\T^d)} \| f \|_{L^p(B_{2\e}(\mathcal{O}))}.
    \end{equation}
\end{proposition}
\begin{proof}
    Define $T_\e f = M_\e(g(X/\e) K_\e (f) )$. Clearly, $T_\e$ is a sublinear operator. We prove the theorem by considering $p = 2$ and $p = \infty$, and then apply the real interpolation theorem.
    For $p  = 2$, we have
    \begin{equation}
        \| T_\e f \|_{L^p(\mathcal{O})} \lesssim \| g(X/\e) K_\e (f) \|_{L^2(B_\e(\mathcal{O}))} \lesssim \| g \|_{L^2(\T^d)} \| f \|_{L^2(B_{2\e}(\mathcal{O}))}.
    \end{equation}
    For $p = \infty$, 
    \begin{equation}
        \| T_\e f \|_{L^\infty(\mathcal{O})} \lesssim \| f \|_{L^\infty(B_{2\e}(\mathcal{O}))} \| M_\e(g(X/\e)) \|_{L^\infty(B_{\e}(\mathcal{O}))} \lesssim \| g \|_{L^2(\T^d)} \| f \|_{L^\infty(B_{2\e}(\mathcal{O}))}.
    \end{equation}
    The desired estimate follows from Marcinkiewicz interpolation Theorem.
\end{proof}

\section{The real-variable argument}

The real-variable argument (originating from \cite{CP98}) is built upon the following theorem taken from \cite[Theorem 4.2.3]{ShenBook}.
\begin{theorem}\label{thm.full real-variable}
Let $B_0$ be a ball in $\R^d$ and $F \in L^2(4B_0)$. Let $q>2$ and $f \in L^p(4B_0)$ for some $2 < p < q$. Suppose that for each ball $B \subset 2B_0$ with $|B| \le c_1 |B_0|$, there exist two measurable functions $F_B$ and $R_B$ on $2B$ such that $|F| \le |F_B| + |R_B|$ on $2B$, and
\begin{equation}\label{cond.real-variable}
\left\{
\begin{aligned}
& \bigg( \fint_{2B} |F_B|^2 \bigg)^{1/2} 
\leq N_1 \sup_{4B_0 \supset B' \supset B} \bigg( \fint_{B'} |f|^2  \bigg)^{1/2},\\
& \bigg( \fint_{2B} |R_B|^q \bigg)^{1/q} 
\leq N_2 \bigg\{ 
\bigg( \fint_{\beta B} |F|^2  \bigg)^{1/2} 
+ \sup_{4B_0 \supset B' \supset B} \bigg(  \fint_{B'} |f|^2 \bigg)^{1/2} 
\bigg\},
\end{aligned}
\right.
\end{equation}
where $N_1, N_2 > 0, 0<c_1<1$ and $\beta > 2$. Then $F \in L^p(B_0)$ and
\begin{equation}
    \bigg( \fint_{B_0} |F|^p \bigg)^{1/p} \le C \bigg\{ \bigg( \fint_{4B_0} |F|^2 \bigg)^{1/2} + \bigg( \fint_{4B_0} |f|^p \bigg)^{1/p} \bigg\},
\end{equation}
where $C$ depends only on $d$, $p$, $q$, $N_1$, $N_2, c_1$, and $\beta$.
\end{theorem}

In this paper, we actually need a large-scale version of the real-variable argument, in which the assumption \eqref{cond.real-variable} holds only above a certain scale. In this case, we only have large-scale estimates above this particular scale. The following theorem or its similar versions also appears in \cite[Remark 4.2]{S23} and \cite[Theorem 2.6]{Z21}.

\begin{theorem}\label{thm.real-variable>t}
    Let $B_0$ be a ball in $\R^d$ and $F \in L^2(4B_0)$. Let $q>2$ and $f \in L^2(4B_0)$ for some $2 < p < q$. Suppose that there exists $t > 0$ such that for each ball $B \subset 2B_0$ with $|B| \le c_1 |B_0|$ and $\diam(B) > 2t$, there exist two measurable functions $F_B$ and $R_B$ on $2B$ such that $|F| \le |F_B| + |R_B|$ on $2B$, and \eqref{cond.real-variable} holds. Then $M_t(F) \in L^p(B_0)$ and
    \begin{equation}\label{est.real-variable-t}
    \bigg( \fint_{B_0} |M_t(F)|^p \bigg)^{1/p} \le C \bigg\{ \bigg( \fint_{4B_0} |F|^2 \bigg)^{1/2} + \bigg( \fint_{4B_0} |M_t(f)|^p \bigg)^{1/p} \bigg\},
    \end{equation}
    where $C$ depends only on $d$, $p$, $q$, $N_1$, $N_2, c_1$, and $\beta$.
\end{theorem}
\begin{proof}
    The proof is simple, using Theorem \ref{thm.full real-variable}. Actually, let $\overline{F} = M_t(F)$. Then $\overline{F}$ satisfies $\overline{F} \le M_t(F_B) + M_t(R_B)$ and the assumption \eqref{cond.real-variable} with $F_B, R_B$ and $f$ replaced by $M_t(F_B), M_t(R_B)$ and $M_t(f)$ for any $B$ with arbitrarily small radius, namely
    \begin{equation}\label{cond.real-variable-t}
\left\{
\begin{aligned}
& \bigg( \fint_{2B} |M_t(F_B)|^2 \bigg)^{1/2} 
\lesssim \sup_{4B_0 \supset B' \supset B} \bigg( \fint_{B'} |M_t(f)|^2  \bigg)^{1/2},\\
& \bigg( \fint_{2B} |M_t(R_B)|^q \bigg)^{1/q} 
\lesssim \bigg\{ 
\bigg( \fint_{\beta B} |\overline{F}|^2  \bigg)^{1/2} 
+ \sup_{4B_0 \supset B' \supset B} \bigg(  \fint_{B'} |M_t(f)|^2 \bigg)^{1/2} 
\bigg\}.
\end{aligned}
\right.
\end{equation}
    In fact, if for $B = B_s$ and $s>t$, then \eqref{cond.real-variable} directly implies \eqref{cond.real-variable-t}, due to the comparison of the $L^2$ norm of $F$ and $M_t(F)$ in Proposition \ref{prop.A1}. For $s<t$, thanks to the first inequality in \eqref{est.Mef-Bs-2}, we have
    \begin{equation}
        \bigg( \fint_{B_{2s}} |M_t(F_B)|^2 \bigg)^{1/2} \lesssim \bigg( \fint_{B_{t+2s}} |F_B|^2 \bigg)^{1/2}.
    \end{equation}
    Now $\diam(B_{t+2s}) = 2(t+2s) > 2t$ and we can apply the first condition in \eqref{cond.real-variable} at this scale and the second inequality in \eqref{est.Mef-Bs-2} to get
    \begin{equation}
        \bigg( \fint_{B_{2s}} |M_t(F_B)|^2 \bigg)^{1/2} \lesssim \sup_{4B_0 \supset B' \supset B_s} \bigg( \fint_{B'} |M_t(f)|^2  \bigg)^{1/2}.
    \end{equation}
    This is a large-scale version of the first condition of \eqref{cond.real-variable}.
    Similarly, a large-scale version of the second condition in \eqref{cond.real-variable} is valid for $M_t(R_B)$. Thus, applying Theorem \ref{thm.full real-variable} to $\overline{F} = M_t(F)$, we get \eqref{est.real-variable-t}. Note that we changed $M_t(F)$ back to $F$ on the right-hand side of \eqref{est.real-variable-t} by \eqref{est.Mef<f}.
\end{proof}

\section{Regular elliptic operators in $C^1$ domains}\label{A-D}

Consider the elliptic equation 
$\mathcal{L}_0(u)=0$ in a bounded domain $\Omega$, where $\mathcal{L}_0$ is a second-order elliptic operator with constant and symmetric  coefficients.
If $\Omega$ is $C^1$ and $1<p< \infty$, it was proved  in \cite{Fabes1978} by the method of layer potentials that the Dirichlet problem,
\begin{equation}\label{D-1}
\left\{
\aligned 
& \mathcal{L}_0 (u) =0 \ \ \text{ in } \Omega,\\
& u=f \in L^p(\partial\Omega) \   \text{ on } \partial\Omega \text{ and } N(u) \in L^p(\partial\Omega)
\endaligned
\right.
\end{equation}
is solvable, and that  the solution satisfies the estimate, 
\begin{equation}\label{D-2}
\| N(u) \|_{L^p(\partial\Omega)}
\lesssim \| f \|_{L^p(\partial\Omega)}.
\end{equation}
Moreover, if $f\in W^{1, p}(\partial\Omega)$, then 
\begin{equation}\label{D-3}
\| N(\nabla u)\|_{L^p(\partial\Omega)}
\lesssim \| f\|_{W^{1, p}(\partial\Omega)}.
\end{equation}
Furthermore, if $\int_{\partial \Omega} g d\sigma =0$, the Neumann problem,
\begin{equation}\label{D-4}
\left\{
\aligned 
& \mathcal{L}_0 (u) =0 \ \ \text{ in } \Omega,\\
& \frac{\partial u}{\partial \nu_0}=g \in L^p(\partial\Omega) \   \text{ on } \partial\Omega \text{ and } N(\nabla u) \in L^p(\partial\Omega)
\endaligned
\right.
\end{equation}
is solvable for $1< p< \infty$, and  
the solution satisfies 
\begin{equation}\label{D-5}
    \| N(\nabla u) \|_{L^p(\partial\Omega)}
\lesssim \| g \|_{L^p(\partial\Omega)}.
\end{equation}
We point out that the results above also follow from some general results in 
\cite{DPP07, Dindos2017} for elliptic operators with coefficients that satisfy the so-called ``small Carleson norm'' condition, particularly including $C^\alpha$ H\"{o}lder continuous coefficients. Moreover, \eqref{D-3} yields the local estimate \eqref{est.Ca.Rp} for any $p\in (1,\infty)$. Similar estimate holds for the local Neumann problem in $C^1$ domains.


\bibliographystyle{abbrv}
\bibliography{ref}
\end{document}